\documentclass[a4paper,reqno,10pt]{amsart}
\usepackage{amsfonts}
\usepackage{amsmath}
\usepackage{amssymb}
\usepackage{mathtools}
\usepackage{commath}
\usepackage{cite}
\usepackage{bbm}
\usepackage{amscd}
\usepackage[hidelinks]{hyperref}
\usepackage[left=2cm,right=2cm,bottom=3cm,top=3cm]{geometry}
\usepackage[sort&compress,numbers]{natbib}

\newtheorem{theorem}{Theorem}[section]
\newtheorem{proposition}[theorem]{Proposition}
\theoremstyle{definition}
\newtheorem{lemma}[theorem]{Lemma}
\newtheorem{definition}[theorem]{Definition}
\newtheorem{corollary}[theorem]{Corollary}

\newtheorem{problem}[theorem]{Problem}

\newtheorem*{claim*}{Claim}
   
\let\norm\undefined 
\DeclarePairedDelimiter\norm{\lVert}{\rVert}
\AtBeginDocument{   \def\MR#1{}
}

\begin{document}
\title[An invitation to model theory and C*-algebras]{An invitation to model
theory and C*-algebras}
\author[Martino Lupini]{Martino Lupini}
\address{Martino Lupini, Mathematics Department, California Institute of
Technology, 1200 East California Boulevard, Mail Code 253-37, Pasadena, CA
91125}
\email{lupini@caltech.edu}
\urladdr{http:n//www.lupini.org/}
\thanks{This survey originates from the lecture notes for the masterclass on ``Applications of Model Theory to Operator Algebras'' given by the author at the University of Houston from July 31st to August 4th, 2017, and
supported by the NSF grant DMS-1700316. We are grateful to the organizers,
Mehrdad Kalantar and Mark Tomforde, as well as the Department of Mathematics
at the University of Houston, for their hospitality in such an occasion. The
author has also been partially supported by the NSF Grant DMS-1600186.}
\dedicatory{}
\maketitle

\begin{abstract}
We present an introductory survey to first order logic for metric structures
and its applications to C*-algebras.
\end{abstract}

\section{Introduction}


This survey is designed as an introduction to the study of C*-algebras from
the perspective of model theory for metric structures. The intended
readership consists of anyone interested in learning about this subject, and
naturally includes both logicians and operator algebraists. Considering
this, we will not assume in these notes any previous knowledge of model
theory, nor any in-depth knowledge of functional analysis, beyond a standard
graduate-level course. A familiarity with C*-algebra theory and the
classification programme \cite{elliott_regularity_2008} might be useful as a
source of motivation and examples. Several facts from C*-algebra theory will
be used, and detailed references provided. Most of the references will be to
the comprehensive monographs \cite%
{blackadar_operator_2006,rordam_classification_2002}.

Logic for metric structures is a generalization of classical (or \emph{%
discrete}) logic, suitable for applications to \emph{metric objects} such as
C*-algebras. The monograph \cite{ben_yaacov_model_2008} presents a quick but
complete introduction to this subject, and explains how the fundamental
results from classical model theory can be recast in the metric setting. The
model-theoretic study of C*-algebras has been initiated in \cite%
{farah_model_2013,farah_model_2014,farah_model_2014-1} where, in particular,
it shown how C*-algebras fit into the framework of model theory for metric
structures. The motivations behind this study are manifold. With no pretense
of exhaustiveness, we attempt to illustrate some of them here.

The most apparent contribution of first-order logic is to provide a \emph{%
syntactic }counterpart to the semantic construction of ultraproducts and
ultrapowers: the notion of formulas. Formulas allow one to express the
fundamental properties of ultrapowers of C*-algebras (saturation) and
diagonal embeddings into the ultrapower (elementarity). These general
principles underpin most of the applications of the ultraproduct
construction in C*-algebra theory, as they have appeared in various places
in the literature under various names---Kirchberg's $\varepsilon $-test,
reindexing arguments, etc. Isolating such general principles provides a
valuable service of clarification and uniformization in the development of
C*-algebra theory. In particular, this allows one to distinguish between, on
one hand, what is just an instance of \textquotedblleft general
nonsense\textquotedblright\ and, on the other, what is a salient point where
C*-algebras theory is crucially used.

This abstract model-theoretic point of view also makes it easier to
recognize analogies between different contexts. Furthermore, it provides a
language to formalize such analogies as precise mathematical statements,
rather than just intuitive ideas. For instance, this paradigm can be applied
to some aspects of the \emph{equivariant} theory of C*-algebras, which
studies C*-algebras endowed with a group action (C*-dynamical systems). At
least when the acting group is compact or discrete, C*-dynamical systems fit
in the setting of first-order logic \cite{gardella_rokhlin_2017}. Adopting
this perspective, one can naturally and effortlessly transfer ideas and
arguments from the nonequivariant theory, as long as these are presented in
terms of model-theoretic notions and principles. An instance of this
phenomenon is the general theory of strongly self-absorbing C*-algebras,
which admits a natural model-theoretic treatment; see Section \ref%
{Section:strongly}. The equivariant analog of such a notion has been
recently introduced and studied by\ Szab\'{o} in a series of papers \cite%
{szabo_strongly_2015,szabo_strongly_2016,szabo_strongly_2017}, where the
theory is developed in close parallel to the nonequivariant setting.

Beyond the motivations above, model theory provides the right tools for the
study of ultrapowers and central sequence algebras \emph{per se}. Questions
on the number of nonisomorphic ultrapowers arise naturally in operator
algebra theory, and can be traced back to McDuff's study of central
sequences in the context of factors \cite{mcduff_central_1970}. Many of such
questions have been answered in \cite%
{farah_model_2013,farah_model_2014,farah_model_2014-1} through the
application of rather general model-theoretic principles. We will explore
some of these problems in Section \ref{Section:continuum}, focusing on the
situation under the Continuum Hypothesis. Ultrapowers and central sequence
algebras are in fact just some instances of naturally occurring
\textquotedblleft massive C*-algebras\textquotedblright . Other examples are
Calkin algebras and, more generally, corona algebras. Deep questions about
such algebras have been recently addressed in \cite%
{phillips_calkin_2007,farah_all_2011,farah_all_2011-1,eagle_saturation_2015,farah_calkin_2013,coskey_automorphisms_2014,farah_homeomorphisms_2012,vignati_nontrivial_2017,farah_rigidity_2016,farah_calkin_2016,farah_countable_2013}%
. Methods from model theory, set theory, and forcing are key components of
this line of research.

A further thread of applications of model theory comes from the use of
techniques for constructing \textquotedblleft generic
objects\textquotedblright , and the potential of using these techniques to
construct interesting new examples of C*-algebras. This is particularly
relevant considering that one of the most important open problems in
C*-algebra theory (the UCT problem) depends on the existence of other
methods of constructing nuclear C*-algebras other than the standard
constructions of C*-algebra theory; see Section \ref{Section:nuclear}. Using
the techniques of model-theoretic forcing and building models by games, many
deep open problems in operator algebra theory have been reformulated in
terms of model-theoretic notions, in the hope that these might be more
amenable to a direct attack. Examples of these problems include the famous
Connes Embedding Problem and some its C*-algebraic counterparts (the
Kirchberg Embedding Problem, the MF problem); see \cite%
{farah_model_2017,goldbring_enforceable_2017,goldbring_robinson_2017,goldbring_kirchbergs_2015}%
.

While the list above does not exhaust the applications of model theory to
operator algebras, we hope it will sufficiently motivate the choice of
topics of this survey. After a general introduction to the logic for metric
structures (Section \ref{Section:one}), we will explain how many classes of
C*-algebras can be described through formulas (Section \ref{Section:axiom}).
Ultraproducts and their model-theoretic properties are considered in Section %
\ref{Section:ultra}, and the question on the number of ultraproducts in
Section \ref{Section:continuum}. The important class of strongly
self-absorbing C*-algebra and its model-theoretic treatment is the subject
of Section \ref{Section:strongly}. We conclude in Section \ref%
{Section:nuclear} with a quick introduction to the classification programme
of nuclear C*-algebras, a description of the model-theoretic content of
nuclearity and other regularity properties, and an outlook on the
applications of model-theoretic forcing to produce interesting examples of
nuclear C*-algebras.

No result presented in this survey is original, although the presentation of
some of the material is new. The results concerning the general theory of
first-order logic for metric structures can be found in \cite%
{ben_yaacov_model_2008,farah_model_2014}. Axiomatizability of the classes of
C*-algebras presented here and many more is contained in \cite%
{farah_model_2017,carlson_omitting_2014}, as well as the model-theoretic
description of nuclearity and other regularity properties. The theorem on
the number of nonisomorphic ultrapowers of C*-algebras is one of the main
results of \cite{farah_model_2013,farah_model_2014}, together with the
corresponding fact for II$_{1}$ factors. The model-theoretic treatment of
strongly self-absorbing C*-algebras is the subject of \cite%
{farah_relative_2017}. The model-theoretic proof of the characterization of $%
D$-absorption presented here is in some respects original, although heavily
inspired by the proofs from the literature \cite[Theorem 2.2]%
{toms_strongly_2007}, especially those in the equivariant setting from \cite%
{szabo_strongly_2015}. Finally, the technique of model-theoretic forcing in
the metric setting has been first considered in \cite{ben_yaacov_model_2009}
and then further developed in \cite%
{farah_omitting_2014,goldbring_enforceable_2017,farah_model_2017}.

\subsubsection*{Acknowledgments}

We are grateful to Piotr Koszmider, Andrew Ostergaard, and Richard Timoney
for their comments on a preliminary version of the present survey.

\section{First-order logic for metric structures\label{Section:one}}

\subsection{Languages\label{Subsection:single-sorted}}

Model theory focuses on the study of \emph{classes }of objects, rather than
single objects on their own. The important notion of \emph{language }or 
\emph{signature }has the purpose of formalizing the assertion that a certain
class is made of objects \textquotedblleft of the same
kind\textquotedblright . It also allows one to make explicit which
operations on the given objects are being considered. Formally, a \emph{%
language} (or signature) usually denoted by $L$,\emph{\ }is a collection of 
\emph{symbols}. These symbols are of two kinds: \emph{function symbols} and 
\emph{relation symbols}. Each function symbol $f$ in the language $L$ has
attached a natural number $n_{f}$, called its \emph{arity}, and a function $%
\varpi ^{f}:[0,+\infty )^{n_{f}}\rightarrow \lbrack 0,+\infty )$ continuous
at $0$ and vanishing at $0$, called its \emph{continuity modulus}.
Similarly, each \emph{relation symbol }$R$ in the language $L$ has attached
a natural number $n_{R}$, called its arity, a function $\varpi
^{R}:[0,+\infty )^{n_{R}}\rightarrow \lbrack 0,+\infty )$ continuous at $0$
and vanishing at $0$, called its continuity modulus, and a compact internal $%
J_{R}$ of $\mathbb{R}$, called its \emph{bound}. The language $L$ includes a
distinguished relation symbol, called the \emph{metric symbol}. This is
denoted by $d$, and it has arity $2$, and continuity modulus $\varpi
^{d}\left( t_{0},t_{1}\right) =t_{0}+t_{1}$. As customary we will call \emph{%
binary }a symbol of arity $2$, and \emph{unary }a symbol of arity $1$. The
arity $n_{f}$ of a function symbol $f$ is allowed to be $0$, in which case
one says that $f$ is a \emph{constant symbol}.

Given a language $L$, one can then define the notion of $L$-structure.
Briefly, an $L$-structure is a set endowed with functions and relations
corresponding to the symbols in $L$. Precisely, an $L$-structure is a
complete metric space $\left( M,d^{M}\right) $ together with assignments $%
f\mapsto f^{M}$ and $R\mapsto R^{M}$, which assign to each function symbol $%
f $ in $L$ its \emph{interpretation }$f^{M}$ in $M$, and to each relation
symbol $R$ in $L$ its \emph{interpretation} $R^{M}$ in $M$. These
interpretations are required to satisfy the following properties. If $f$ is
a function symbol of arity $n_{f}$ and continuity modulus $\varpi ^{f}$,
then $f^{M}$ is a function $f^{M}:M^{n_{f}}\rightarrow M$ such that, for
every $\bar{a},\bar{b}\in M^{n_{f}}$, $d^{M}\left( f\left( \bar{a}\right)
,f\left( \bar{b}\right) \right) \leq \varpi ^{f}\left( d^{M}\left( \bar{a},%
\bar{b}\right) \right) $. When the arity $n_{f}$ of $f$ is zero, i.e.\ when $%
f$ is a constant symbol, by convention the set $M^{n_{f}}$ consists of a
single point, and the function $f^{M}:M^{n_{f}}\rightarrow M$ can be simply
seen as a \emph{distinguished element }of $M$. If $R$ is a relation symbol
of arity $n_{R}$, continuity modulus $\varpi ^{R}$, and bound $J^{R}$, then $%
R^{M}$ is a function $R^{M}:M^{n_{R}}\rightarrow J_{R}\subset \mathbb{R}$
such that, for every $\bar{a},\bar{b}\in M^{n_{f}}$, $\left\vert R\left( 
\bar{a}\right) -R\left( \bar{b}\right) \right\vert \leq \varpi ^{R}\left(
d^{M}\left( \bar{a},\bar{b}\right) \right) $. Furthermore, the
interpretation of the \emph{metric symbol }$d$ of $L$ is required to be
equal to the metric $d^{M}$ of $M$ (consistently with the notation above).

To summarize, an $L$-structure is a space endowed with some extra operations
(functions and relations) as prescribed by the language. Furthermore, the
language contains names (or symbols) for such operations. This allows one to
uniformly and unambiguously refer to such operations when considering the
class of all $L$-structures.

\subsection{Metric groups\label{Subsection:metric-groups}}

At this point, examples are in order. We consider for now examples from
metric geometry and group theory. The first natural example is the language $%
L$ containing no function symbol, and whose unique relation symbol is the
distinguished symbol for the metric $d$. One also has to specify in $L$ a
bound $J_{d}$ for $d$, which we can choose to be $\left[ 0,1\right] $. The
continuity modulus for $d$ can be defined to be $\varpi ^{d}\left(
t_{0},t_{1}\right) =t_{0}+t_{1}$. This completely defines a language $L$ in
the sense above.\ It is clear that, for such a language, an $L$-structure is
just a complete metric space $\left( M,d^{M}\right) $ whose metric attains
values in $\left[ 0,1\right] $. Thus the class of $L$-structure consists of
the class of complete metric spaces of diameter at most $1$.

A slightly more sophisticated example can be obtained by adding to this
language some function symbols, to describe some additional algebraic
structure which might be present on a complete metric space. The first
natural example is the case of a language $L$ containing a single binary
function symbol, to describe a binary operation. Since we want to think of
it as a binary operation, we denote such a binary function symbol by $\cdot $
, for which we use the usual infix notation. We also need to prescribe a
continuity modulus for such a binary function symbol, which we define to be $%
\left( t_{0},t_{1}\right) \mapsto t_{0}+t_{1}$. As above, we assume that the
unique relation symbol in $L$ is the metric symbol. In this case, an $L$%
-structure is a complete metric space $\left( M,d^{M}\right) $ with diameter 
$1$ endowed with a binary operation $\cdot ^{M}$. The choice of continuity
modulus for the symbol $\cdot $ in $L$ forces $\cdot ^{M}$ to satisfy%
\begin{equation}
d^{M}(a_{0}\cdot ^{M}b_{0},a_{1}\cdot ^{M}b_{1})\leq d^{M}\left(
a_{0},a_{1}\right) +d^{M}\left( b_{0},b_{1}\right) \text{.\label%
{Equation:bi-invariant}}
\end{equation}%
Conversely, any complete metric space with diameter $1$ endowed with a
binary operation satisfying Equation \eqref{Equation:bi-invariant} can be
seen as an $L$-structure. In this case, the class of $L$-structures contains
the important example of \emph{bi-invariant metric groups}. A bi-invariant
metric group is a complete metric space $\left( G,d^{G}\right) $ endowed
with a group operation $\cdot ^{G}$ with the property that left and right
translations, i.e.\ the maps $x\mapsto ax$ and $x\mapsto xa$ for $a\in G$,
are isometries. It is clear from the discussion above that a bi-invariant
metric group can be seen as a structure in the language $L$ just described.

Bi-invariant metric groups arise naturally in operator algebras, group
theory, and metric geometry. For instance, for every $n\in \mathbb{N}$, the
group $U_{n}$ of $n\times n$ unitary matrices is a bi-invariant metric group
when endowed with the metric $d\left( u,v\right) =2^{-1/2}\norm{u-v}_{2}$.
Here $\left\Vert a\right\Vert _{2}$ denotes the normalized \emph{%
Hilbert--Schmidt norm }(or Frobenius norm) $\tau \left( a^{\ast }a\right)
^{1/2}$ of a matrix $a$, where $\tau $ is the canonical trace of $n\times n$
matrices, suitable normalized so that $\tau \left( 1\right) =1$. This is a
particular instance of a more general class of examples, arising from von
Neumann algebra theory. If $M$ is a von Neumann algebra endowed with a
faithful normalized trace $\tau $, then the unitary group $U(M)$ of $M$ is a
bi-invariant metric group. The metric now is defined, as above, by $d\left(
u,v\right) =\left\Vert u-v\right\Vert _{2}$, where $\left\Vert a\right\Vert
_{2}$ denotes the $2$-norm $\tau \left( a^{\ast }a\right) ^{1/2}$ of $a$
with respect to the trace $\tau $. The case of unitary groups of matrices is
recovered in the case when $M$ is a full matrix algebra.

One can also consider different metrics on the unitary group $U_{n}$. For
instance, one can consider the metric $d\left( u,v\right) =\left\Vert
u-v\right\Vert $ induced by the \emph{operator norm }of matrices, which can
be concretely defined as the largest singular value. It is clear that left
and right translations in $U_{n}$ are isometries also with respect to this
metric, and so it yields another example of bi-invariant metric group.
Again, this is a particular instance of a more general class of examples
arising from C*-algebra theory. Indeed, if $A$ is a C*-algebra, then one can
consider the unitary group $U(A)$ as a bi-invariant metric group endowed
with a metric induced by the norm of $A$.

Generalizing the examples above, one can consider an arbitrary
unitary-invariant complete metric $d$ bounded by $1$ on $M_{n}\left( \mathbb{%
C}\right) $, and then endow the unitary group $U_{n}$ with the bi-invariant
metric induced by $d$. Any choice of such a unitary-invariant metric gives
rise to a different $L$-structure. Beside the ones considered above, an
important example of unitary-invariant metric $d$ is the (normalized) rank
metric, defined by $d\left( a,b\right) =\mathrm{rank}\left( a-b\right) /n$,
where $\mathrm{rank}$ denotes usual rank of matrices. Clearly, such a metric
is not only unitary-invariant, but also invariant with respect to
multiplication by arbitrary invertible matrices. Therefore, it defines an
invariant metric also in the group $\mathrm{GL}_{n}\left( \mathbb{C}\right) $
of invertible $n\times n$ complex matrices.

Another class of examples arises by considering, for $n\in \mathbb{N}$, the
symmetric group $S_{n}$, consisting of permutations of the set $\left\{
1,2,\ldots ,n\right\} $. In this case, the bi-invariant metric is given by
the (normalized) \emph{Hamming metric }%
\begin{equation*}
d\left( \sigma ,\tau \right) =\frac{1}{n}\left\vert \left\{ i\in \left\{
1,2,\ldots ,n\right\} :\sigma \left( i\right) \neq \tau \left( i\right)
\right\} \right\vert \text{.}
\end{equation*}
Finally, any (discrete) group $\Gamma $ can be regarded as a bi-invariant
metric group with respect to the \emph{trivial} $\left\{ 0,1\right\} $%
-valued metric defined by $d\left( g,h\right) =1$ whenever $g,h\in \Gamma $
are distinct.

We mention in passing that the classes of metric groups introduced here play
a crucial role in defining and studying important $\emph{regularity}$
properties for countable discrete groups. Indeed, given a countable discrete
group $\Gamma $ and a class $\mathcal{C}$ of bi-invariant metric groups, one
says that $\Gamma $ is $\mathcal{C}$-approximable if there exist strictly
positive real numbers $\delta _{g}$ for $g\in \Gamma \setminus \left\{
1\right\} $ such that, for every finite subset $F$ of $\Gamma \setminus
\left\{ 1\right\} $ and for every $\varepsilon >0$, there exists a
bi-invariant metric group $\left( G,d^{G},\cdot ^{G}\right) $ in $\mathcal{C}
$ and a function $\Phi :\Gamma \rightarrow G$ such that $\Phi \left(
1\right) =1$, $d\left( \Phi \left( gh\right) ,\Phi \left( g\right) \Phi
\left( h\right) \right) <\varepsilon $, and $d\left( \Phi \left( g\right)
,1\right) >\delta _{g}$ for every $g,h\in F$ \cite[Definition 1.6]%
{thom_about_2012}. By varying the class $\mathcal{C}$, one obtains various
regularity properties for countable discrete groups:

\begin{itemize}
\item when $\mathcal{C}$ is the class of permutation groups $S_{n}$ for $%
n\in \mathbb{N}$ endowed with the Hamming distance, a group is $\mathcal{C}$%
-approximable if and only if it is \emph{sofic }\cite%
{gromov_endomorphisms_1999,pestov_hyperlinear_2008,capraro_introduction_2015}%
;

\item when $\mathcal{C}$ is the class of unitary groups $U_{n}$ for $n\in 
\mathbb{N}$ endowed with the Hilbert-Schmidt distance, a group is $\mathcal{C%
}$-approximable if and only if it is $\emph{hyperlinear}$, which is in turn
equivalent to the assertion that the corresponding group von Neumann algebra
satisfies the Connes Embedding Problem \cite{ozawa_about_2004};

\item when $\mathcal{C}$ is the class of unitary groups $U_{n}$ for $n\in 
\mathbb{N}$ endowed with the operator norm, a group is $\mathcal{C}$%
-approximable if and only if it is \emph{matricially finite }or \emph{MF }%
\cite{blackadar_generalized_1997,carrion_groups_2013}, which is in turn
equivalent to the assertion that the corresponding group C*-algebra is
quasidiagonal \cite[Theorem 2.8]{carrion_groups_2013};

\item when $\mathcal{C}$ is the class of groups $\mathrm{GL}_{n}\left( 
\mathbb{C}\right) $ endowed with the rank metric, a group is $\mathcal{C}$%
-approximable if and only if it is \emph{linear sofic} \cite%
{arzhantseva_linear_2017}.
\end{itemize}

The interest on these properties is due on one hand to the fact that several
long-standing open problems in group theory---such as Gottschalk's
conjecture \cite{gromov_endomorphisms_1999}, Kaplansky's direct finiteness
conjecture \cite{elek_sofic_2004}, and the algebraic eigenvalue conjecture 
\cite{thom_sofic_2008}---have been settled for groups satisfying these extra
regularity properties. On other hand, these extra assumptions seem to be
very generous, to the point that no countable discrete group that does \emph{%
not }satisfy any of the approximation properties mentioned above is
currently known.

To conclude this detour on metric groups, and this list of examples of
languages, we mention a language closely related to the one just considered.
We let $L$ be a language consisting of a binary function symbol $\cdot $ as
above, together with a unary function symbol suggestively denoted by $%
\mathrm{inv}$ with the identity map as continuity modulus, and a constant
symbol $1$. A bi-invariant metric group $G$ can be naturally regarded as a
structure in this richer language $L$, where the interpretation of the unary
function $\mathrm{inv}$ is just the function $g\mapsto g^{-1}$ assigning to
an element of $G$ its inverse, and the constant symbol $1$ is interpreted as
the identity element of $G$. This example showcases an important point: a
given object, such as a bi-invariant metric group, can be seen as a
structure in possibly different languages. The choice of the language allows
one to keep track of which operations one is considering.

\subsection{Languages with domains of quantification\label%
{Subsection:domains}}

Ultimately, we would like to regard unital C*-algebras as structures in a
suitable language $L$ containing symbols for all the C*-algebra operations.
While not impossible, it is somewhat inconvenient to regard C*-algebras as
structures in the restricted setting considered in Subsection \ref%
{Subsection:single-sorted}. Naturally, such a language $L$ would contain
binary function symbols $+$ and $\cdot $ for sum and multiplication, a unary
function symbol for the adjoint operation, constant symbols for the additive
and multiplicative neutral elements, and a unary function symbols for the
scalar multiplication function $x\mapsto \lambda x$ for any given $\lambda
\in \mathbb{C}$, in addition to the distinguished metric symbol. However, in
view of the requirements on \emph{bounds }on relation symbols (including the
metric symbol) one can only consider \emph{bounded }metric spaces as
structures.\ It is therefore natural to then restrict to the \emph{unit ball 
}$A^{1}$ of a given C*-algebra $A$. This is not a real restriction, since it
is clear that all the information about $A$ is already present in $A^{1}$.
This however makes it problematic to have a function symbol for the addition
operation, since $A^{1}$ is not invariant under addition. A solution to this
would be replace the binary function symbol for addition with a binary
function symbol to denote the \emph{average operation }$\left( x,y\right)
\mapsto \left( x+y\right) /2$. (More generally, one could consider for $n\in 
\mathbb{N}$ and $\left( \lambda _{1},\ldots ,\lambda _{n}\right) \in \mathbb{%
C}^{n}$ such that $\sum_{i=1}^{n}\left\vert \lambda _{i}\right\vert \leq 1$,
an $n$-ary function symbol for the function $\left( x_{1},\ldots
,x_{n}\right) \mapsto \lambda _{1}x_{1}+\cdots +\lambda _{n}x_{n}$.) While
this is possible, the corresponding notion of structure that one obtains
seems to be far from the way in which, in practice, C*-algebras are regarded
as structures by C*-algebraists.

We pursue therefore a different path, which consists in introducing a more
general framework than the one considered in Subsection \ref%
{Subsection:single-sorted}. This is the framework of languages with \emph{%
domains of quantification}. Briefly, in this setting one adds to the
language a collection of symbols (domains of quantification) to be
interpreted as closed subsets of the structure. In this case, all the
requirements concerning the interpretations of function and relation
symbols, including the boundedness requirement on the metric, are only
imposed \emph{relatively }to a given domain, or tuple of domains. This
allows one consider structures, such as C*-algebras, where the metric is
globally unbounded, although it is bounded when restricted to any given
choice of domains of quantifications. In the case of C*-algebras, the
domains of quantifications will be interpreted as the balls centered at the
origin.

We now present the details. As we have just mentioned, in this setting, a
language $L$ is endowed with a collection $\mathcal{D}$ of \emph{domains of
quantification}. The set $\mathcal{D}$ is endowed with an ordering, which is
upward directed. In this case, for each function symbol $f$ in $L$ and for
each choice of \emph{input domains }$D_{1},\ldots ,D_{n_{f}}$, the language $%
L$ prescribes:

\begin{itemize}
\item its arity $n_{f}$;

\item an \emph{output domain }$D_{D_{1},\ldots ,D_{n_{f}}}^{f}$;

\item a continuity modulus\emph{\ }$\varpi _{D_{1},\ldots
,D_{n_{f}}}^{f}:[0,+\infty )^{n_{f}}\rightarrow \lbrack 0,+\infty )$.
\end{itemize}

Similarly, for each relation symbol $R$ in $L$ and for each choice of \emph{%
input domains }$D_{1},\ldots ,D_{n_{f}}$, the language $L$ prescribes:

\begin{itemize}
\item its arity $n_{R}$;

\item a \emph{bound }$J_{D_{1},\ldots ,D_{n_{f}}}^{R}$;

\item a continuity modulus $\varpi _{D_{1},\ldots ,D_{n_{R}}}^{R}:[0,+\infty
)^{n_{f}}\rightarrow \lbrack 0,+\infty )$.
\end{itemize}

Again, the language $L$ is assumed to contain a distinguished binary
function symbol $d$ (metric symbol).

In this case, an $L$-structure is a complete metric space $\left(
M,d^{M}\right) $ together with assignments:

\begin{itemize}
\item $D\mapsto D^{M}$ from the set of domains of quantification in $L$ to
the collection of closed subsets of $M$;

\item $f\mapsto f^{M}$ from the set of function symbols in $L$ to the
collection of functions $M^{n}\rightarrow M$ for $n\in \mathbb{N}$;

\item $R\mapsto R^{M}$ from the set of relation symbols in $L$ to the
collection of functions $M^{n}\rightarrow \mathbb{R}$ for $n\in \mathbb{N}$;
\end{itemize}

satisfying the following properties:

\begin{enumerate}
\item the collection $\left\{ D^{M}:D\in \mathcal{D}\right\} $ of closed
subsets of $M$ has dense union;

\item the assignment $D\mapsto D^{M}$ is order-preserving, where the
collection of closed subsets of $M$ is ordered by inclusion;

\item for every function symbol $f$ and choice of input domains $%
D_{1},\ldots ,D_{n_{f}}$, and $\bar{a},\bar{b}\in D_{1}\times \cdots \times
D_{n_{f}}$, one has that $f^{M}\left( \bar{a}\right) \in D_{D_{1},\ldots
,D_{n_{f}}}^{f}$ and $d^{M}\left( f\left( \bar{a}\right) ,f\left( \bar{b}%
\right) \right) \leq \varpi _{D_{1},\ldots ,D_{n_{f}}}^{f}\left( d(\bar{a},%
\bar{b})\right) $;

\item for every relation symbol $R$ and choice of input domains $%
D_{1},\ldots ,D_{n_{R}}$, and $\bar{a},\bar{b}\in D_{1}\times \cdots \times
D_{n_{R}}$, one has that $R^{M}\left( \bar{a}\right) \in J_{D_{1},\ldots
,D_{n_{f}}}^{R}$ and $\left\vert R\left( \bar{a}\right) -R\left( \bar{b}%
\right) \right\vert \leq \varpi _{D_{1},\ldots ,D_{n_{f}}}^{f}\left( d(\bar{a%
},\bar{b})\right) $.
\end{enumerate}

We conclude by noting that any language as defined in Subsection \ref%
{Subsection:single-sorted} can be seen as a particular instance of a
language with domains of quantification as defined in this section, by
declaring that the set $\mathcal{D}$ of domains of quantification is a
singleton $\left\{ D\right\} $. In this case, we omit the reference to such
a unique domain in the quantifiers, and write simply $\sup_{x}$ and $%
\inf_{x} $ instead of $\sup_{x\in D}$ and $\inf_{x\in D}$.

\subsection{C*-algebras as structures\label{Subsection:C*-language}}

We are now ready to discuss how C*-algebras can be seen as structures in
continuous logic, when one consider the framework introduced in Subsection %
\ref{Subsection:domains}.\ We briefly recall that a C*-algebra is,
abstractly, a complex algebra $A$ endowed with an involution $a\mapsto
a^{\ast }$ and a complete norm $a\mapsto \left\Vert a\right\Vert $
satisfying $\left\Vert ab\right\Vert \leq \left\Vert a\right\Vert \left\Vert
b\right\Vert $ for $a,b\in A$, and the C*-identity $\left\Vert a^{\ast
}a\right\Vert =\left\Vert a\right\Vert ^{2}$ for $a\in A$. For every $n\in 
\mathbb{N}$, the algebraic tensor product $M_{n}\left( \mathbb{C}\right)
\otimes A=M_{n}\left( A\right) $ is endowed with a canonical norm which
turns into a C*-algebra. We will always regard $M_{n}\left( A\right) $ as a
C*-algebra endowed with such a norm. We will only consider \emph{unital }%
C*-algebras, which are moreover endowed with a multiplicative identity
(unit) $1$.

The Gelfand--Neimark theorem guarantees that one can concretely represent
any (abstract) C*-algebras as a closed subalgebra of the algebra $B(H)$ of
bounded linear operators on a Hilbert space $H$. In this case the involution
is the map assigning to an operator its Hermitian adjoint, the norm is the
operator norm given by $\left\Vert T\right\Vert =\sup \left\{ \left\Vert
T\xi \right\Vert :\xi \in H,\left\Vert \xi \right\Vert \leq 1\right\} $ for $%
T\in B(H)$, and $1$ is the identity operator.

The language of C*-algebras $L^{\text{C*}}$ consists of:

\begin{itemize}
\item a sequence $\left\{ D_{n}:n\in \mathbb{N}\right\} $ of domains of
quantifications, linearly ordered by setting $D_{n}<D_{m}$ if and only if $%
n<m$;

\item binary function symbols for addition and multiplication;

\item for every $\lambda \in \mathbb{C}$, a unitary function symbol for the
scalar multiplication function $x\mapsto \lambda x$;

\item constant symbols for $0$ and $1$;

\item a unary relation symbol for the norm, as well as the metric symbol;

\item for every $n\in \mathbb{N}$, a $n^{2}$-ary relation symbol for the
norm of elements of $M_{n}\left( A\right) $.
\end{itemize}

We want to regard a C*-algebra $A$ as an $L^{\text{C*}}$-structure, where
the domain $D_{n}$ is interpreted as the ball of $A$ of center $0$ and
radius $n$, the function and relation symbols are interpreted in the obvious
way. Keeping this in mind, it is clear that one can define continuity
moduli, output domains, and bounds for the given function and relation
symbols in such a way that any C*-algebra meets the requirements that an $L^{%
\text{C*}}$-structure has by definition to satisfy. For example, let us
consider the function symbol for multiplication, and the input domains $%
D_{n} $ and $D_{m}$. In this case, one can declare the output domain to be $%
D_{nm}$, and the continuity modulus to be the function $\left(
t_{1},t_{2}\right) \mapsto mt_{1}+nt_{2}$.

\subsection{Formulas\label{Subsection:formulas}}

One of the upshots of regarding a class of objects as structures in
continuous logic is to obtain a corresponding notion of first-order
property. These are the properties that can be expressed through \emph{%
formulas}. We begin with the \emph{syntax }of formulas, and describe how
formulas in a given language are defined. So let us fix a language $L$, and
define the notion of $L$-formula. Intuitively, an $L$-formula is an
expression that describes a property of an $L$-structure, or of a tuple of
elements of an $L$-structure, by only referring to the given $L$-structure,
its elements, and its operations which are given by the interpretations of
the function and relation symbols in $L$.

Before we introduce the notion of $L$-formula, we need to consider the
notion of $L$-term. Informally, an $L$-term is an expression that described
how elements of a given $L$-structure can be combined together by using the
function symbols in $L$. To make this precise, we suppose that we have a
collection of symbols, usually denoted by $x,y,z,\ldots $, possibly with
decorations such as $x_{1},x_{2},x_{3},\ldots $, which we call \emph{%
variables}. Variables are used in the definition of terms and formulas, and
they should be thought of as \textquotedblleft
place-holders\textquotedblright , which can be possibly replaced by elements
of a structure. (This is analogous to the role of variables $\bar{x}$ in a
polynomial $p\left( \bar{x}\right) $ with coefficients in a ring $R$. These
variables \textquotedblleft substituted\textquotedblright\ by elements of $R$
when considering the corresponding polynomial function. In fact, this usage
of variables is a particular instance of the usage from model theory.) We
assume that to each variable $x$ is uniquely attached a domain of
quantification $D$, which we call the domain of $x$.

Now one can say briefly that an $L$-term is any expression that can be
formed starting from variables and constant symbols, and applying function
symbols from $L$. More extensively, one declare that:

\begin{itemize}
\item variables are $L$-terms;

\item constant symbols are $L$-terms;

\item if $t_{1},\ldots ,t_{n}$ are $L$-terms, and $f$ is an $n$-ary function
symbol in $L$, then $\left( f\left( t_{1},\ldots ,t_{n}\right) \right) $ is
an $L$-term.
\end{itemize}

Given an $L$-term $t$, one can speak of the variables that appear in $t$.
Formally, if $x$ is a variable, one can define the property that $x$ appears
in $t$ by induction on the complexity of $t$ as follows: if $t$ is a
variable, then $x$ appears in $t$ if and only if $x$ is equal to $t$; if $t$
is a constant symbol then $x$ does not appear in $t$; if $t=f\left(
t_{1},\ldots ,t_{n}\right) $ then $x$ appears in $t$ if and only if $x$
appears in $t_{i}$ for some $i\in \left\{ 1,2,\ldots ,n\right\} $. One then
write $t\left( x_{1},\ldots ,x_{k}\right) $ to denote the fact that the
variables that appear in $t$ are within $x_{1},\ldots ,x_{k}$.

The notion of $L$-term allows one to easily define the notion of \emph{%
atomic }$L$-formula. This is an expression $\varphi $ of the form $R\left(
t_{1},\ldots ,t_{n}\right) $, where $R$ is an $n$-ary function symbol in $L$
and $t_{1},\ldots ,t_{n}$ are $L$-terms. If $t_{1},\ldots ,t_{n}$ are
variables within $x_{1},\ldots ,x_{k}$, then one says that $\varphi $ has
free variables within $x_{1},\ldots ,x_{k}$, and write $\varphi \left(
x_{1},\ldots ,x_{k}\right) $.

Starting from the notion of atomic formula, one can define the arbitrary
formulas. Informally, formulas are expressions obtained by combining atomic
formulas by using \emph{logical connectives} and \emph{quantifiers}. In
classical (discrete) first-order logic, the logical connectives are the
usual symbols $\wedge $, $\vee $, $\lnot $, $\rightarrow $, $\leftrightarrow 
$, which can be thought of as expressions to denote Boolean functions. In
logic for metric structures, these Boolean functions are replaced with
arbitrary \emph{continuous }functions $q:\mathbb{R}^{d}\rightarrow \mathbb{R}
$. On the side of quantifiers, in the discrete setting these are the
expressions $\forall x$ and $\exists x$, where $x$ is a variable. In this
case, one also says that $\forall x$ and $\exists x$ are quantifiers over $x$%
. In the continuous setting, for each domain of quantification and for each
variable $x$ one has quantifiers $\sup_{x\in D}$ and $\inf_{x\in D}$ (this
justifies the name of \emph{domains of quantification}). One then formally
defines $L$-formulas by induction as follows:

\begin{itemize}
\item atomic $L$-formulas are $L$-formulas;

\item if $\varphi _{1},\ldots ,\varphi _{n}$ are $L$-formulas and $q:\mathbb{%
R}^{n}\rightarrow \mathbb{R}$ is a continuous function, then $q\left(
\varphi _{1},\ldots ,\varphi _{n}\right) $ is an $L$-formula;

\item if $\varphi $ is an $L$-formula, and $x$ is a variable with domain $D$%
, then $\inf_{x\in D}\varphi $ and $\sup_{x\in D}\varphi $ are $L$-formulas.
\end{itemize}

For brevity, a sequence of quantifiers $\sup_{x_{1}\in D_{1}}\cdots
\sup_{x_{n}\in D_{n}}$ is abbreviated by $\sup_{\bar{x}\in \bar{D}}$ where $%
\bar{x}$ is the tuple of variables $\left( x_{1},\ldots ,x_{n}\right) $ and $%
\bar{D}$ is the corresponding tuple of domains $\left( D_{1},\ldots
,D_{n}\right) $.

The variables that appear in an $L$-formula can be bound or free, depending
on whether they are in the \emph{scope }of a quantifier over them or not.
Formally, one declares when a variable $x$ appears freely in $\varphi $ by
induction on the complexity of $\varphi $ as follows:

\begin{itemize}
\item if $\varphi $ is an atomic formula $R\left( t_{1},\ldots ,t_{n}\right) 
$, then $x$ appears freely in $\varphi $ iff it appears in $t_{i}$ for some $%
i\in \left\{ 1,2,\ldots ,n\right\} $;

\item $x$ appears freely in $q\left( \varphi _{1},\ldots ,\varphi
_{n}\right) $ iff it appears freely in $\varphi _{i}$ for some $i\in \left\{
1,2,\ldots ,n\right\} $;

\item $x$ appears freely in $\inf_{y\in D}\varphi $ or $\sup_{y\in D}\varphi 
$ iff $x$ appears freely in $\varphi $ \emph{and} $x$ is different from $y$.
\end{itemize}

If the free variables of $\varphi $ are within $x_{1},\ldots ,x_{k}$, then
we write $\varphi \left( x_{1},\ldots ,x_{k}\right) $. In this case, $%
\varphi $ should be thought of as an expression describing how close a given 
$k$-tuple of elements of an $L$-structure is to satisfy a certain property.
When $\varphi $ has \emph{no }free variables, one says that $\varphi $ is an 
$L$-sentence. In this case, $\varphi $ should be thought of as an expression
describing how close a given structure is to satisfying a certain property.

As mentioned above, the purpose of formulas is to describe properties of
structures, or elements of structures. This is made precise by the semantic
notion of \emph{interpretation }of a formula in a given structure. We begin
with the interpretation of an $L$-term $t\left( x_{1},\ldots ,x_{k}\right) $
in an $L$-structure $M$, which is going to be a function $%
t^{M}:D_{1}^{M}\times \cdots \times D_{k}^{M}\rightarrow M$, $\bar{a}\mapsto
t^{M}(\bar{a})$, where $\left( D_{1},\ldots ,D_{k}\right) $ are the domains
of $\left( x_{1},\ldots ,x_{k}\right) $. Briefly, this is defined by
replacing the variables of $t$ with the given tuple $\bar{a}$ of elements of 
$M$, and by replacing constant symbols and function symbols $f$ with their
interpretations. Formally, this is defined, again, by induction of the
complexity, as follows:

\begin{itemize}
\item if $t\left( \bar{x}\right) =x_{i}$ for some $i\in \left\{ 1,2,\ldots
,k\right\} $, then $t^{M}\left( \bar{a}\right) =a_{i}$;

\item if $t$ is a constant symbol $c$, then $t^{M}(\bar{a})=c^{M}$;

\item if $t=f\left( t_{1},\ldots ,t_{n}\right) $, then $t^{M}\left( \bar{a}%
\right) =f^{M}(t_{1}^{M}(\bar{a}),\ldots ,t_{n}^{M}(\bar{a}))$.
\end{itemize}

Starting from the interpretations of terms, one can define the
interpretation of an $L$-formula $\varphi \left( x_{1},\ldots ,x_{k}\right) $%
, which is going to be a function $\varphi ^{M}:D_{1}^{M}\times \cdots
\times D_{k}^{M}\rightarrow \mathbb{R}$, $\bar{a}\mapsto \varphi ^{M}(\bar{a}%
)$. As above, this can be briefly defined by replacing all the terms,
relation symbols, and domains of quantifications that appear, with their
interpretation in $M$. Precisely, one can define this by induction on the
complexity, as follows:

\begin{itemize}
\item if $\varphi $ is the atomic formula $R\left( t_{1},\ldots
,t_{n}\right) $, then $\varphi ^{M}(\bar{a})=R^{M}(t_{1}(\bar{a}),\ldots
,t_{n}(\bar{a}))$;

\item if $\varphi $ is equal to $q\left( \varphi _{1},\ldots ,\varphi
_{n}\right) $, then $\varphi ^{M}(\bar{a})=q\left( \varphi _{1}^{M}(\bar{a}%
),\ldots ,\varphi _{n}^{M}(\bar{a})\right) $;

\item if $\varphi $ is equal to $\inf_{x\in D}\psi $, then $\varphi ^{M}(%
\bar{a})=\inf_{x\in D^{M}}\psi ^{M}(\bar{a})$, and similarly with $\sup $.
\end{itemize}

Clearly, when $\varphi $ is an $L$-sentence, its interpretation $\varphi
^{M} $ can be seen simply as a single real number. At this point, let us
pause and see what the notions of terms and formulas just introduced
correspond in the examples of languages that we have seen so far.

When $L$ is the language only containing the metric symbol, the only terms
are just single variables, and the only atomic formulas are of the form $%
d\left( x,y\right) $ where $x,y$ are variables. The interpretation of such a
formula in a complete metric space $\left( M,d^{M}\right) $ is the function $%
M\times M\rightarrow \mathbb{R}$, $\left( a,b\right) \mapsto d^{M}\left(
a,b\right) $. In this case, an example of sentence is the formula $\varphi $
given by $\sup_{x}\sup_{y}d\left( x,y\right) $. The interpretation of $%
\varphi $ is a metric space $\left( M,d^{M}\right) $ is, clearly, the
diameter of $M$.

A slightly more interesting example comes from considering the language $L$
for bi-invariant metric groups. In this case, a \emph{term }in the free
variables $x_{1},\ldots ,x_{n}$ can be seen as a \emph{parenthesized word }%
in the variables $x_{1},\ldots ,x_{n}$, such as $\left( x_{1}\cdot \left(
x_{2}\cdot x_{3}\right) \right) $. (Formally, the terms $\left( x_{1}\cdot
\left( x_{2}\cdot x_{3}\right) \right) $ and $\left( \left( x_{1}\cdot
x_{2}\right) \cdot x_{3}\right) $ are distinct terms, although they have the
same interpretation in any bi-invariant metric group, or more generally in
any $L$-structure for which the interpretation of the binary function symbol 
$\cdot $ is an associative operation.) The interpretation of a term $t\left( 
\bar{x}\right) $ in a bi-invariant metric group $G$ is then the function $%
\bar{a}\mapsto t^{G}(\bar{a})$ that replaces the variables $\bar{x}$ with
the tuple $\bar{a}$, and then computes the products in $G$. An example of $L$%
-sentence in this case is given by the $L$-formula $\sup_{x}\sup_{y}d\left(
x\cdot y,y\cdot x\right) $. The interpretation of such an $L$-sentence $%
\varphi $ in a bi-invariant metric group is then a real number, which is the
supremum of distances of commutators in $G$ from the identity. This can be
thought of as a measure of how far $G$ is from being abelian. Clearly, $G$
is abelian if and only if $\varphi ^{G}=0$.

We conclude this series of examples by considering the case of the language
for C*-algebras $L^{\text{C*}}$. In this case, an $L^{\text{C*}}$-term in
the variables $x_{1},\ldots ,x_{n}$ can be seen as a complex polynomial with
constant term in the variables $x_{1},\ldots ,x_{n}$ and their
\textquotedblleft formal adjoint\textquotedblright\ $x_{1}^{\ast },\ldots
,x_{n}^{\ast }$ ($\ast $-polynomials). This is strictly speaking not
entirely correct, since the term $\left( \left( x_{1}+x_{2}\right) ^{\ast
}\right) $, for instance, is not equal to the term $\left( \left(
x_{1}\right) ^{\ast }+\left( x_{2}\right) ^{\ast }\right) $. However, they
have the same interpretation in every C*-algebra, and hence we can for all
purposes identify them, and write them simply as $x_{1}^{\ast }+x_{2}^{\ast
} $. The same applies to the terms $\left( \left( x_{1}+x_{2}\right)
+x_{3}\right) $ and $\left( x_{1}+\left( x_{2}+x_{3}\right) \right) $, which
we write simply as $x_{1}+x_{2}+x_{3}$. Thus, an atomic $L^{\text{C*}}$%
-formula in the free variables $x_{1},\ldots ,x_{n}$ can be seen as an
expression of the form $\left\Vert \mathfrak{p}(x_{1},\ldots
,x_{n})\right\Vert $ where $\mathfrak{p}(x_{1},\ldots ,x_{n})$ is a
*-polynomial in the variables $x_{1},\ldots ,x_{n}$. Its interpretation in a
C*-algebra $A$ is the function $A^{n}\rightarrow \mathbb{R}$, $\bar{a}%
\mapsto \mathfrak{p}(\bar{a})$. An example of $L^{\text{C*}}$-sentence in
this setting is, for instance the $L^{\text{C*}}$-formula $\sup_{x\in
D_{1}}\left\vert \left\Vert x^{\ast }x\right\Vert -\left\Vert x\right\Vert
^{2}\right\vert $. (Recall that in $L^{\text{C*}}$ we have domains of
quantification $D_{n}$ for $n\in \mathbb{N}$, which are interpreted in a
C*-algebra as the balls of radius $n$ centered at the origin.) Clearly, the
interpretation of such an $L^{\text{C*}}$-sentence is equal to $0$ in any
C*-algebra, in view of the C*-identity.

\subsection{Multi-sorted languages}

One can consider a further generalization of the framework introduced in
Subsection \ref{Subsection:domains}, by allowing \emph{multi-sorted languages%
}. In this setting, the language prescribes a collection $\mathcal{S}$ of
sorts. Each sort $S$ in $\mathcal{S}$ comes with a corresponding collection $%
\mathcal{D}_{S}$ of domains of quantification for $S$.\ Furthermore, each $n$%
-ary function and relation symbol has a prescribed $n$-tuple of $\emph{input}
$ sorts and, in the case of function symbols, an output sort as well. A
structure $M$ then consists of a family $\left( M^{S}\right) _{S\in \mathcal{%
S}}$ of metric spaces, one for each sort, together with the corresponding
interpretations of domains of quantification, and function and relation
symbols, subject to the same requirements as in Subsection \ref%
{Subsection:domains}. All the notions and results that we will present admit
natural generalizations to the case of multi-sorted languages.

\section{Axiomatizability and definability\label{Section:axiom}}

\subsection{Axiomatizable classes}

As we have mentioned in Subsection \ref{Section:one}, model theory focuses
on the study of \emph{classes }of objects of the same kind, rather than
single objects on their own. We have introduced the notion of language, in
order to make precise what it means that a class consists objects of the
same kind. We have also defined the notion of formula in a given language,
which is an expression that allows one to describe properties of an
arbitrary tuple of elements of a structure or, in the case of sentences
(formulas with no free variables), of the structure itself. This leads to
the important notion of elementary or axiomatizable property. By considering
the class of structures that satisfy the given property, one can
equivalently speak of elementary or axiomatizable class.

To give the precise definition, let us fix a language $L$, and a class of $L$%
-structure $\mathcal{C}$. Recall that given an $L$-sentence $\varphi $ and a
real number $r$, we let $\varphi ^{M}$ be the interpretation of $\varphi $
in $M$, which is a real number. An $L$-\emph{condition }is an expression $%
\varphi \leq r$ where $\varphi $ is an $L$-sentence and $r\in \mathbb{R}$ is
a real number. An $L$-structure $M$ satisfies the condition $\varphi \leq r$
if and only if $\varphi ^{M}\leq r$.

\begin{definition}
\label{Definition:axiomatizable}The class $\mathcal{C}$ is $L$-\emph{%
axiomatizable} or $L$-\emph{elementary }if there exists a family of $L$%
-conditions $\varphi _{i}\leq r_{i}$ for $i\in I$ such that, given an
arbitrary $L$-structure $M$, we have that $M$ belongs to $\mathcal{C}$ if
and only if $M$ satisfies the condition $\varphi _{i}\leq r_{i}$ for every $%
i\in I$. In this case, one refers to the conditions $\varphi _{i}\leq r_{i}$
for $i\in I$ as \emph{axioms }for $\mathcal{C}$.

One then says that a property $\mathcal{P}$ is $L$-\emph{elementary} if the
class of $L$-structures that satisfy $\mathcal{P}$ is elementary.
\end{definition}

It is clear that in the definition of $L$-axiomatizable class, up to
replacing each sentence $\varphi _{i}$ with the sentence $\max \left\{
\varphi _{i}-r_{i},0\right\} $, one can always assume that $r_{i}=0$ and
that $\varphi _{i}$ only attains non-negative values. In practice, when the
language $L$ is clear from the context, one simply speaks of axiomatizable
or elementary class, omitting the explicit reference to the language $L$.
Intuitively, the assertion that a property is $L$-elementary means that it
can be described by only referring to the elements of a given structure and
to the operations of the structure which are named by the language $L$.

In order to gain some familiarity with the concept of axiomatizable class,
let us look at examples, drawn from the list of languages and structures
that we have considered in Subsection \ref{Section:one}. We have introduced
above the class of bi-invariant metric groups. For convenience, we can
regard these objects as structures in the language $L$ that contains,
besides the metric symbol, a binary function symbol $\cdot $ for the
operation, a unary function symbol \textrm{inv }for the inverse map, and a
constant symbol $1$ for the neutral element. Thus the class $\mathcal{C}$ of
bi-invariant metric groups forms a class of $L$-structures. This class is
easily seen to be axiomatizable, as witnessed by the axiom:

\begin{itemize}
\item $\sup_{x}d\left( x\cdot 1,x\right) \leq 0$, which prescribes that $1$
is interpreted as a neutral element;

\item $\sup_{x}d\left( x\cdot \mathrm{inv}\left( x\right) ,1\right) \leq 0$,
which prescribes that \textrm{inv}$\left( x\right) $ is interpreted as the
inverse of $x$;

\item $\sup_{x,y,z}d\left( \left( x\cdot \left( y\cdot z\right) \right)
,\left( \left( x\cdot y\right) \cdot z\right) \right) \leq 0$, which
prescribes that the operation is associative;

\item $\sup_{x_{0},x_{1},y}\max \left\{ \left\vert d\left(
x_{0}y,x_{1}y\right) -d\left( x_{0},x_{1}\right) \right\vert ,\left\vert
d\left( yx_{0},yx_{1}\right) -d\left( x_{0},x_{1}\right) \right\vert
\right\} \leq 0$, which forces left and right translations to be isometric.
\end{itemize}

Several natural properties of bi-invariant metric groups are elementary in
this language. For instance, the property of being abelian is elementary, as
witnessed by the single axiom $\sup_{x,y}d\left( xy,yx\right) \leq 0$.

Given a collection $\mathcal{C}$ of structures in a language $L$, one of the
fundamental problems of the model-theoretic study of $\mathcal{C}$ is
understanding which sub-classes of $\mathcal{C}$ (including $\mathcal{C}$
itself) are $L$-axiomatizable. While in some cases this might be apparent,
other cases might be more subtle. Often a proof of axiomatizability of a
given property might require obtaining an equivalent \textquotedblleft
explicit\textquotedblright\ or \textquotedblleft
quantitative\textquotedblright\ characterization of such a property, which
is easily seen to be captured by $L$-sentences. Finally, there are many
natural properties which turn out to be \emph{not }elementary (although
there might be other ways to describe them model-theoretically).

Due to the importance of this task, model theory has developed many useful
criteria that can be used in axiomatizability proofs. Often such criteria
provide a \textquotedblleft softer\textquotedblright\ approach, which allows
one to prove that a certain class is axiomatizable without the need of
explicitly write down axioms for it. Here, we will content ourselves to
verify that certain classes of structures are axiomatizable by directly
applying the definition. In the next section, we will consider the subtle
problem of axiomatizability for various important classes of C*-algebras.

\subsection{Axiomatizability in C*-algebras\label%
{Subsection:axiomatizability-C*}}

A substantial amount of the recent efforts in the model-theoretic study of
C*-algebras has been directed into understanding which classes of
C*-algebras are axiomatizable in the canonical language $L^{\text{C*}}$
described in Subsection \ref{Subsection:C*-language}, starting from the
class of C*-algebras itself. Recall that we are tacitly assuming all
C*-algebras to be unital.

\subsubsection{C*-algebras}

A C*-algebra $A$ is a unital Banach algebra with a conjugate-linear
involution $a\mapsto a^{\ast }$ (unital Banach *-algebra) satisfying the
C*-algebra identity $\left\Vert a^{\ast }a\right\Vert =\left\Vert
a\right\Vert ^{2}$. It is fairly obvious that one can write down sentences
that describe that an $L^{\text{C*}}$-structure satisfying such properties.
For example, the assertion that the involution is conjugate linear is
captured by the family of axioms 
\begin{equation*}
\sup_{x\in D_{n}}\left\Vert \left( \lambda x\right) ^{\ast }-\overline{%
\lambda }x^{\ast }\right\Vert \leq 0\text{,}
\end{equation*}%
where $n\in \mathbb{N}$, $\lambda $ varies among all complex numbers, and $%
\overline{\lambda }$ denotes the conjugate of $\lambda $. Similarly,
submultiplicativity of multiplication is reflected by the conditions%
\begin{equation*}
\sup_{x,y\in D_{n}}\left( \left\Vert xy\right\Vert -\left\Vert x\right\Vert
\left\Vert y\right\Vert \right) \leq 0
\end{equation*}%
for $n\in \mathbb{N}$, while the C*-identity is captured by the conditions 
\begin{equation*}
\sup_{x\in D_{n}}\left\vert \left\Vert x^{\ast }x\right\Vert -\left\Vert
x\right\Vert ^{2}\right\vert \leq 0
\end{equation*}%
for $n\in \mathbb{N}$.

The only tricky point is that, as we discussed, when we regard a C*-algebra
as an $L^{\text{C*}}$-structure, we insist that the domain $D_{n}$ is
interpreted as the ball of radius $n$ centered at the origin. Now, in
general this need not be true in an arbitrary $L^{\text{C*}}$-structure.
Therefore we need to add axioms that enforce this behaviour of the
interpretation of domains. For $n\in \mathbb{N}$, we can consider the
conditions 
\begin{equation}
\sup_{x\in D_{n}}\left\Vert x\right\Vert \leq n\text{,\label{Equation:ball-1}%
}
\end{equation}%
which clearly guarantees that the norm of any element in the interpretation
of $D_{n}$ is at most $n$. At this point, we are only missing axioms that
guarantee that any element of norm at most $n$ actually belongs to $D_{n}$.
This is made sure by the axioms%
\begin{equation}
\sup_{x\in D_{m}}\inf_{y\in D_{n}}\left( \left\Vert x-y\right\Vert -\max
\left\{ \left\Vert x\right\Vert -n,0\right\} \right) \leq 0\text{\label%
{Equation:ball}}
\end{equation}%
for $n\leq m$. Indeed, observe that if $x$ is an element of a C*-algebra $A$
which belongs to the ball of radius $m$ centered at the origin, if one
actually has that $\left\Vert x\right\Vert \leq n$, then $x$ can also be
seen as an element $y$ of the ball of radius $n$, witnessing that %
\eqref{Equation:ball} holds.\ Otherwise if $n<\left\Vert x\right\Vert \leq m$%
, then $y:=\frac{n}{\left\Vert x\right\Vert }x$ is an element of the ball of
radius $n$ such that $\left\Vert x-y\right\Vert \leq \left\Vert x\right\Vert
-n$, again witnessing that \ref{Equation:ball} holds.

In order to see that these axioms are sufficient to enforce the desired
behaviour on the interpretation of the domains of quantification $D_{n}$ for 
$n\in \mathbb{N}$, suppose that $M$ is an $L^{\text{C*}}$-structure
satisfying \eqref{Equation:ball-1} and \eqref{Equation:ball}, as well as the
axioms for unital Banach *-algebras satisfying the C*-identity. We claim
then that the interpretation $D_{n}^{M}$ is precisely the ball of $M$ of
radius $n$ centered at the origin. Indeed, suppose that $x\in M$ is such
that $\left\Vert x\right\Vert \leq n$.\ Then, by the definition of
structure, $\bigcup_{n\in \mathbb{N}}D_{n}^{M}$ is dense in $M$. Hence, for
every $\varepsilon >0$ there exists $m\in \mathbb{N}$ and $y\in D_{m}^{M}$
such that $\left\Vert x-y\right\Vert \leq \varepsilon $. After replacing $y$
with $\frac{y}{1+\varepsilon /n}$ we can assume that $\left\Vert
y\right\Vert \leq n$. Therefore by \eqref{Equation:ball} there exists $z\in
D_{n}^{M}$ such that $\left\Vert y-z\right\Vert \leq \varepsilon $ and hence 
$\left\Vert x-z\right\Vert \leq 2\varepsilon $. Since this holds for every $%
\varepsilon >0$, since $D_{n}^{M}$ is closed we conclude that $x\in
D_{n}^{M} $. Conversely if $x\in D_{n}^{M}$ then $\left\Vert x\right\Vert
\leq n$ by \eqref{Equation:ball-1}.

In conclusion, we have shown that the class of C*-algebras is axiomatizable
in the language $L^{\text{C*}}$ introduced in Subsection \ref%
{Subsection:C*-language}. This paves up the way of establishing similar
results for other subclasses that naturally arise in C*-algebra theory. This
can be seen as the necessary first step towards the application of methods
from logic to C*-algebras.

\subsubsection{Abelian C*-algebras}

The first natural class to consider is the class of \emph{abelian}
C*-algebras, which are the C*-algebras for which the multiplication is
commutative. This very definition makes it clear that this class is
axiomatizable, by the axiom $\sup_{x,y\in D_{1}}\left\Vert xy-yx\right\Vert
\leq 0$. Abelian C*-algebras are precisely those of the form $C\left(
X\right) $ for some compact Hausdorff space $X$ (endowed with the pointwise
operations and the supremum norm). This motivates the assertion that
arbitrary C*-algebras can be regarded as a noncommutative analog of compact
Hausdorff spaces, and C*-algebra theory as noncommutative topology.

The class of \emph{nonabelian }C*-algebras is also axiomatizable, although
this is not immediately obvious from the definition. However, this is made
it apparent by the following equivalent characterization: a C*-algebra is
nonabelian if and only if it contains an element $x$ such that $\left\Vert
x\right\Vert =1$ and $x^{2}=0$ \cite[Proposition II.6.4.14]%
{blackadar_operator_2006}. Thus the class of nonabelian C*-algebra is
axiomatizable as witnessed by the condition $\inf_{x\in D_{1}}\left\Vert
x^{2}\right\Vert -\left\Vert x\right\Vert \leq -1$.

\subsubsection{Nontrivial C*-algebras}

A C*-algebra $A$ is nontrivial if it has dimension at least $2$ or,
equivalent, $A$ is not isomorphic to $\mathbb{C}$ with its canonical
C*-algebra structure. If $A$ is nontrivial, then it contains a selfadjoint
element $a\in A$ such that the abelian C*-subalgebra $A_{0}$ of $A$
generated by $a$ and $1$ has dimension at least $2$. Thus $A_{0}$ is
isomorphic to the algebra $C\left( X\right) $ of continuous function over a
compact Hausdorff space $X$ with at least $2$ points. Hence by normality of $%
X$ we can find positive elements $b,c\in A_{0}$ such that $\left\Vert
b\right\Vert =\left\Vert c\right\Vert =\left\Vert b-c\right\Vert =1$. Recall
that any positive element $b$ in a C*-algebra is of the form $a^{\ast }a$.
This shows that the class of nontrivial C*-algebras is axiomatized by the
condition%
\begin{equation*}
\sup_{x,y\in D_{1}}\min \left\{ \left\Vert x\right\Vert ,\left\Vert
y\right\Vert ,\left\Vert x^{\ast }x-y^{\ast }y\right\Vert \right\} \geq 1%
\text{.}
\end{equation*}

\subsubsection{$n$-subhomogeneous C*-algebras}

As a generalization of the class of abelian C*-algebras, one can consider
the class of $n$-subhomogeneous C*-algebras for some $n\in \mathbb{N}$.\
Recall that a C*-algebra is $n$-subhomogeneous if and only if all its
irreducible representation are $k$-dimensional for some $k\leq n$. When $n=1$%
, this recovers the class of abelian C*-algebras. It is not obvious by this
definition that the class of $n$-subhomogeneous C*-algebras is
axiomatizable. Indeed, this definition refers to entities, such as
irreducible representations, that are \emph{external }to the algebra itself.
We would rather need an equivalent characterization that only refers to the 
\emph{elements }of the algebra and their relations as expressed by the norm
and *-algebra operations. Such a characterization can be extracted from a
theorem of Amitsur--Levitzki, which isolates an algebraic relation that is
satisfied by all the elements of $M_{k}(\mathbb{C})$ for $k\leq n$, but is
not satisfied by some elements of $M_{d}\left( \mathbb{C}\right) $ whenever $%
d>n$. This relation is given by the expression%
\begin{equation}
\sum_{\sigma \in S_{2n}}\mathrm{sgn}\left( \sigma \right) x_{\sigma \left(
1\right) }x_{\sigma \left( 2\right) }\cdots x_{\sigma \left( 2n\right) }=0%
\text{\label{Equation:AL}}
\end{equation}%
where $S_{2n}$ denotes the group of permutations of the set $\left\{
1,2,\ldots ,2n\right\} $, and $\mathrm{sgn}\left( \sigma \right) \in \left\{
\pm 1\right\} $ denotes the parity of the given permutation. This allows one
to conclude that a C*-algebra $A$ is $n$-subhomogeneous if and only if any $%
n $-tuple of elements of $A$ satisfies \eqref{Equation:AL}. In order words,
the class of $n$-subhomogeneous C*-algebras is axiomatized by the condition 
\begin{equation*}
\sup_{x\in D_{1}}\left\Vert \sum_{\sigma \in S_{2n}}\mathrm{sgn}\left(
\sigma \right) x_{\sigma \left( 1\right) }x_{\sigma \left( 2\right) }\cdots
x_{\sigma \left( 2n\right) }\right\Vert \leq 0\text{.}
\end{equation*}%
A similar argument shows that the class of algebras that are \emph{not }$n$%
-subhomogeneous is also axiomatizable.

While the above discussion is reassuring, it turns out that several natural
classes of C*-algebras which are key importance in modern C*-algebra theory
and in the classification program are \emph{not }axiomatizable. These
include the classes of simple C*-algebras, nuclear C*-algebra, exact
C*-algebras, UHF C*-algebras, AF C*-algebras, and several other. To see why
this is the case, we will need to develop a bit more machinery, so we
postpone the proof to Subsection \ref{Subsection:nonelementary}. On the
positive side, these classes of algebras can still be captured
model-theoretically, although in a slightly more sophisticated way. This
will be subject of Subsection \ref{Subsection:nuclearity}.

\subsection{Definable sets\label{Subsection:definable-sets}}

The notion and study of definability is arguably one of the cornerstones of
model theory, both in the discrete setting and in the continuous one. In
discrete first-order logic, a subset of a structure is definable whenever
can be written as the set of elements that satisfy a certain formula. The
naive analogue of this definition in the metric setting turns out to be too
generous. The right generalization involves the notion, which is unique to
the metric setting, of \emph{stability} of formulas and relations.

As usual, we begin with a general discussion of stability in logic for
metric structures, before specifying the analysis to C*-algebras. Let us
therefore fix an arbitrary language $L$ (with domains of quantification),
and an elementary class $\mathcal{C}$. (To fix the ideas, one can think of $%
L $ to be the language of C*-algebras, and $\mathcal{C}$ be the class of all
C*-algebras.) Fix also a tuple $\bar{x}=\left( x_{1},\ldots ,x_{n}\right) $
of variables with corresponding domains $\bar{D}=\left( D_{1},\ldots
,D_{n}\right) $, and let $\mathfrak{F}\left( \bar{x}\right) $ be the
collection of $L$-formulas with free variables from $\bar{x}$. Then $%
\mathfrak{F}\left( \bar{x}\right) $ admits a natural real Banach algebra
structure, induced from the algebra structure on $\mathbb{R}$. For instance,
the sum of formulas $\varphi ,\psi $ is just the formula $\varphi +\psi $.
Furthermore, one can define a seminorm on $\mathfrak{F}\left( \bar{x}\right) 
$ by setting%
\begin{equation*}
\left\Vert \varphi \right\Vert =\sup \left\{ \varphi ^{M}(\bar{a}):M\in 
\mathcal{C},\bar{a}\in D_{1}^{M}\times \cdots \times D_{n}^{M}\right\} \text{%
.}
\end{equation*}%
The Hausdorff completion $\mathfrak{M}\left( \bar{x}\right) $ of $\mathfrak{F%
}\left( \bar{x}\right) $ with respect to such a seminorm is then a Banach
algebra, whose elements are called \emph{definable predicates }(in the
language $L$ relative to the class $\mathcal{C}$). After identifying a
formula with the corresponding element of $\mathfrak{M}\left( \bar{x}\right) 
$, one can regard formulas as definable predicates. Conversely, definable
predicates are precisely the uniform limits of formulas. Given a definable
predicate $\varphi $ and a structure $M\in \mathcal{C}$ one can define its
interpretation $\varphi ^{M}$, which is a function $\varphi
^{M}:D_{1}^{M}\times \cdots \times D_{k}^{M}\rightarrow \mathbb{R}$.

As a natural \emph{completion }of the space of formulas, definable
predicates make it easier to develop the theory smoothly. At the same time,
definable predicates can for all purposes being replaced by formulas, and
vice versa. For instance, it is easy to see that in the definition of
axiomatizable class---Definition \ref{Definition:axiomatizable}---one can
equivalently consider definable predicates rather than sentences. While this
is an obvious observation, we state it explicitly due to its importance.

\begin{proposition}
\label{Proposition:axiomatize-with-predicates}Suppose that $L$ is a
language, and $\mathcal{C}$ is a class of $L$-structure. Then $\mathcal{C}$
is $L$-axiomatizable if and only if there exist a family $\left( \varphi
_{i}\right) _{i\in I}$ of \emph{definable predicates} in the language $L$
with no free variables and a family $\left( r_{i}\right) _{i\in I}$ of real
numbers such that, for every $L$-structure $M$, $M\in \mathcal{C}$ if and
only if $\varphi _{i}^{M}\leq r_{i}$ for every $i\in I$.
\end{proposition}

Among definable predicates, there is a particularly important class: the
stable\emph{\ }ones.

\begin{definition}
\label{Definition:stable-predicate}A definable predicate $\varphi \left( 
\bar{x}\right) $ is \emph{stable} if it satisfies the following: for every $%
\varepsilon >0$ there exists $\delta >0$ such that if $M\in \mathcal{C}$ and 
$\bar{a}\in D_{1}^{M}\times \cdots \times D_{n}^{M}$ satisfies $\left\vert
\varphi \left( \bar{a}\right) \right\vert <\delta $, then there exists $\bar{%
b}\in D_{1}^{M}\times \cdots \times D_{n}^{M}$ such that $d\left( \bar{a},%
\bar{b}\right) <\varepsilon $ and $\varphi \left( \bar{b}\right) =0$.
\end{definition}

The terminology just introduced is consistent with \cite%
{carlson_omitting_2014}, although the term \textquotedblleft \emph{weakly }%
stable\textquotedblright\ is used in \cite{farah_model_2017} in lieu of
\textquotedblleft stable\textquotedblright , to match the corresponding
notion of \textquotedblleft weakly stable relation\textquotedblright\ from
the C*-algebra literature.

Verifying that a given predicate is stable can be a subtle problem, and such
stability problems arise frequently in operator algebras, geometric group
theory, and metric geometry. Their solution, both in the positive and in the
negative, often requires to use or develop deep and interesting theory. At
the same time, establishing that a given predicate is indeed stable often
has numerous and interesting consequences for the class of structures under
consideration. In the theory of operator algebras, a problem closely related
to stability is the problem of \emph{liftability }of relations; see \cite%
{loring_lifting_1997}.

The notion of stable predicate allows one to introduce the notion of
definability. As above, we assume that $\mathcal{C}$ is an elementary class
of $L$-structure. Suppose now that $S:M\mapsto S\left( M\right) $ is an
assignment of closed subspaces $S\left( M\right) \subset D_{1}^{M}\times
\cdots \times D_{n}^{M}$ to structures $M$ in $\mathcal{C}$.

\begin{definition}
The assignment $S:M\mapsto S\left( M\right) $ is a \emph{definable set }if
there exists a stable definable predicate $\varphi $ such that $S\left(
M\right) $ is the \emph{zeroset} $Z\left( \varphi ^{M}\right) =\left\{ \bar{a%
}\in D_{1}^{M}\times \cdots \times D_{n}^{M}:\varphi ^{M}\left( \bar{a}%
\right) =0\right\} $.
\end{definition}

Among other things, the importance of definable sets lies in the fact that
they can be allowed as additional \emph{domains of quantification }without
changing the notion of definable predicate and axiomatizable class. This is
the content of the following proposition, established in \cite[Theorem 9.17]%
{ben_yaacov_model_2008}; see also \cite[Theorem 3.2.2]{farah_model_2017}.

\begin{proposition}
\label{Proposition:quantify}Suppose that $S:M\mapsto S\left( M\right)
\subset D_{1}^{M}\times \cdots \times D_{n}^{M}$ is a definable set for the
elementary class of $L$-structures $\mathcal{C}$. If $\bar{x}=\left(
x_{1},\ldots ,x_{n}\right) $ is a tuple of variables with domains $\left(
D_{1},\ldots ,D_{n}\right) $, $\bar{y}=\left( y_{1},\ldots ,y_{k}\right) $
is a tuple of variables with domains $\left( E_{1},\ldots ,E_{k}\right) $,
and $\varphi \left( \bar{x},\bar{y}\right) \in \mathfrak{M}\left( \bar{x},%
\bar{y}\right) $ is a definable predicate in the free variables $\left( \bar{%
x},\bar{y}\right) $, then there exists a definable predicate $\psi \left( 
\bar{y}\right) \in \mathfrak{M}\left( \bar{y}\right) $ such that, for every
structure $M$ in $\mathcal{C}$ and $\bar{b}\in E_{1}^{M}\times \cdots \times
E_{k}^{M}$,%
\begin{equation*}
\psi ^{M}\left( \bar{b}\right) =\inf \left\{ \varphi ^{M}\left( \bar{a},\bar{%
b}\right) :\bar{a}\in S\left( M\right) \right\} \text{.}
\end{equation*}%
The same conclusion holds when replacing $\inf $ with $\sup $.
\end{proposition}

One way to think about Proposition \ref{Proposition:quantify} is that, if $%
\varphi \left( \bar{x},\bar{y}\right) $ is a definable predicate, then the
expression%
\begin{equation*}
\inf_{\bar{x}\in S}\varphi \left( \bar{x},\bar{y}\right)
\end{equation*}%
is also, or more precisely can be identified with, a definable predicate.
This gives us much more flexibility in constructing definable predicates,
and it will be used in crucial way to show that certain classes of
structures are elementary.

At this point, examples of the general notions just introduced are in order.
As usual, we being with examples from the theory of bi-invariant metric
groups, where it is easier to fix the ideas. We consider the language $L$
for bi-invariant metric groups consisting of function symbols for group
operations, beside the metric symbol, and the elementary class $\mathcal{C}$
of bi-invariant metric group. A natural definable predicate, which is in
fact a formula, that one can consider in this setting is $\varphi \left(
x,y\right) =d\left( x\cdot y,y\cdot x\right) $. Clearly, such a formula
measures how close two elements of a given bi-invariant metric group are to
be commuting. Several important problems in operator algebras and metric
geometry boils down to the question of whether $\varphi $ is stable for a
certain class of bi-invariant metric groups. An important result of
Voiculescu shows that $\varphi $ is \emph{not }stable for the class $\left\{
U_{n}:n\in \mathbb{N}\right\} $ of unitary groups endowed with the operator
norm \cite{voiculescu_asymptotically_1983}. On the other hand, such a
formula is stable for the class of unitary groups endowed with the
normalized Hilbert--Schmidt norm, as shown by Glebsky \cite%
{glebsky_almost_2010}. Recently the analogous assertion for the class $%
\left\{ S_{n}:n\in \mathbb{N}\right\} $ of permutation groups with the
normalized Hamming distance has been established by Arzhantseva--Paunescu 
\cite{arzhantseva_almost_2015}. This important result relies on the
Elek--Szab\'{o} result on essential uniqueness of sofic representations of
amenable groups \cite{elek_sofic_2011}.

\subsection{Definability in C*-algebras\label{Subsection:definability-C*}}

Understanding which subsets of C*-algebras are definable is a basic but
fundamental problem, that underpins most of the further model-theoretic
analysis of C*-algebras. It turns out that, fortunately, several important
subsets of C*-algebras are indeed definable. Let us therefore consider the
language of C*-algebras $L^{\text{C*}}$ and the class $\mathcal{C}$ of $L^{%
\text{C*}}$-structures consisting of all C*-algebras. In the following we
will consider definable predicates and definable sets with respect to such a
class $\mathcal{C}$.

\subsubsection{The unitary group}

We being by considering the $\emph{unitary}$ \emph{group }$U\left( A\right) $
of a C*-algebra $A$. This is the set of elements $u$ of $A$ such that $%
uu^{\ast }=u^{\ast }u=1$. Since $U\left( A\right) $ is a subset of the unit
ball of $A$, the assignment $A\mapsto U\left( A\right) \subset D_{1}^{A}$
fits in the framework considered in Subsection \ref%
{Subsection:definable-sets}. Naturally, this is the zeroset of the formula $%
\max \left\{ \left\Vert xx^{\ast }-1\right\Vert ,\left\Vert x^{\ast
}x-1\right\Vert \right\} $.\ The fact that such a formula is stable can be
verified by using the \emph{polar decomposition }of operators \cite[Theorem
3.2.17]{pedersen_analysis_1989}. Indeed, if $A\subset B\left( H\right) $ is
a C*-algebra, and $a\in A$, then one can write $a=u\left( a^{\ast }a\right)
^{1/2}$ where $u\in B\left( H\right) $, $u^{\ast }u$ is the orthogonal
projection onto the orthogonal complement of the kernel of $\left( a^{\ast
}a\right) ^{1/2}$, and $uu^{\ast }$ is the orthogonal projection onto the
orthogonal complement of the kernel of $\left( aa^{\ast }\right) ^{1/2}$.
Now, if $\left\Vert a^{\ast }a-1\right\Vert \leq \delta $ and $\left\Vert
aa^{\ast }-1\right\Vert \leq \delta $ for some small enough $\delta $, $%
\left( a^{\ast }a\right) ^{1/2}$ and $\left( aa^{\ast }\right) ^{1/2}$ are
invertible, which forces $u$ to be a unitary and to belong to $A$. Finally, 
\begin{equation*}
\left\Vert u-a\right\Vert \leq \left\Vert 1-\left( a^{\ast }a\right)
^{1/2}\right\Vert
\end{equation*}%
which can be made arbitrarily small by choosing $\delta $ small enough. A
similar argument shows that the set of \emph{isometries }(i.e.\ elements $v$
of $A$ satisfying $v^{\ast }v=1$) is also definable.

\subsubsection{Positive contractions}

A similar argument allows one to conclude that the set $A_{+}^{1}$ of
positive contractions in $A$ is a definable set. Recall that an element $a$
of $A$ is \emph{positive }if it is a positive operator in any faithful
representation $A\subset B\left( H\right) $. This is equivalent to the
assertion that $a$ is of the form $b^{\ast }b$ for some $b\in A$. A positive
contraction is just an positive element of norm at most $1$. Therefore the
set of positive contraction is the zeroset of the formula $\inf_{y\in
D^{1}}\left\Vert x-y^{\ast }y\right\Vert $, which is obviously stable. In a
similar way, considering the stable formula $\left\Vert x^{\ast
}-x\right\Vert $, shows that the set $A_{\mathrm{sa}}^{1}$ of selfadjoint
elements of norm at most $1$ is definable as well.

\subsubsection{Projections}

We now consider the set of \emph{projections }in a C*-algebra $A$.\ These
are the selfadjoint elements $p$ of $A$ satisfying $p^{2}=p$. By definition,
this is the zeroset of the formula $\varphi \left( x\right) $ given by $\max
\left\{ \left\Vert x^{\ast }-x\right\Vert ,\left\Vert x^{2}-x\right\Vert
\right\} $. To see that such a formula is stable, suppose that $a\in A$ is
an element of norm at most $1$ satisfying $\left\Vert a^{\ast }-a\right\Vert
\leq \delta $ and $\left\Vert a^{2}-a\right\Vert \leq \delta $ for some $%
\delta \in \left( 0,1/2\right) $. Consider $a_{0}:=\left( a+a^{\ast }\right)
/2$ and observe that $\left\Vert a_{0}\right\Vert \leq 1$, $\left\Vert
a-a_{0}\right\Vert <\delta /2$ and hence $\left\Vert
a_{0}^{2}-a_{0}\right\Vert <2\delta $. This implies that the spectrum of $%
a_{0}$ is contained in $\left( -\varepsilon ,\varepsilon \right) \cup
(1-\varepsilon ,1]$ where $\varepsilon =\frac{1-\sqrt{1-8\delta }}{2}\leq
\delta $. Therefore the function 
\begin{equation*}
f(t)=\left\{ 
\begin{array}{cc}
0 & \text{if }t\in (-\delta ,\delta )\text{,} \\ 
1 & \text{if }t\in (1-\delta ,1]%
\end{array}%
\right.
\end{equation*}%
is continuous on the spectrum of $a$. By continuous functional calculus, one
can consider the element $p:=f\left( a_{1}\right) $ of $A$. Since the
spectrum of $p$ is the range of $f$, it is contained in $\left\{ 0,1\right\} 
$, which implies that $p$ is a projection. Furthermore $\left\Vert
p-a_{1}\right\Vert $ is equal to the supremum of $\left\vert
f(t)-t\right\vert $ where $t$ ranges in the spectrum of $a$, which is at
most $\delta $. In conclusion $\left\Vert p-a_{0}\right\Vert \leq \left\Vert
p-a_{1}\right\Vert +\left\Vert a_{1}-a_{0}\right\Vert <\delta +2\delta \leq
3\delta $. In conclusion, this shows that if $a$ is an element of the unit
ball of a C*-algebra $A$ satisfying $\varphi ^{A}\left( a\right) <\delta $,
then there exists $p$ in the unit ball of $A$ such that $\varphi ^{A}\left(
p\right) =0$ and $\left\Vert p-a\right\Vert <3\delta $.

In a similar fashion, one can show that the set of $n$-tuples $\left(
p_{1},\ldots ,p_{n}\right) $ of \emph{pairwise orthogonal} projections of $A$
is definable. (Two projections $p,q$ are orthogonal if $pq=qp=0$.) Let us
consider for simplicity the case $n=2$. Let $\varphi \left( x,y\right) $ be
the formula%
\begin{equation*}
\max \left\{ \left\Vert x^{\ast }-x\right\Vert ,\left\Vert
x^{2}-x\right\Vert ,\left\Vert y-y^{\ast }\right\Vert ,\left\Vert
y^{2}-y\right\Vert ,\left\Vert xy\right\Vert ,\left\Vert yx\right\Vert
\right\} \text{.}
\end{equation*}%
Clearly, the zeroset of $\varphi $ is the set of pairs of orthogonal
projections. In order to see that $\varphi $ is stable, fix $\varepsilon >0$%
, and suppose that $a,b$ are elements of the unit ball of a C*-algebra $A$
satisfying $\varphi ^{A}\left( a,b\right) <\delta <\varepsilon /2$ for some
small enough $\delta $. By stability of the formula defining projections, we
can assume that $a,b$ are already projections. Then one can set $p=a$ and
consider $b_{0}:=\left( 1-p\right) b\left( 1-p\right) \in \left( 1-p\right)
A\left( 1-p\right) $. Since $\left\Vert ab\right\Vert <\delta $ and $%
\left\Vert ba\right\Vert <\delta $ one has that $\left\Vert
b_{0}-b\right\Vert <2\delta $. Again by stability of the formula defining
projections applied to the C*-algebra $\left( 1-p\right) A\left( 1-p\right) $%
, for $\delta $ small enough there exists a projection $q\in \left(
1-p\right) A\left( 1-p\right) $ such that $\left\Vert q-b_{0}\right\Vert
<\varepsilon /2$ and hence $\left\Vert q-b\right\Vert <\varepsilon $.
Observing that $q\in A$ is orthogonal to $p$ concludes the proof that $%
\varphi $ is stable.

\subsubsection{Partial isometries\label{Partial isometries}}

Recall that a \emph{partial isometry }in a C*-algebra $A$ is an element $%
v\in A$ such that $v^{\ast }v$ is a projection, called the support
projection of $v$.\ This implies that $vv^{\ast }$ is also a projection,
called the range projection of $v$. We want to show that the set of partial
isometries is definable. The following lemma can be extracted from the
classical paper of Glimm classifying UHF algebras; see \cite[Lemma 1.9]%
{glimm_certain_1960}.

\begin{lemma}
Suppose that $\delta \in \left( 0,1/40\right) $ and $A$ is a C*-algebra. Fix
a concrete representation $A\subset B\left( H\right) $ of $A$ as an algebra
of operators on a Hilbert space $H$. Suppose that $p,q\in A$ are
projections, and $\tilde{p},\tilde{q}\in B\left( H\right) $ are projections
such that $\left\Vert p-\tilde{p}\right\Vert <\delta $, $\left\Vert q-\tilde{%
q}\right\Vert <\delta $. Assume that there exist a partial isometry $\tilde{v%
}\in B\left( H\right) $ and $a\in A$ such that $\tilde{v}^{\ast }\tilde{v}=%
\tilde{p}$, $\tilde{v}\tilde{v}^{\ast }=\tilde{q}$, $\left\Vert a-\tilde{v}%
\right\Vert <\delta $, and $\left\Vert a\right\Vert \leq 1$. Then there
exists a partial isometry $v\in A$ such that $v^{\ast }v=p$, $vv^{\ast }=q$,
and $\left\Vert v-\tilde{v}\right\Vert <30\delta $.
\end{lemma}

We can use this lemma to show that the set of \emph{partial isometries }of $%
A $ is definable.\ The set of partial isometries is the zeroset of the
formula $\max \{\left\Vert x^{\ast }x-\left( x^{\ast }x\right)
^{2}\right\Vert ,\left\Vert xx^{\ast }-\left( xx^{\ast }\right)
^{2}\right\Vert \}$. We claim that such a formula is stable. Suppose that $%
a\in A\subset B\left( H\right) $ satisfies $\left\Vert a^{\ast }a-\left(
a^{\ast }a\right) ^{2}\right\Vert \leq \delta \leq 1/1600$. By functional
calculus this implies that $\left\Vert \left( a^{\ast }a\right)
^{1/2}-\left( a^{\ast }a\right) \right\Vert \leq \delta $. Using the polar
decomposition of operators, we can write $a=\tilde{v}\left( a^{\ast
}a\right) ^{1/2}$, where $\tilde{v}\in B\left( H\right) $ is a partial
isometry with support projection the orthogonal projection $\tilde{p}$ onto $%
\mathrm{\mathrm{Ker}}\left( a\right) ^{\bot }=\mathrm{\mathrm{Ker}}\left(
a^{\ast }a\right) ^{\bot }=\overline{\mathrm{Ran}(a^{\ast }a)}$ and range
projection the orthogonal projection onto $\mathrm{Ker}\left( a^{\ast
}\right) ^{\bot }=\mathrm{\mathrm{Ker}}\left( aa^{\ast }\right) ^{\bot }=%
\overline{\mathrm{Ran}\left( aa^{\ast }\right) }$. (For an operator $T\in
B\left( H\right) $ we let $\mathrm{\mathrm{Ker}}\left( T\right) $ be the
kernel $\left\{ \xi \in H:T\xi =0\right\} $ of $T$, and $\overline{\mathrm{%
Ran}\left( T\right) }$ the closure of the range $\left\{ T\xi :\xi \in
H\right\} $ of $T$.) Now for $\xi \in \mathrm{Ran}(a^{\ast }a)$ we have that 
$\xi =\left( a^{\ast }a\right) \eta $ for some $\eta \in H$ and hence%
\begin{equation*}
\left\Vert \left( a^{\ast }a\right) \xi -\xi \right\Vert =\left\Vert \left(
a^{\ast }a\right) ^{2}\eta -\left( a^{\ast }a\right) \eta \right\Vert \leq
\left\Vert \left( a^{\ast }a\right) ^{2}-\left( a^{\ast }a\right)
\right\Vert \leq \delta \text{.}
\end{equation*}%
This shows that $\left\Vert a^{\ast }a-\tilde{p}\right\Vert \leq \delta $.
Similarly, one has that $\left\Vert aa^{\ast }-\tilde{q}\right\Vert \leq
\delta $. Thus 
\begin{equation*}
\left\Vert a-\tilde{v}\right\Vert =\left\Vert \tilde{v}\left( a^{\ast
}a\right) ^{1/2}-\tilde{v}\right\Vert \leq \left\Vert \left( a^{\ast
}a\right) ^{1/2}-p\right\Vert \leq \sqrt{\delta }\text{.}
\end{equation*}%
Furthermore by stability of formula defining projections we have that there
exist projections $p,q\in A$ such that $\left\Vert p-\left( a^{\ast
}a\right) ^{1/2}\right\Vert \leq 3\delta $ and $\left\Vert q-\left( aa^{\ast
}\right) ^{1/2}\right\Vert \leq 3\delta $. Therefore $\left\Vert p-\tilde{p}%
\right\Vert \leq 4\delta \leq \sqrt{\delta }$ and $\left\Vert q-\tilde{q}%
\right\Vert \leq 4\delta \leq \sqrt{\delta }$. Therefore by Lemma there
exists a partial isometry $v\in A$ with support projection $p$ and range
projection $q$ such that $\left\Vert v-\tilde{v}\right\Vert \leq 30\sqrt{%
\delta }$.

\subsubsection{Murray--von Neumann equivalence}

Murray--von Neumann equivalence is an important relation among projections
in a C*-algebra, which is crucial for the definition of the $K_{0}$-group.
Recall that an element $v$ of a C*-algebra $A$ is a \emph{partial isometry}
if $v^{\ast }v$ is a projection (support projection), and $vv^{\ast }$ is a
projection (range projection). Two projections $p,q\in A$ are \emph{%
Murray--von Neumann equivalent}---in formulas $p\sim q$---if there exists a
partial isometry $v\in A$ such that $p$ is the support projection of $v$ and 
$q$ is the range projection of $v$. We recall that if $p,q\in A$ are
projections such that $\left\Vert p-q\right\Vert <1$ then $p\sim q$.

We want to show that the relation of Murray--von Neumann equivalence of
projections (as a set of pairs) is definable. Indeed, consider the definable
predicate $\varphi \left( x,y\right) $ given by%
\begin{equation*}
\max \left\{ \left\Vert x^{2}-x\right\Vert ,\left\Vert x-x^{\ast
}\right\Vert ,\left\Vert y^{2}-y\right\Vert ,\left\Vert y-y^{\ast
}\right\Vert ,\inf_{z\text{ partial isometry}}\max \left\{ \left\Vert
z^{\ast }z-x\right\Vert ,\left\Vert zz^{\ast }-y\right\Vert \right\}
\right\} \text{.}
\end{equation*}%
Then we claim that the zeroset of $\varphi $ in a C*-algebra $A$ is the
relation of Murray--von Neumann equivalence of projections. Indeed, if $%
\varphi ^{A}\left( p,q\right) =0$ then clearly $p,q$ are projections.
Furthermore, there exists a partial isometry $v\in A$ such that $\left\Vert
v^{\ast }v-p\right\Vert <1$ and $\left\Vert vv^{\ast }-q\right\Vert <1$.
This implies $p\sim v^{\ast }v\sim vv^{\ast }\sim q$. It follows easily from
stability of the formula defining projections that $\varphi $ is stable as
well. This shows that the relation of Murray--von Neumann equivalence of
projections is a definable set. Adding the clauses $\left\vert \left\Vert
x\right\Vert -1\right\vert $ and $\left\vert \left\Vert y\right\Vert
-1\right\vert $ shows that the relation of Murray--von Neumann equivalence
of \emph{nonzero} projections is a definable set. A similar argument shows
that the set of $n$-tuples of pairwise orthogonal and pairwise Murray--von
Neumann equivalent projections is definable.

\subsubsection{Infinite projections}

Suppose that $A$ is a C*-algebra. Recall that for projections $p,q\in A$ one
sets $p\leq q$ if $pq=qp=p$. Observe that if $p\leq q$ and $p\neq q$, then $%
\left\Vert p-q\right\Vert =1$. A nonzero projection $r\in A$ is called

\begin{itemize}
\item \emph{infinite }if there is a nonzero projection $r_{0}\leq r$ such
that $r_{0}\neq r$ and $r_{0}\sim r$,

\item \emph{properly infinite }if there exist nonzero orthogonal
Murray--von\ Neumann equivalent projections $p,q\in A$ such that $p+q=r$.
\end{itemize}

We claim that the set of infinite projections in a C*-algebra is definable.
As we have observed above, there is a definable predicate $\psi \left(
x,y\right) $ whose zeroset is the relation of Murray--von Neumann
equivalence of \emph{nonzero} projections. We can therefore consider the
definable predicate $\theta \left( x\right) $ given by%
\begin{equation*}
\inf_{y\in D_1}\max \left\{ \psi \left( x,y\right) ,\left\Vert
yx-y\right\Vert ,\left\Vert xy-y\right\Vert ,\left\vert 1-\left\Vert
x-y\right\Vert \right\vert \right\} \text{.}
\end{equation*}%
It is clear by the above remarks that the set of infinite projections is
contained in the zeroset of $R$. In order to show that $\theta $ is a \emph{%
stable }definable predicate whose zeroset is the set of infinite projections
it remains to show the following: for every $\varepsilon >0$ there exists $%
\delta >0$ such that for every C*-algebra $A$ and $p\in A$ satisfying $%
\theta ^{A}\left( p\right) <\delta $ there exists an infinite projection $%
p^{\prime }\in A$ such that $\left\Vert p-p^{\prime }\right\Vert
<\varepsilon $. Let us thus consider $p\in A$ such that $\theta ^{A}\left(
p\right) <\delta $. By stability of the predicate $p$, we can assume that $p$
is itself a nonzero projection. Furthermore, by stability of the predicate $%
Q $, we can assume, up to replacing $\delta $ with a smaller positive
number, that there exists a nonzero projection $q\in A$ such that $q\sim p$, 
$\left\Vert pq-q\right\Vert <\delta $, $\left\Vert qp-q\right\Vert <\delta $%
, and $\left\vert 1-\left\Vert p-q\right\Vert \right\vert <\delta $.
Consider now $pqp\in pAp$ and observe that, again by stability of the
predicate defining projections, and upon replacing $\delta $ with a smaller
positive number, we can find a projection $q^{\prime }\in pAp$ such that $%
\left\Vert q^{\prime }-q\right\Vert <\delta $. This guarantees that $%
pq^{\prime }=q^{\prime }p=q^{\prime }$ and $\left\vert 1-\left\Vert
p-q^{\prime }\right\Vert \right\vert <2\delta $. As long as $\delta <1/2$
this ensures that $q^{\prime }$ is a nonzero projection such that $q^{\prime
}\leq p$, $q^{\prime }\neq p$, and $q^{\prime }\sim q\sim p$. Therefore $p$
is itself infinite projection, concluding the proof that $\theta $ is stable.

We now claim that the set of properly infinite projections in a C*-algebra
is definable. Indeed, let now $\varphi \left( x\right) $ be the definable
predicate whose zeroset is the set of nonzero\emph{\ }projections. Then we
can consider, in view of definability of the relation of Murray--von Neumann
equivalence of \emph{orthogonal nonzero }projections, the definable
predicate $\eta \left( x\right) $ given by%
\begin{equation*}
\max \left\{ \varphi \left( x\right) ,\inf_{y,z\text{ nonozero orthogonal
projections, }y\sim z}\left\Vert x-(y+z)\right\Vert \right\} \text{.}
\end{equation*}%
As above, it is clear that the set of properly infinite projections is
contained in the zeroset of $S$. In order to show that $\eta $ is a stable
definable predicate whose zeroset is the set of properly infinite
projections, it remains to show the following: for every $\varepsilon >0$
there exists $\delta >0$ such that if $A$ is a C*-algebra and $r\in A$ is
such that $\eta ^{A}\left( r\right) <\delta $ then there exists a properly
infinite projection $r^{\prime }\in A$ such that $\left\Vert r-r^{\prime
}\right\Vert <\varepsilon $. Suppose then that $r\in A$ satisfies $\eta
^{A}\left( r\right) <\delta $. Again, by stability of $\varphi $ we can
suppose that $r$ is itself a nonzero projection. Furthermore there exist
orthogonal nonzero projections $p,q\in A$ such that $\left\Vert r-\left(
p+q\right) \right\Vert <\delta $. Thus $r^{\prime }:=p+q$ is a nonzero
purely infinite projection such that $\left\Vert r-r^{\prime }\right\Vert
<\delta $, concluding the proof.

\subsubsection{Scalars}

Recall that we are tacitly assuming C*-algebras to be unital. Thus, after
identifying a complex number with the corresponding scalar multiple of the
unit, we can identify $\mathbb{C}$ with a subalgebra of any given
C*-algebra. Below we let $x$ be a variable with corresponding domain $D_{1}$%
. Then we have that $\left[ 0,1\right] $ is definable, as the zeroset of the
stable formula $\left\Vert x-\left\Vert x\right\Vert 1\right\Vert $. Using
this fact and Proposition \ref{Proposition:quantify}, one can conclude that $%
\inf_{t_{0},t_{1}\in \left[ 0,1\right] }\left\Vert x-\left(
t_{0}-t_{1}\right) \right\Vert $ is also a definable predicate, which is
obviously stable. Its zeroset is $\left[ -1,1\right] $. Analogously, the set 
$\mathbb{D}=\left\{ \lambda \in \mathbb{C}:\left\vert \lambda \right\vert
\leq 1\right\} $ is the zeroset of the definable predicate $%
\inf_{t_{0},t_{1}\in \lbrack -1,1]}\left\Vert x-\left( t_{0}+it_{1}\right)
\right\Vert $, and the set $\mathbb{T}=\left\{ \lambda \in \mathbb{C}%
:\left\vert \lambda \right\vert =1\right\} $ is the zeroset of the definable
predicate $\max \{\left\Vert x^{\ast }x-1\right\Vert ,\left\Vert xx^{\ast
}-1\right\Vert ,\inf_{t_{0},t_{1}\in \lbrack -1,1]}\left\Vert x-\left(
t_{0}+it_{1}\right) \right\Vert \}$. Thus all these sets are definable.

\subsection{More axiomatizable classes of C*-algebras\label%
{Subsection:axiomatize-C*-more}}

At this point, we will use the possibility of quantifying over definable
sets guaranteed by Proposition \ref{Proposition:quantify} to show that
several other important classes of C*-algebras are axiomatizable.

\subsubsection{Stably finite C*-algebra}

Suppose that $A$ is a C*-algebra. An isometry $v$ in $A$ is an element of $A$
satisfying $v^{\ast }v=1$. A C*-algebra $A$ is \emph{finite }if every
isometry is a unitary, and \emph{stably finite }if $M_{n}\left( A\right) $
is finite for every $n\in \mathbb{N}$. A similar proof as in the case of
partial isometries---see \S \ref{Partial isometries}---shows that the set of
isometries is definable. Therefore the class of finite C*-algebra is
elementary, as witnessed by the axiom%
\begin{equation*}
\sup_{v\text{ isometry}}\left\Vert vv^{\ast }-1\right\Vert \leq 0\text{.}
\end{equation*}%
Since the language of C*-algebras contains relation symbols for the norm in $%
M_{n}\left( A\right) $ for $n\in \mathbb{N}$, one can similarly conclude
that the set of isometries in $M_{n}\left( A\right) $ is definable.
Henceforth, the same argument shows that the class of stably finite
C*-algebras is axiomatizable as well.

\subsubsection{Real rank zero C*-algebras}

By definition, a C*-algebra has \emph{real rank zero }if the set of
selfadjoint elements with finite spectrum is dense in the set of all
selfadjoint elements. While this definition does not make it apparent that
this is a definable property, we can consider the following useful
equivalent characterization. A C*-algebra has real rank zero if and only if
for every pair of positive elements $a,b$ in $A$ of norm at most $1$ and
every $\varepsilon >0$ there exists a projection $p\in A$ such that $\max
\left\{ \left\Vert pa\right\Vert ,\left\Vert \left( 1-p\right) b\right\Vert
\right\} <\left\Vert ab\right\Vert ^{1/2}+\varepsilon $ \cite[Theorem 2.6]%
{brown_c*-algebras_1991}. In other words, a C*-algebra has real rank zero if
and only if it satisfies the condition%
\begin{equation*}
\sup_{\substack{ x,y\in D_{1}  \\ x,y\text{ positive}}}\inf_{\substack{ z\in
D_{1}  \\ z\text{ projection}}}\max \left\{ \left\Vert zx\right\Vert
,\left\Vert \left( 1-z\right) y\right\Vert \right\} -\left\Vert
xy\right\Vert ^{1/2}\leq 0\text{.}
\end{equation*}%
By Proposition \ref{Proposition:quantify} and the fact that the sets of
projections and positive contractions are definable, this witnesses that
real rank zero C*-algebras form an axiomatizable class.

\subsubsection{Purely infinite simple C*-algebras}

The notion of \emph{purely infinite} C*-algebra can be defined in terms of
the notion of Cuntz equivalence of positive elements. This is an equivalence
relation for positive elements in a C*-algebra, generalizing the relation of
Murray--von Neumann equivalence for projections. Suppose that $A$ is a
C*-algebra, and $a,b$ are positive elements of $A$. Then one sets $a\precsim
b$ if there exists a sequence $\left( x_{n}\right) $ in $A$ such that $%
x_{n}^{\ast }bx_{n}\rightarrow a$ for $n\rightarrow +\infty $. Then $a$ and $%
b$ are \emph{Cuntz equivalent}, in formulas $a\sim b$, if $a\precsim b$ and $%
b\precsim a$. It can be shown that this is indeed an equivalence relation
which, in the case of projections, coincides with Murray-von Neumann
equivalence.

For a selfadjoint element $a$ of a C*-algebra $A$ we let $a_{+}$ be the 
\emph{positive part }of $A$. In other words, $a_{+}$ is the element $f\left(
a\right) $ of $A$ where $f:\mathbb{R}\rightarrow \mathbb{R}$ is the function%
\begin{equation*}
f\left( t\right) =\left\{ 
\begin{array}{cc}
0 & \text{if }t\leq 0 \\ 
t & \text{otherwise.}%
\end{array}%
\right.
\end{equation*}%
The following lemma is proved in \cite[Lemma 2.4]{kirchberg_infinite_2002}.

\begin{lemma}
\label{Lemma:subequivalence}If $a,b$ are positive elements of a C*-algebra $%
A $ of norm at most $1$, and $n\in \mathbb{N}$ is such that $a\precsim (b-%
\frac{1}{n^{2}})_{+}$ then there exists $c\in A$ such that $\left\Vert
c\right\Vert \leq 2n$ and $a=c^{\ast }bc$.
\end{lemma}

A nontrivial C*-algebra $A$ is \emph{purely infinite }if it has no
nontrivial abelian quotients, and whenever $a,b$ are nonzero positive
elements of $A$ such that $b$ belongs to the closed two-sided ideal
generated by $a$, then $a\precsim b$ \cite[Proposition 4.1.1]%
{rordam_classification_2002}. A C*-algebra is \emph{simple }if it contains
no nontrivial closed two-sided ideals. For simple C*-algebras, being purely
infinite is equivalent to the assertion that $A$ is nontrivial, and whenever 
$a,b$ are nonzero positive elements of $A$, $a\sim b$. Furthermore, for
nontrivial simple C*-algebras being purely infinite is equivalent to the
assertion that $A$ has real rank zero, and any nonzero projection in $A$ is
properly infinite.

We have already seen that the class of real rank zero C*-algebras is
axiomatizable. Furthermore, the sets of nonzero projections and of properly
infinite projections are both definable, which easily implies that the class
of C*-algebras with the property that any nonzero projection is properly
infinite, is axiomatizable. Let now $\mathcal{C}$ be the axiomatizable class
of real rank C*-algebras with the property that any nonzero projection is
properly infinite. In view of the remarks above, for a C*-algebra $A$ in $%
\mathcal{C}$ being simple (and hence purely infinite) is equivalent to the
assertion that any two nonzero positive elements of $A$ are Cuntz
equivalent. Since any nonzero positive element $a$ of $A$ is Cuntz
equivalent to any nonzero positive scalar multiple of $A$, this is in turn
equivalent to the assertion that for any two positive elements $a,b$ of $A$
satisfying $\left\Vert a\right\Vert >1/2$ and $\left\Vert b\right\Vert >1/2$
one has that $a\sim b$. By Lemma \ref{Lemma:subequivalence} applied in the
case when $n=2$ this is in turn equivalence to the assertion that if $a,b$
are positive elements of $A$ then for every $\varepsilon >0$ there exists $%
c\in A$ such that $\left\Vert c\right\Vert \leq 4$ and $\left\Vert a-c^{\ast
}bc\right\Vert <\varepsilon $. This condition is clearly axiomatized by the
condition%
\begin{equation*}
\sup_{x,y\in D_{1}}\min \left\{ \left\Vert x\right\Vert -1/2,\left\Vert
y\right\Vert -1/2,\inf_{z\in D_{1}}\left\Vert a-\left( 4z\right) ^{\ast
}b\left( 4z\right) \right\Vert \right\} \leq 0\text{.}
\end{equation*}%
This shows that the class of purely infinite simple C*-algebras is
axiomatizable.

\section{Ultraproducts and ultrapowers\label{Section:ultra}}

\subsection{Ultraproducts in the logic for metric structures}

Ultraproducts and ultrapowers are a fundamental construction in model
theory, both in its discrete version and its generalization for metric
structures. We introduce this construction in the general setting of an
arbitrary language $L$. We work in the setting of languages with domains of
quantification introduced in Subsection \ref{Subsection:domains}.

Let then $I$ be a set, to be considered as an \emph{index set}. (The reader
can consider the case when $I=\mathbb{N}$, for simplicity.) An \emph{%
ultrafilter }$\mathcal{U}$ over $I$ is a nonempty collection of subsets of $%
I $ with the property that $\varnothing \notin \mathcal{U}$, if $A,B\in 
\mathcal{U}$ then $A\cap B\in \mathcal{U}$, and for every $A\in I$ either $%
A\in \mathcal{U}$ or $I\setminus A\in \mathcal{U}$. One should consider $%
\mathcal{U}$ as a notion of largeness, where a set $A\subset I$ is large if
it belongs to $\mathcal{U}$. In the spirit of this interpretation, $\mathcal{%
U}$ can be thought of as a finitely-additive $\left\{ 0,1\right\} $-valued
measure on $I$ which is defined for arbitrary subsets of $A$ by 
\begin{equation*}
A\mapsto \left\{ 
\begin{array}{cc}
1 & \text{if }A\in \mathcal{U}\text{,} \\ 
0 & \text{if }A\notin \mathcal{U}\text{.}%
\end{array}%
\right.
\end{equation*}%
Conversely, any finitely-additive $\left\{ 0,1\right\} $-valued measure on $%
I $ which is defined on all subsets of $I$ arises from an ultrafilter in
this fashion. An ultrafilter $\mathcal{U}$ over $I$ is \emph{principal} if
it contains a finite set (whence it contains a singleton), and \emph{%
nonprincipal} otherwise. A stronger properly than being nonprincipal is
being \emph{countably incomplete}, which means that it contains a sequence $%
\left( X_{n}\right) $ of elements with empty intersection. When $I$ is
countable, these two notions coincide.

Fix an ultrafilter $\mathcal{U}$ over $I$. Consistently with the
interpretation of ultrafilters as notions of largeness, following \cite%
{todorcevic_introduction_2010} we introduce the notation of \emph{%
ultrafilter quantifiers}. Let $P$ be a property that elements of $I$ may or
may not have. Then we write $\left( \mathcal{U}i\right) P\left( i\right) $
if the set of elements $i$ of $I$ for which $P$ holds belongs to $\mathcal{U}
$.

Suppose now that $f$ is a continuous function $f$ from $I$ to a compact
Hausdorff space $X$. One can then define the limit $\lim_{i\rightarrow 
\mathcal{U}}f\left( i\right) \in X$. This is the unique element $t$ of $X$
such that, for every neighborhood $U$ of $t$ in $X$, $\left( \mathcal{U}%
i\right) f\left( i\right) \in U$. It is clear that, since $X$ is Hausdorff,
there exists at most one such an element of $X$. In order to see that such
an element of $X$ exists, consider the collection $\mathcal{F}$ of nonempty
closed subsets of the form%
\begin{equation*}
\overline{\left\{ f\left( i\right) :i\in A\right\} }
\end{equation*}%
for $A\in \mathcal{U}$. Then $\mathcal{F}$ satisfies the finite intersection
property, and by compactness of $X$ one has that $\bigcap \mathcal{F}$ is
nonempty. If $x$ is an element of $\bigcap \mathcal{F}$ and $U$ is an open
neighborhood of $x$, then we claim that $\left( \mathcal{U}i\right) f\left(
i\right) \in U$. Indeed, if this is not the case, then $\left( \mathcal{U}%
i\right) f\left( i\right) \notin U$. Thus there exists $A\in \mathcal{U}$
such that $\overline{\left\{ f\left( i\right) :i\in A\right\} }\subset
X\setminus U$. Since $\overline{\left\{ f\left( i\right) :i\in A\right\} }%
\in \mathcal{F}$, this contradicts the fact that $x\in \bigcap \mathcal{F}$.

Suppose now that $\left( M_{i}\right) _{i\in I}$ is a family of $L$%
-structures indexed by $I$. The ultraproduct $\prod_{\mathcal{U}}M_{i}$ is
the $L$-structure $M$ defined as follows. For every domain $D$ in the
language $L$, consider the product $\prod_{i\in I}D^{M_{i}}$. This is
naturally endowed with a pseudometric given by $d\left( \boldsymbol{a},%
\boldsymbol{b}\right) =\lim_{i\rightarrow \mathcal{U}}d\left(
a_{i},b_{i}\right) $. Here and in the following, we denote by $\boldsymbol{a}
$ an $I$-sequence $\left( a_{i}\right) _{i\in I}$ with $a_{i}\in M_{i}$ for $%
i\in I$. Then one can define $M^{D}$ to be the metric space obtained from
such a pseudometric. If $\boldsymbol{a}$ is an element of $\prod_{i\in
I}D^{M_{i}}$, then we let $\left[ \boldsymbol{a}\right] $ be the
corresponding element of $D^{M}$. If $D_{0},D_{1}$ are domains such that $%
D_{0}\leq D_{1}$, then by definition of structure one has that $%
D_{0}^{M_{i}}\subset D_{1}^{M_{i}}$ for every $i\in I$. Thus one can
canonically identify isometrically $D_{0}^{M}$ with a subspace of $D_{1}^{M}$%
. Since the collection $\mathcal{D}$ of domains is directed, the union $%
\bigcup_{D\in \mathcal{D}}D^{M}$ is itself a metric space, and we let $M$ to
be the completion of such a metric space.

We now define the interpretation of function symbols in $M$. Suppose that $f$
is an $n$-ary function symbol in $L$ and $D_{1},\ldots ,D_{n}$ are domains.
Then one can consider the corresponding output domain $D=D_{D_{1},\ldots
,D_{n}}^{f}$ and the continuity modulus $\varpi =\varpi _{D_{1},\ldots
,D_{n}}^{f}$ as prescribed by the language $L$. Then, for every $i\in I$,
the interpretation of $f$ in $M_{i}$ gives a function $%
f^{M_{i}}:D_{1}^{M_{i}}\times \cdots \times D_{n}^{M_{i}}\rightarrow
D^{M_{i}}$ which is uniformly continuous with modulus $\varpi $. Therefore
one can define $f^{M}:D_{1}^{M}\times \cdots \times D_{n}^{M}\rightarrow
D^{M}$ by setting%
\begin{equation*}
f^{M}\left( \left[ \boldsymbol{a}^{(1)}\right] ,\ldots ,\left[ \boldsymbol{a}%
^{(n)}\right] \right) =[(f^{M_{i}}(a_{i}^{(1)},\ldots ,a_{i}^{(n)}))_{i\in
I}]\text{.}
\end{equation*}%
This is again a uniformly continuous function with modulus $\varpi $.
Letting $D_{1},\ldots ,D_{n}$ range among all the domains in the language $L$
one can then define, by extending it to the completion, the function $%
f^{M}:M^{n}\rightarrow M$, which is still uniformly continuous with modulus $%
\varpi $.

Suppose now that $R$ is an $n$-ary relation symbol in $L$ and $D_{1},\ldots
,D_{n}$ are domains. Then one can consider the corresponding output domain $%
D=D_{D_{1},\ldots ,D_{n}}^{f}$, the continuity modulus $\varpi =\varpi
_{D_{1},\ldots ,D_{n}}^{f}$, and the bound $J=J_{D_{1},\ldots ,D_{n}}^{R}$
as prescribed by the language $L$. Then for every $i\in I$, the
interpretation of $R$ in $M_{i}$ gives a function $R^{M_{i}}:D_{1}^{M_{i}}%
\times \cdots \times D_{n}^{M_{i}}\rightarrow J$ which is uniformly
continuous with modulus $\varpi $. Therefore one can define $%
R^{M}:D_{1}^{M}\times \cdots \times D_{n}^{M}\rightarrow J$ by setting%
\begin{equation*}
R^{M}\left( \left[ \boldsymbol{a}^{(1)}\right] ,\ldots ,\left[ \boldsymbol{a}%
^{(n)}\right] \right) =\lim_{i\rightarrow \mathcal{U}}R^{M_{i}}(a_{i}^{(1)},%
\ldots ,a_{i}^{(n)})\in J\text{.}
\end{equation*}%
This is again a uniformly continuous function with modulus $\varpi $.
Letting $D_{1},\ldots ,D_{n}$ range among all the domains in the language $L$
one can then define, by extending it to the completion, the function $%
R^{M}:M^{n}\rightarrow \mathbb{R}$, which is still uniformly continuous with
modulus $\varpi $.

\subsection{\L os' theorem}

\L os' theorem is the fundamental result in model theory that relates the
construction of ultraproducts with notion of formulas. Let us adopt the
notation of the previous section.

Assume that $t$ is an $L$-term with variables within $x_{1},\ldots ,x_{n}$
which have $D_{1},\ldots ,D_{n}$ as corresponding domains. Then one can
easily see by induction on the complexity of $t$ that one can define an
output domain $D=D_{D_{1},\ldots ,D_{n}}^{t}$ and a continuity modulus $%
\varpi =\varpi _{D_{1},\ldots ,D_{n}}^{t}$ in terms of the output domains
and continuity moduli of the function symbols in $L$, such that for any $L$%
-structure $N$ the interpretation $t^{N}$ of $t$ in $N$ is a function $%
t^{N}:D_{1}^{N}\times \cdots \times D_{n}^{N}\rightarrow D^{N}$ with
continuity modulus $\varpi $. In particular, this guarantees that, if $M_{%
\mathcal{U}}$ is the ultraproduct $\prod_{\mathcal{U}}M_{i}$, then the
function $D_{1}^{M_{\mathcal{U}}}\times \cdots \times D_{n}^{M_{\mathcal{U}%
}}\rightarrow D^{M_{\mathcal{U}}}$, 
\begin{equation*}
\left( \left[ \boldsymbol{a}^{(1)}\right] ,\ldots ,\left[ \boldsymbol{a}%
^{(n)}\right] \right) \mapsto \lbrack t^{M_{i}}(a_{i}^{(1)},\ldots
,a_{i}^{(n)})]
\end{equation*}%
is a well-defined function with continuity modulus $\varpi $. Furthermore,
it is also easy to show by induction on the complexity of the term $t$, and
using the definition of the interpretation of function symbols in $M_{%
\mathcal{U}}$, that such a function coincides with the interpretation $t^{M_{%
\mathcal{U}}}$ of the term $t$ in $M_{\mathcal{U}}$.

Suppose now that $\varphi $ is an $L$-formula with free variables within $%
x_{1},\ldots ,x_{n}$. Again, one can show by induction on the complexity of $%
\varphi $ that one can define a bound $J=J_{D_{1},\ldots ,D_{n}}^{\varphi }$
and a continuity modulus $\varpi =\varpi _{D_{1},\ldots ,D_{n}}^{\varphi }$%
---in terms of the bounds, output domains, and continuity moduli of the
terms, connectives, and relation symbols that appear in $\varphi $---such
that for any $L$-structure $M$ the interpretation $\varphi ^{M}$ of $\varphi 
$ in $M$ is a function $\varphi ^{M}:D_{1}^{M}\times \cdots \times
D_{n}^{M}\rightarrow J$ with continuity modulus $\varpi $. Again, this
guarantees that, if $M_{\mathcal{U}}$ is the ultraproduct $\prod_{\mathcal{U}%
}M_{i}$, the function $D_{1}^{M_{\mathcal{U}}}\times \cdots \times D_{n}^{M_{%
\mathcal{U}}}\rightarrow J$, 
\begin{equation*}
\left( \left[ \boldsymbol{a}^{(1)}\right] ,\ldots ,\left[ \boldsymbol{a}%
^{(n)}\right] \right) \mapsto \lim_{i\rightarrow \mathcal{U}}\varphi
^{M_{i}}(a_{i}^{(1)},\ldots ,a_{i}^{(n)})]
\end{equation*}%
is well defined and uniformly continuous with modulus $\varpi $.
Furthermore, an induction on the complexity of $\varphi $ shows that such a
function coincides with the interpretation of $\varphi $ in $M_{\mathcal{U}}$%
. Summarizing, we have the following statement, which is the content of \L %
os' theorem.

\begin{theorem}
\label{Theorem:Los}Let $L$ be a language, and $\mathcal{U}$ be an
ultrafilter on a set $I$. Fix an $I$-sequence $\left( M_{i}\right) _{i\in I}$
of $L$-structures, and denote by $M_{\mathcal{U}}$ their ultraproduct $%
\prod_{\mathcal{U}}M_{i}$. Then for any $L$-formula $\varphi \left(
x_{1},\ldots ,x_{n}\right) $ with free variables within $x_{1},\ldots ,x_{n}$
with domains $D_{1},\ldots ,D_{n}$, and for every $\left( \left[ \boldsymbol{%
a}^{(1)}\right] ,\ldots ,\left[ \boldsymbol{a}^{(n)}\right] \right) \in
D_{1}^{M_{\mathcal{U}}}\times \cdots \times D_{n}^{M_{\mathcal{U}}}$ one has
that%
\begin{equation*}
\varphi ^{M_{\mathcal{U}}}\left( \left[ \boldsymbol{a}^{(1)}\right] ,\ldots ,%
\left[ \boldsymbol{a}^{(n)}\right] \right) =\lim_{i\rightarrow \mathcal{U}%
}\varphi ^{M_{i}}(a^{(1)},\ldots ,a_{i}^{(n)})\text{.}
\end{equation*}%
In particular, if $\varphi $ is a sentence, then%
\begin{equation*}
\varphi ^{M_{\mathcal{U}}}=\lim_{i\rightarrow \mathcal{U}}\varphi ^{M_{i}}%
\text{.}
\end{equation*}
\end{theorem}

It follows from \L os' theorem that and axiomatizable class of C*-algebras
is closed under ultraproducts. From Theorem \ref{Theorem:Los} one can deduce
that the same conclusions hold for definable predicates rather than
formulas. Using this fact, one can reformulate semantically the assertion
that a predicate $P$ is stable in terms of ultraproducts, as follows:

\begin{proposition}
\label{Proposition:characterize-stable}Suppose that $P\left( \bar{x}\right) $
is a definable predicate for the elementary class of $L$-structures $%
\mathcal{C}$. Then the following assertions are equivalent:

\begin{enumerate}
\item $P\left( \bar{x}\right) $ is stable;

\item for any sequence $\left( M_{n}\right) _{n\in \mathbb{N}}$ of
structures in $\mathcal{C}$, for any nonprincipal ultrafilter $\mathcal{U}$
over $\mathbb{N}$, any tuple $\bar{a}$ in $M:=\prod_{\mathcal{U}}M_{n}$
satisfying the condition $P\left( \bar{x}\right) =0$ admits a representative
sequence $\left( \bar{a}^{(n)}\right) _{n\in \mathbb{N}}$ of tuples $\bar{a}%
^{(n)}$ in $M_{n}$ such that $\left( \mathcal{U}n\right) $, $P^{M_{n}}\left( 
\bar{a}^{(n)}\right) =0$.
\end{enumerate}
\end{proposition}

\begin{proof}
We prove that (1) implies (2). Suppose that $P\left( \bar{x}\right) $ is
stable. Thus for every $m\in \mathbb{N}$ there exists $\delta _{m}\in
(0,2^{-m}]$ such that, if $M$ is a structure in $\mathcal{C}$ and $\overline{%
b}$ is a tuple in $M$ satisfying the condition $\left\vert P\left( \bar{x}%
\right) \right\vert \leq \delta _{m}$, then there is a tuple $\bar{a}$ in $M$
such that $d\left( \bar{a},\overline{b}\right) \leq 2^{-m}$, and $\bar{a}$
satisfies the condition $P\left( \bar{x}\right) =0$.

Suppose that $\mathcal{U}$, $\left( M_{n}\right) _{n\in \mathbb{N}}$, and $%
\overline{\boldsymbol{a}}$ are as in (2). Fix any representative sequence $(%
\bar{b}^{(n)})_{n\in \mathbb{N}}$ of $\overline{\boldsymbol{a}}$. Then, for
every $m\in \mathbb{N}$, $\left( \mathcal{U}n\right) $, $\bar{b}^{(n)}$
satisfies the condition $\left\vert P\left( \bar{x}\right) \right\vert \leq
\delta _{m}$. Thus, for every $m\in \mathbb{N}$, the set%
\begin{equation*}
J_{m}:=\left\{ n\in \mathbb{N}:n\geq m\text{ and }\bar{b}^{(n)}\text{
satisfies the condition }\left\vert P\left( \bar{x}\right) \right\vert \leq
\delta _{m}\right\}
\end{equation*}%
belongs to $\mathcal{U}$. Since $\bigcap_{m\in \mathbb{N}}J_{m}=\varnothing $%
, for every $n\in \mathbb{N}$ there exists a largest $m\left( n\right) \in 
\mathbb{N}$ such that $i\in J_{m\left( n\right) }$, if it exists. Otherwise,
we set $m\left( n\right) =0$.

For $n\in \mathbb{N}$ such that $m\left( n\right) >0$ define $\bar{a}^{(n)}$
to be a tuple in $M_{n}$ satisfying the conditions $d(x,\bar{b}^{(n)})\leq
2^{-m}$ and $P\left( \bar{x}\right) =0$. Observe that such a tuple exists by
the choice of $\delta _{m}$ and the definition of $m\left( n\right) $. If $%
m\left( n\right) =0$ define $\bar{a}^{(n)}$ arbitrarily. If $m\in \mathbb{N}$%
, then we have that the set of $n\in \mathbb{N}$ such that $d(a^{(n)},\bar{b}%
^{(n)})\leq 2^{-m}$ and $P^{M_{n}}(\bar{a}^{(n)})=0$ contains $J_{m}$, which
belongs to $\mathcal{U}$. Therefore $\left( \bar{a}^{(n)}\right) $ is a
representative sequence for $\overline{\boldsymbol{a}}$. This concludes the
proof that (1) implies (2). The converse implication can be easily proved
reasoning by contradiction.
\end{proof}

\subsection{Ultraproducts of C*-algebras}

The general notion of ultraproduct in the logic for metric structures
recovers the usual notion of ultraproduct of C*-algebras, when these are
considered as structures in the language $L^{\text{C*}}$ introduced in
Subsection \ref{Subsection:C*-language}. Explicitly, suppose that $\mathcal{U%
}$ is an ultrafilter over a set $I$, and $\left( A_{i}\right) _{i\in I}$ is
an $I$-sequence of C*-algebras. Then one can let $\ell ^{\infty }\left(
A_{i}\right) _{i\in I}$ be the C*-algebra consisting of all bounded
sequences $\boldsymbol{a}\in \prod_{i\in I}A_{i}$ endowed with the supremum
norm $\left\Vert \boldsymbol{a}\right\Vert =\sup_{i\in I}\left\Vert
a_{i}\right\Vert $. This C*-algebra contains the closed two-sided ideal $J_{%
\mathcal{U}}$ consisting of those elements $\boldsymbol{a}\in \ell ^{\infty
}\left( A_{i}\right) _{i\in I}$ such that $\lim_{i\rightarrow \mathcal{U}%
}\left\Vert a_{i}\right\Vert =0$. Then $\prod_{\mathcal{U}}A_{i}$ is by
definition the quotient of $\ell ^{\infty }\left( A_{i}\right) _{i\in I}$ by 
$J_{\mathcal{U}}$. If $\left[ \boldsymbol{a}\right] $ is the image in $%
\prod_{\mathcal{U}}A_{i}$ of an element $\boldsymbol{a}$ of $\ell ^{\infty
}\left( A_{i}\right) _{i\in I}$, then $\left\Vert \left[ \boldsymbol{a}%
\right] \right\Vert =\lim_{i\rightarrow \mathcal{U}}\left\Vert
a_{i}\right\Vert $. This clearly shows that such a notion of ultraproduct
indeed coincides with the notion of ultraproduct of C*-algebras as $L^{\text{%
C*}}$-structures.

Let now $A$ be a C*-algebra, and $\mathcal{U}$ is an ultrafilter. One can
then consider the ultrapower $A^{\mathcal{U}}$, and identify $A$ as a
C*-subalgebra of $A^{\mathcal{U}}$ via the diagonal embedding. The \emph{%
relative commutant }$A^{\prime }\cap A^{\mathcal{U}}$ is the set of elements 
$a$ of $A^{\mathcal{U}}$ that commute with every element of $A$.

\subsection{Quantifier-free formulas and weakly semiprojective C*-algebras 
\label{Subsection:quantifier-free}}

Suppose that $L$ is a language. An $L$-formula $\varphi $ is \emph{%
quantifier-free }if no quantifier appears in $\varphi $. Equivalently, $%
\varphi $ is of the form $q\left( \varphi _{1},\ldots ,\varphi _{n}\right) $
where $q:\mathbb{R}^{n}\rightarrow \mathbb{R}$ is a continuous function and $%
\varphi _{1},\ldots ,\varphi _{n}$ are atomic formulas. The notion of \emph{%
positive quantifier-free }formula is defined similarly, where one
furthermore demands that the connective $q:\mathbb{R}^{n}\rightarrow \mathbb{%
R}$ be nondecreasing, in the sense that $q\left( \bar{s}\right) \leq q\left( 
\bar{r}\right) $ whenever $\bar{s},\bar{r}\in \mathbb{R}^{n}$ satisfy $%
s_{i}\leq r_{i}$ for $i=1,2,\ldots ,n$.

One can then define the notion of (positive) \emph{quantifier-free }%
definable predicate, by replacing arbitrary formulas with (positive)
quantifier-free ones. A condition $\varphi \left( \bar{x}\right) \leq r$ is
then called quantifier-free if the definable predicate $\varphi \left( \bar{x%
}\right) $ is quantifier-free. A (positive) \emph{quantifier-free definable
set} is then a definable set, whose definability is witnessed by a
(positive) quantifier-free definable predicate. For instance, the unitary
group or the set of projections in a C*-algebra are positive quantifier-free
definable sets.

Let us consider now the language of C*-algebras $L^{\text{C*}}$, and let $%
\varphi \left( \bar{x}\right) $ be a positive quantifier-free definable
predicate. One can then define the \emph{universal C*-algebra }$A_{\varphi }$%
, if it exists, satisfying the condition $\varphi \left( \bar{x}\right) =0$.
This is a C*-algebra $A_{\varphi }$ containing a tuple $\bar{a}$ satisfying
the condition $\varphi \left( \bar{x}\right) =0$ and generating $A_{\varphi
} $ as a C*-algebra, which satisfies the following universal property: if $B$
is any C*-algebra, and $\bar{b}$ is a tuple in $B$ satisfying $\varphi
\left( \bar{x}\right) =0$, then there exists a unital *-homomorphism $\Phi
:A_{\varphi }\rightarrow B$ such that $\Phi \left( \bar{a}\right) =\bar{b}$.
For instance, in the case of the condition $\max \left\{ \left\Vert x^{\ast
}x-1\right\Vert ,\left\Vert xx^{\ast }-1\right\Vert \right\} =0$ whose
zeroset is the unitary group, the universal C*-algebra is the algebra $%
C\left( \mathbb{T}\right) $ of continuous functions over the set $\mathbb{T}$
of complex numbers of modulus $1$. For the condition%
\begin{equation*}
\max \left\{ \left\Vert x_{i}^{2}-x_{i}\right\Vert ,\left\Vert
x_{i}-x_{i}^{\ast }\right\Vert ,\left\Vert x_{i}x_{j}\right\Vert :1\leq
i\neq j\leq n\right\} =0
\end{equation*}%
whose zeroset is the set of $n$-tuples of pairwise orthogonal projections,
the universal C*-algebra is $\mathbb{C}^{n}$.

One can equivalently reformulate stability of a positive quantifier-free
definable predicate $\varphi \left( \bar{x}\right) $ in terms of properties
of the universal C*-algebra $A_{\varphi }$, if it exists. Suppose that $%
\left( B_{n}\right) $ is a sequence of C*-algebras, and $\mathcal{U}$ is an
ultrafilter over $\mathbb{N}$. Given an element $J$ of $\mathcal{U}$, one
can define a canonical quotient mapping $\pi _{J}:\ell ^{\infty }\left(
B_{n}\right) _{n\in J}\rightarrow \prod_{\mathcal{U}}B_{n}$ mapping $\left(
a_{n}\right) _{n\in J}$ to the element of $\prod_{\mathcal{U}}B_{n}$ having $%
\left( a_{n}\right) _{n\in J}$ as representing sequence.

\begin{definition}
A C*-algebra $A$ is \emph{weakly semiprojective }if, for every sequence $%
\left( B_{n}\right) _{n\in \mathbb{N}}$ of C*-algebras, for every
ultrafilter $\mathcal{U}$ over $\mathbb{N}$, and for every unital
*-homomorphism $\Phi :A\rightarrow \prod_{\mathcal{U}}B_{n}$, there exists $%
J\in \mathcal{U}$ and a unital *-homomorphism $\hat{\Phi}:A\rightarrow \ell
^{\infty }\left( B_{n}\right) _{n\in J}$ such that $\Phi =\pi _{J}\circ \hat{%
\Phi}$.
\end{definition}

Suppose now that $\varphi $ is a positive quantifier-free definable
predicate such that the condition $\varphi \left( \bar{x}\right) =0$ has a
universal C*-algebra $A_{\varphi }$.\ Then one can easily deduce from
Proposition \ref{Proposition:characterize-stable} the following
characterization, which recovers \cite[Theorem 4.1.4]{loring_lifting_1997}.

\begin{theorem}
\label{Theorem:weakly-semiprojective}The positive quantifier-free definable
predicate $\varphi $ is stable if and only if $A_{\varphi }$ is weakly
semiprojective.
\end{theorem}

As an application of Theorem \ref{Theorem:weakly-semiprojective}, one can
consider the C*-algebra $M_{n}\left( \mathbb{C}\right) $ of complex $n\times
n$ matrices endowed with the operator norm. This is the universal C*-algebra
associated with the definable predicate $\varphi _{M_{n}\left( \mathbb{C}%
\right) }\left( x_{ij}\right) $ given by%
\begin{equation*}
\max \left\{ \left\Vert x_{ij}x_{k\ell }-\delta _{jk}x_{i\ell }\right\Vert
,\left\vert 1-\left\Vert x_{ij}\right\Vert \right\vert ,\left\Vert
x_{ij}^{\ast }-x_{ji}\right\Vert ,\left\vert
1-\sum_{m=1}^{n}x_{m,m}\right\vert :1\leq i,j,k,\ell \leq n\right\} \text{.}
\end{equation*}%
The zeroset of $\varphi _{M_{n}\left( \mathbb{C}\right) }$ in a C*-algebra $%
A $ is the set of \emph{matrix units }for a unital copy of $M_{n}\left( 
\mathbb{C}\right) $ inside $A$. The C*-algebra $M_{n}\left( \mathbb{C}%
\right) $ is weakly semiprojective \cite[Theorem 10.2.3]{loring_lifting_1997}%
, hence the predicate $P$ is stable. More generally, given a
finite-dimensional C*-algebra $F$, i.e.\ a finite sum $M_{n_{1}}\left( 
\mathbb{C}\right) \oplus \cdots \oplus M_{n_{\ell }}\left( \mathbb{C}\right) 
$, one can consider the definable predicate $P_{F}$ whose zeroset in a
C*-algebra $A$ is the set of matrix units for a unital copy of $F$ inside $A$%
. The fact that $F$ is weakly semiprojective shows that such a definable
predicate $P_{F}$ is stable.

\subsection{Saturation}

One important feature of ultrapowers is their being very \textquotedblleft
rich\textquotedblright . To make this precise, we introduce the notion of
countably saturated structure. Suppose that $L$ is a language, and $M$ is an 
$L$-structure. Let $A$ be a subset of $M$. Then one can consider the
language $L\left( A\right) $ obtained starting from $L$ by adding a constant
symbol $c_{a}$ for every element $a$ of $A$. Then one can canonically regard 
$M$ as an $L\left( A\right) $ structure by defining the interpretation of
the constant symbol $c_{a}$ in $M$ to be $a$ itself. We refer to formulas in
the language $L\left( A\right) $ as $L$-formulas with parameters from $A$.

If $\mathcal{C}$ is a class of $L$-structure, then we say that $L$ is \emph{%
separable for} $\mathcal{C}$ if, for every $n\in \mathbb{N}$ and variables $%
x_{1},\ldots x_{n}$ of domains $D_{1},\ldots ,D_{n}$, the seminorm on the
space of $L$-formulas $\varphi \left( x_{1},\ldots ,x_{n}\right) $ defined by%
\begin{equation*}
\left\Vert \varphi \right\Vert =\sup \left\{ \left\vert \varphi \left( \bar{a%
}\right) \right\vert :M\in \mathcal{C},\bar{a}\in D_{1}^{M}\times \cdots
\times D_{n}^{M}\right\}
\end{equation*}%
is separable.

Recall that and $L$-condition in the variables $x_{1},\ldots ,x_{n}$ of
domains $D_{1},\ldots ,D_{n}$ is an expression of the form $\varphi \left(
x_{1},\ldots ,x_{n}\right) \leq r$ where $\varphi $ is an $L$-formula---or,
more generally, a definable predicate---in the free variables $x_{1},\ldots
,x_{n}$, and $r\in \mathbb{R}$. If $M$ is an $L$-structure and $\bar{a}\in
D_{1}^{M}\times \cdots \times D_{n}^{M}$, then $\left( a_{1},\ldots
,a_{n}\right) $ realizes such a condition if $\varphi ^{M}\left(
a_{1},\ldots ,a_{n}\right) \leq r$. A collection of $L$-conditions in the
variables $\bar{x}=\left( x_{n}\right) _{n\in \mathbb{N}}$ is called an $L$%
-type in the variables $\bar{x}$. If $\mathfrak{t}$ is an $L$-type in the
variables $\bar{x}$, then we also write $\mathfrak{t}\left( \bar{x}\right) $%
. Given an $L$-type $\mathfrak{t}$ we let $\mathfrak{t}^{+}$ be the $L$-type
consisting of conditions $\varphi \left( \bar{x}\right) \leq r+\varepsilon $
where $\varphi \left( \bar{x}\right) \leq r$ is a condition in $\mathfrak{t}$
and $\varepsilon >0$. Given an $L$-structure $M$ and a tuple $\bar{a}$ in $M$%
, one defines the \emph{complete }$L$-type of $\bar{a}$ to be the type $%
\mathfrak{t}\left( \bar{x}\right) $ consisting of all the $L$-conditions
that are satisfied by $\bar{a}$.

\begin{definition}
Suppose that $\mathfrak{t}$ is an $L$-type in the variables $\bar{x}$ and $M$
is an $L$-structure. Then $\mathfrak{t}$ is:

\begin{itemize}
\item \emph{realized }in $M$ if there exist $a_{n}\in D_{n}^{M}$ for $n\in 
\mathbb{N}$ such that the sequence $\left( a_{n}\right) $ realizes every
condition in $\mathfrak{t}$;

\item \emph{approximately realized }in $M$ if every finite set of conditions
in $\mathfrak{t}^{+}$ is realized in $M$.
\end{itemize}
\end{definition}

The notion realized and approximately realized types allows one to define
the property of countable saturation for structures, which formalizes the
intuitive idea of being \textquotedblleft rich\textquotedblright .

\begin{definition}
\label{Definition:saturated}An $L$-structure $M$ is \emph{countably saturated%
} if for every separable subset $A$ of $M$ and for every $L\left( A\right) $%
-type $\mathfrak{t}$, if $\mathfrak{t}$ is approximately realized in $M$
then $\mathfrak{t}$ is realized in $M$.
\end{definition}

A fundamental feature of ultraproducts associated with countably incomplete
ultrafilters is their being countably saturated.

\begin{proposition}
\label{Proposition:countably-saturated}Let $L$ be a first-order language,
and $\mathcal{C}$ is a class of $L$-structures such that $L$ is separable
for $\mathcal{C}$. If $\mathcal{U}$ is a countably incomplete filter on a
set $I$, and $\left( M_{i}\right) _{i\in I}$ is an $I$-sequence of
structures in $\mathcal{C}$, then the ultraproduct $\prod_{\mathcal{U}}M_{i}$
is countably saturated.
\end{proposition}

\begin{proof}
Set $M:=\prod_{\mathcal{U}}M_{i}$. Suppose that $A$ is a separable subset of 
$M$. We need to show that if an $L\left( A\right) $-type $\mathfrak{t}$ is
approximately realized in $M$, then it is realized in $M$. For every $\left[ 
\boldsymbol{a}\right] \in A$ fix a representative sequence $\boldsymbol{a}%
=\left( a_{i}\right) _{i\in I}$. For every $i\in I$ we can regard $M_{i}$ as
an $L\left( A\right) $-structure by declaring the interpretation $c_{%
\boldsymbol{a}}$ of $\boldsymbol{a}$ to be equal to $a_{i}$. Observe that,
since $A$ is separable and $L$ is separable for $\mathcal{C}$, we have that $%
L\left( A\right) $ is separable for $\left\{ M_{i}:i\in I\right\} \cup
\left\{ M\right\} $. Thus after replacing $L$ with $L\left( A\right) $ we
can assume without loss of generality that $A$ is empty.

Let thus $\mathfrak{t}$ be an $L$-type in the variables $\bar{x}$ of domains 
$\bar{D}$ which is approximately realized in $M$. Since $L$ is separable for 
$\mathcal{C}$, we can assume without loss of generality that $\mathfrak{t}$
consists of a sequence of conditions $\varphi _{n}\left( x_{1},\ldots
,x_{n}\right) \leq 0$ for $n\in \mathbb{N}$. Since by assumption $\mathcal{U}
$ is countably incomplete, we can fix a decreasing sequence $\left(
I_{n}\right) _{n\in \mathbb{N}}$ of elements of $\mathcal{U}$ with empty
intersection.

For every $n\in \mathbb{N}$ define $\psi _{n}$ to be the sentence $%
\inf_{x_{1},\ldots ,x_{n}}\max \{\varphi _{1}\left( x_{1}\right) ,\ldots
,\varphi _{n}\left( x_{1},\ldots ,x_{n}\right) \}$. By assumption, $%
\mathfrak{t}$ is approximately realized in $M$, hence $\psi _{n}^{M}<\frac{1%
}{n}$. Therefore by \L os' theorem, the set%
\begin{equation*}
J_{n}=\left\{ i\in I_{n}:\psi _{n}^{M}<\frac{1}{n}\right\}
\end{equation*}%
belongs to $\mathcal{U}$. Observe that $\left( J_{n}\right) _{n\in \mathbb{N}%
}$ is a decreasing sequence of elements of $\mathcal{U}$ with empty
intersection. Thus every $i\in I$ only belongs to finitely many of the $%
J_{n} $'s, and hence there exists a largest element $n\left( i\right) $ of $%
\mathbb{N}$ such that $i\in J_{n\left( i\right) }$. If $i$ does not belong
to any of the $J_{n}$'s then we set $n\left( i\right) =0$.

Define now, for $i\in I$, if $n\left( i\right) >0$, $\bar{a}_{i}=\left(
a_{1,i},\ldots ,a_{n\left( i\right) ,i}\right) \in D_{1}^{M_{i}}\times
\cdots \times D_{n\left( i\right) }^{M_{i}}$ such that $\psi _{n\left(
i\right) }(\bar{a}_{i})<1/n\left( i\right) $. For $i\in I$ such that $%
n\left( i\right) =0$ define $\bar{a}_{i}\in \bar{D}^{M_{i}}$ arbitrarily.
Now, it is clear from the definition of $n\left( i\right) $ that, for every $%
n\in \mathbb{N}$ and $i\in J_{n}$, $n\left( i\right) \geq n$. Let then $%
\boldsymbol{a}_{n}$ for $n\in \mathbb{N}$ be the element of $\prod_{\mathcal{%
U}}M_{i}$ with representative sequence $\left( a_{n,i}\right) _{i\in I}$.
Since $J_{n}\in \mathcal{U}$, we have that $\left( \mathcal{U}i\right) $, $%
\psi _{n}^{M_{i}}\left( \bar{a}_{i}\right) \leq 1/n$. Therefore by \L os'
theorem, $\psi _{n}^{M}\left( \boldsymbol{\bar{a}}\right) \leq 1/n$ for
every $n\in \mathbb{N}$. Since this holds for every $n\in \mathbb{N}$, we
have that $\varphi _{n}^{M}(\boldsymbol{\bar{a}})\leq 0$ for every $n\in 
\mathbb{N}$. This shows that $\boldsymbol{\bar{a}}$ witnesses that the type $%
\mathfrak{t}$ is realized in $M$.
\end{proof}

A type $\mathfrak{t}$ is called (positive) \emph{quantifier-free }if all the
conditions in $\mathfrak{t}$ involve (positive) quantifier-free formulas.
One can then naturally define the notion of \emph{(positive) quantifier-free
countably saturated structure} as in\ Definition \ref{Definition:saturated},
by replacing arbitrary types with (positive) quantifier-free ones.

It is easy to see that, if $M$ is a quantifier-free countably saturated
structure, and $\mathfrak{t}$ is a quantifier-free type, then the set of
realizations of $\mathfrak{t}$ in $M$ is quantifier-free countably
saturated. Particularly, if $A$ is a C*-algebra, and $\mathcal{U}$ is a
countably incomplete ultrafilter, then the relative commutant $A^{\prime
}\cap A^{\mathcal{U}}$ is quantifier-free countably saturated.

\subsection{Elementary equivalence and elementary embeddings}

Elementary equivalence is a key notion in model theory.\ Roughly speaking,
it asserts that two structures are indistinguishable as long as first-order
properties (i.e.\ properties that are captured by formulas) are concerned.
To make this precise, let $L$ be a language, and $M$ be an $L$-structure.
The \emph{theory }of $M$ is the multiplicative functional $\varphi \mapsto
\varphi ^{M}$ defined on the space of $L$-sentences, which maps an $L$%
-sentence to its interpretation in $M$.

\begin{definition}
Two $L$-structures are \emph{elementarily equivalent }if they have the same
theory.
\end{definition}

A straightforward induction on formulas shows that two isomorphic structures
are, in particular, elementarily equivalent. It follows from \L os' theorem
that if $M$ is an $L$-structure and $\mathcal{U}$ is an ultrafilter, then $M$
is elementarily equivalent to $M^{\mathcal{U}}$. In particular, if two
structures have isomorphic ultrapowers, then they are elementarily
equivalent. The continuous version of a classical result of Keisler and
Shelah asserts that, in fact, the converse holds as well. The notion of 
\emph{elementary embedding} is tightly connected with elementary equivalence.

\begin{definition}
\label{Definition:elementary-equivalence}An \emph{embedding }of $M$ to $N$
is a function $\Phi :M\rightarrow N$ such that, for every domain $D$ in $L$,
the image of $D^{M}$ under $\Phi $ is contained in $D^{N}$, and such that $%
\varphi (\Phi \left( \bar{a}\right) )=\varphi \left( \bar{a}\right) $ for
every \emph{atomic} formula $\varphi \left( \bar{x}\right) $ in the
variables $\bar{x}$ of domains $\bar{D}$ and $\bar{a}\in \bar{D}^{M}$.
\end{definition}

We say that $M$ is a \emph{substructure }of $N$ if $M\subset N$ and the
inclusion map is an embedding. Clearly, after renaming the elements of $M$,
one can always assume that a given embedding $\Phi :M\rightarrow N$ is
simply the inclusion map.

Suppose that $M\subset N$ is a substructure. Then one can regard $N$ as an $%
L\left( M\right) $-structure in the obvious way, by interpreting the
constant symbol $c_{a}$ associated with $a\in M$ as $a$ itself, regarded as
an element of $N$.

\begin{definition}
\label{Definition:elementary-embedding}Suppose that $M\subset N$ is a
substructure.\ Then $M$ is an \emph{elementary substructure }of $N$ if $M$
and $N$ are elementarily equivalent as $L\left( M\right) $-structures.
\end{definition}

The notion of \emph{elementary embedding} $\Phi :M\rightarrow N$ is defined
analogously. In this case, one can consider $N$ as an $L\left( M\right) $%
-structure by interpreting $c_{a}$ for $a\in M$ as $\Phi \left( a\right) $.

If $M$ is any $L$-structure and $\mathcal{U}$ is an ultrafilter, then there
is a canonical embedding $\Delta _{M}:M\rightarrow M^{\mathcal{U}}$ obtained
by mapping an element $a$ of $M$ to the element of $M^{\mathcal{U}}$ that
admits the sequence constantly equal to $a$ as representing sequence.\ It is
a consequence of \L os' theorem that this is in fact an \emph{elementary }%
embedding. A useful criterion to verify that an inclusion $M\subset N$ is
elementary is the following Tarski--Vaught test.

\begin{proposition}
\label{Proposition:Tarski-Vaught}Suppose that $M,N$ are $L$-structures such
that $M\subset N$. Assume that for every $L$-formula $\varphi \left( \bar{x}%
,y\right) $, where $\bar{x}$ are variables with domains $\bar{D}$ and $y$
has domain $D$, and for every tuple $\bar{a}\in \bar{D}^{M}$ one has that%
\begin{equation*}
\inf \left\{ \varphi \left( \bar{a},b\right) :b\in D^{N}\right\} =\inf
\left\{ \varphi \left( \bar{a},b\right) :b\in D^{M}\right\} \text{.}
\end{equation*}%
Then $M$ is an elementary substructure of $N$.
\end{proposition}

\begin{proof}
One needs to show by induction on the complexity of a given formula $\psi
\left( \bar{x}\right) $ and tuple $\bar{a}$ in $M$, where $\bar{x}$ are
variables with domains $\bar{D}$ and $\bar{a}\in \bar{D}^{M}$, that $\psi
^{M}\left( \bar{a}\right) =\psi ^{N}\left( \bar{a}\right) $. The base case
when $\psi $ is atomic (or quantifier-free) is obvious. The assumption is
used to deal with the quantifier case.
\end{proof}

\subsection{Existential equivalence and existential embeddings}

In many cases, it is sufficient, and useful, to consider a suitably
restricted class of formulas. We have already considered the class of
quantifier-free formulas. The next natural restricted class of formulas
consists of \emph{existential formulas}. These are the formulas of the form $%
\inf_{\bar{x}}\varphi $ where $\varphi $ is a quantifier-free formula. The
name \textquotedblleft existential\textquotedblright\ is due to the fact
that $\inf $ is regarded as the continuous analog of the existential
quantifier $\exists $ from the logic for discrete structures.

An even more restrictive class consists of the \emph{positive }existential
formulas, which are those existential formulas of the form $\inf_{\bar{x}%
}q\left( \varphi _{1},\ldots ,\varphi _{n}\right) $ where $\varphi
_{1},\ldots ,\varphi _{n}$ are atomic formulas, and $q:\mathbb{R}%
^{n}\rightarrow \mathbb{R}$ is a continuous function which is \emph{%
nondecreasing}, in the sense that $q\left( \bar{r}\right) \leq q\left( \bar{s%
}\right) $ if $\bar{r},\bar{s}\in \mathbb{R}^{n}$ are such that $r_{i}\leq
s_{i}$ for $i\in \left\{ 1,2,\ldots ,n\right\} $. \emph{Universal formulas }%
and \emph{positive universal formulas }are defined and characterized in a
similar fashion, by replacing $\inf $ with $\sup $.

One can then define the (positive) existential theory of a structure $M$ to
be the functional $\varphi \mapsto \varphi ^{M}$ defined on the space of
(positive) existential formulas. It follows from \L os' theorem and
countable saturation of ultrapowers that, if $M,N$ are separable structures
and $\mathcal{U}$ is a countably incomplete ultrafilter, then $M$ and $N$
have the same existential theory if and only if $M$ embeds into $N^{\mathcal{%
U}}$ and $N$ embeds into $M^{\mathcal{U}}$. One can similarly characterize
the property of having the same positive existential theory in terms of
ultraproducts and morphisms.

\begin{definition}
Suppose that $M,N$ are structures and $\Phi :M\rightarrow N$ is a function.
Then $\Phi $ is a \emph{morphism} if for every domain $D$ , the image of $%
D^{M}$ under $\Phi $ is contained in $D^{N}$, and such that $\varphi (\Phi
\left( \bar{a}\right) )\leq \varphi \left( \bar{a}\right) $ for every \emph{%
atomic} formula $\varphi \left( \bar{x}\right) $ in the variables $\bar{x}$
of domains $\bar{D}$ and $\bar{a}\in \bar{D}^{M}$.
\end{definition}

One can show using countable saturation of ultraproducts and \L os' theorem
that, if $M$ and $N$ are separable structures and $\mathcal{U}$ is a
countably incomplete ultrafilter, then $M$ and $N$ have the same positive
existential theory if and only if $M$ admits a morphism to $N^{\mathcal{U}}$
and $N$ admits a morphism to $M^{\mathcal{U}}$. More generally, one has that 
$\varphi ^{M}\geq \varphi ^{N}$ for every existential (respectively,
positive existential) sentence if and only if there is an embedding
(respectively, a morphism) from $M$ to $N^{\mathcal{U}}$.

The notions of (positively) existential substructure and (positively)
existential embedding are defined in the same fashion as elementary
substructure and elementary embeddings, replacing arbitrary formulas with
(positive) existential formulas. The following characterization of
(positively) existential embedding follows again from \L os' theorem and
countable saturation of ultraproducts.

\begin{proposition}
\label{Proposition:characterize-existential}Suppose that $M,N$ are
structures, and $\Phi :M\rightarrow N$ is an embedding. Fix a countably
incomplete ultrafilter $\mathcal{U}$. The following assertions are
equivalent:

\begin{enumerate}
\item $\Phi $ is an existential (respectively, positively existential)
embedding;

\item there exists an embedding (respectively, a morphism) $\Psi
:N\rightarrow M^{\mathcal{U}}$ such that $\Psi \circ \Phi $ is equal to the
diagonal embedding $\Delta _{M}:M\rightarrow M^{\mathcal{U}}$.
\end{enumerate}
\end{proposition}

Suppose now that $\mathcal{C}$ is an elementary class of $L$-structures. An $%
L$-structure $M$ is said to be (positively) \emph{existentially closed }in $%
\mathcal{C}$ if it belongs to $\mathcal{C}$ and, whenever $N$ is a structure
in $\mathcal{C}$ containing $M$ as a substructure, the inclusion $M\subset N$
is existential.

\subsection{Positively existential embeddings of C*-algebras}

In the case of C*-algebras regarded as $L^{\text{C*}}$-structures, an \emph{%
embedding }is an injective unital *-homomorphism, and a \emph{morphism }is a
unital *-homomorphism. The general notion of (positively) existential
embedding yields a notion of (positively) existential embedding between
C*-algebras. If $A,B$ are separable C*-algebras, an embedding $\Phi
:A\rightarrow B$ is positively existential if and only if it is sequentially
split in the sense of \cite{barlak_sequentially_2016}.

Using positively existential embeddings, one can give a model-theoretic
description of relative commutants, as follows.

\begin{proposition}
\label{Proposition:commutant-existential}Suppose that $A,C$ are separable
C*-algebras, and $\mathcal{U}$ is a countably incomplete ultrafilter. Then
the following assertions are equivalent:

\begin{enumerate}
\item the embedding $1_{C}\otimes \mathrm{id}_{A}:A\rightarrow C\otimes
_{\max }A$ is positively existential;

\item there exists a morphism from $C$ to $A^{\prime }\cap A^{\mathcal{U}}$.
\end{enumerate}
\end{proposition}

\begin{proof}
(1)$\Rightarrow $(2) We identify $A$ with its image $1\otimes _{\max }A$
inside $C\otimes _{\max }A$. By countable saturation of $A^{\mathcal{U}}$ it
suffices to prove the following. Suppose that $\varphi \left( \bar{x}\right) 
$ is a positive existential $L^{\text{C*}}$-formula in the variables $%
x_{1},\ldots ,x_{n}$, $\varepsilon >0$, $a_{1},\ldots ,a_{k}$ is a tuple in $%
A$, and $\bar{c}$ is a tuple in $C$ satisfying the condition $\varphi \left( 
\bar{x}\right) \leq 0$. Then there exists a tuple in $A$ satisfying the $L^{%
\text{C*}}\left( A\right) $-condition $\psi \left( \bar{x}\right) \leq
\varepsilon $ where $\psi \left( \bar{x}\right) $ is the $L^{\text{C*}%
}\left( A\right) $-formula%
\begin{equation*}
\max \left\{ \varphi \left( \bar{x}\right) ,\left\Vert
x_{i}a_{j}-a_{j}x_{i}\right\Vert :i=1,2,\ldots ,n\text{ and }j=1,2,\ldots
,k\right\} \text{.}
\end{equation*}%
Considering the tuple $c_{i}\otimes 1_{A}\in C\otimes A$ for $i=1,2,\ldots
,n $ shows that the $L^{\text{C*}}\left( A\right) $-condition $\psi \left( 
\bar{x}\right) \leq 0$ is satisfied in $C\otimes A$. Since the inclusion $%
A\subset C\otimes _{\max }A$ is positively existential by assumption, we
conclude that the $L^{\text{C*}}\left( A\right) $-condition $\psi \left( 
\bar{x}\right) \leq \varepsilon $ is satisfied in $A$. This concludes the
proof.

(2)$\Rightarrow $(1) Suppose that there exists a morphism $\eta
:C\rightarrow A^{\prime }\cap A^{\mathcal{U}}$. Then the function $\left(
A^{\prime }\cap A^{\mathcal{U}}\right) \times A\rightarrow A^{\mathcal{U}}$
given by $\left( \left[ a_{i}\right] ,b\right) \mapsto \left[ a_{i}b\right] $
induces by the universal property of maximal tensor products a morphism $%
\Psi :\left( A^{\prime }\cap A^{\mathcal{U}}\right) \otimes _{\max
}A\rightarrow A^{\mathcal{U}}$. One can then define $\hat{\Psi}:=\Psi \circ
\left( \eta \otimes \mathrm{id}_{A}\right) :C\otimes _{\max }A\rightarrow A^{%
\mathcal{U}}$. Observe that this is a morphism such that $\hat{\Psi}\circ
\left( 1_{C}\otimes \mathrm{id}_{A}\right) $ is the diagonal embedding of $A$
into $A^{\mathcal{U}}$. This shows that $1_{C}\otimes \mathrm{id}_{A}$ is a
positively existential embedding.
\end{proof}

\section{The effect of the Continuum Hypothesis\label{Section:continuum}}

\subsection{The Continuum Hypothesis}

The continuum $\mathfrak{c}$ is, by definition, the cardinality of the set $%
\mathbb{R}$ of real numbers. The Continuum Hypothesis (CH) is the assertion
that $\mathfrak{c}$ coincides with the least uncountable cardinal $\aleph
_{1}$. A famous open problem in set theory asked whether the Continuum
Hypothesis holds, or more precisely whether it follows from the usual axioms
for set theory known as Zermelo--Frankel axioms with Choice (ZFC). In 1940 G%
\"{o}del \cite{godel_consistency_1940} showed that the Continuum Hypothesis
is \emph{consistent }with ZFC, in the sense that it can be added to ZFC
without leading to a contradiction (assuming that ZFC itself is not
contradictory). In the early 1960s, Cohen developed the method of forcing,
and used it to show that the \emph{negation }of the Continuum Hypothesis
(the assertion that $\mathfrak{c}$ is strictly larger than $\aleph _{1}$) is
also consistent with ZFC \cite%
{cohen_independence_1963,cohen_independence_1964}. These results together
imply that the Continuum Hypothesis is \emph{independent }of ZFC, in the
sense that it can not be either proved nor disproved from the axioms of ZFC.

The value of the continuum turns out to be irrevelant for what concerns
sufficiently simple statement. As a rule of thumb, any \textquotedblleft
reasonable statement\textquotedblright\ concerning \emph{separable }%
C*-algebras which can be proved assuming the Continuum Hypothesis, can also
be proved without the Continuum Hypothesis. (This assertion can be made
precise in set theory through the notion of \emph{absoluteness}, and it is
the subject of several absoluteness results such as Shoenfield's
absoluteness theorem \cite{shoenfield_problem_1961}.)

On the other hand, the value of the continuum, or more generally additional
set-theoretic axioms, can have a deep influence on the structure and
properties of \textquotedblleft massive C*-algebras\textquotedblright .
Paradigmatic in this sense is the question of whether all automorphisms of
the Calkin algebra $\mathcal{Q}$ are inner. Recall that $\mathcal{Q}$ is the
quotient of the algebra $B\left( H\right) $ of bounded linear operators on
the separable Hilbert space $H$ by the closed two-sided ideal of compact
operators. Even when $H$ is separable, $\mathcal{Q}$ is nonseparable, and in
fact it has density character $\mathfrak{c}$.

Originally posed by Brown--Douglas--Fillmore in \cite{brown_extensions_1977}%
, this problem was initially addressed in 2007 by Phillips--Weaver \cite%
{phillips_calkin_2007}, who showed that, assuming CH, $\mathcal{Q}$ has an
automorphism which is not inner. Later on, Farah has proved that, under
different set-theoretic assumptions, which imply in particular the negation
of CH, all the automorphisms of the Calkin algebra are inner \cite%
{farah_all_2011,farah_all_2011-1}. These results have later been generalized
in \cite%
{coskey_automorphisms_2014,farah_homeomorphisms_2012,vignati_nontrivial_2017,farah_rigidity_2016}
to other massive C*-algebras, which are obtained as \emph{corona algebras}
of separable C*-algebras.


Another problem which is sensitive of the value of the continuum concerns
the number of ultrapowers of a fixed separable C*-algebra
infinite-dimensional C*-algebra $A$ with respect to nonprincipal
ultrafilters over $\mathbb{N}$. CH implies that all such ultrapowers of $A$
are isomorphic. On the other hand, as shown by Farah--Hart--Sherman \cite%
{farah_model_2013}, if CH fails then there exist two nonisomorphic such
ultrapowers of $A$ (in fact, there exist $2^{\mathfrak{c}}$ pairwise
nonisomorphic such ultrapowers of $A$, as proved by Farah--Shelah \cite%
{farah_dichotomy_2010}).\ The same conclusions hold if one considers,
instead of the ultrapower, the \emph{relative commutant }of $A$ inside the
ultrapower. The analogous question in the case of II$_{1}$ factors had been
posed by McDuff \cite{mcduff_central_1970}, and it has also been settled in 
\cite{farah_model_2013}.

\subsection{Isomorphism of countably saturated structures}

Let $L$ be a language. Recall that an $L$-structure $M$ is countably
saturated if for every separable subset $A$ of $M$ and every $L\left(
A\right) $-type $\mathfrak{t}\left( \bar{x}\right) $, if $\mathfrak{t}\left( 
\bar{x}\right) $ is approximately realized in $M$, then it is realized in $M$%
. A fundamental fact in model theory is that the any two elementarily
equivalent countably saturated structures $\emph{of}$ \emph{density
character }$\aleph _{1}$ are isomorphic. More generally, we have the
following result.

\begin{theorem}
\label{Theorem:isomorphic-saturated}Suppose that $L$ is a language, and let $%
\mathcal{C}$ be a class of $L$-structures such that $L$ is separable for $%
\mathcal{C}$. Consider two elementarily equivalent countably saturated
structures $M,N$ in $\mathcal{C}$. Then $M$ and $N$ are isomorphic.
Furthermore, if $M_{0}\subset M$ is a separable substructure, and $\Phi
_{0}:M_{0}\rightarrow N$ is an elementary embedding, then $\Phi $ extends to
an isomorphism $M\rightarrow N$.
\end{theorem}

\begin{proof}
We prove the second assertion, the proof of the first assertion being
similar.

Since $M,N$ have density character $\aleph _{1}$, one can enumerate dense
subsets $\left\{ a_{\lambda }:\lambda <\omega _{1}\right\} $ of $M$ and $%
\left\{ b_{\lambda }:\lambda <\omega _{1}\right\} $ of $N$. Define $N_{0}$
to be the range of $\Phi $. Let also $M_{\lambda }$ for $\lambda <\omega
_{1} $ be $M_{0}\cup \left\{ a_{\mu }:\mu <\lambda \right\} \subset M$, and
similarly $N_{\lambda }$ for $\lambda <\omega _{1}$ be $N_{0}\cup \left\{
b_{\mu }:\mu <\lambda \right\} \subset N$. Say that an ordinal $\lambda
<\omega _{1}$ is \emph{odd }if it is of the form $\mu +n$ where $\mu $ is a
limit ordinal and $n\in \omega $ is odd, and it is \emph{even }otherwise.

We define by recursion on $\lambda <\omega _{1}$ elements $\hat{a}_{\lambda
} $ of $M$ and $\hat{b}_{\lambda }$ of $N$ such that:

\begin{enumerate}
\item $\hat{a}_{\lambda }=a_{\lambda }$ if $\lambda $ is even;

\item $\hat{b}_{\lambda }=b_{\lambda }$ if $\lambda $ is odd;

\item the assignment $\Phi _{\lambda }:M_{\lambda }\rightarrow N_{\lambda }$
which is the identity on $M_{0}$ and such that $\Phi \left( \hat{a}_{\mu
}\right) =\hat{b}_{\mu }$ for $\mu <\lambda $ is well defined and satisfies $%
\varphi ^{N}\left( \Phi _{\lambda }\left( \bar{a}\right) \right) =\varphi
^{M}\left( \bar{b}\right) $ for any tuple $\bar{a}$ in $M_{\lambda }$ and $%
L\left( M_{\lambda }\right) $-formula $\varphi \left( \bar{x}\right) $,
where $N_{\lambda }$ is regarded as an $L\left( N_{\lambda }\right) $%
-structure by interpreting the constant $c_{a}$ associated with $a\in
M_{\lambda }$ as $\Phi \left( a\right) $.
\end{enumerate}

Suppose that $\lambda <\omega _{1}$ and $\hat{a}_{\mu },\hat{b}_{\mu }$ have
been defined for $\mu <\lambda $ in such a way that (1),(2),(3) above hold.
(Observe that (3) holds when $\lambda =0$ by the assumption that $\Phi _{0}$
is an elementary embedding). We consider the case when $\lambda $ is even,
the case of $\lambda $ being odd is analogous. We then set $\hat{a}_{\lambda
}=a_{\lambda }$, and the consider the complete $L\left( M_{\lambda }\right) $%
-type of $\hat{a}_{\lambda }$. Recall that this is the $L\left( M_{\lambda
}\right) $-type $\mathfrak{t}\left( x\right) $ consisting of all the $%
L\left( M_{\lambda }\right) $ conditions satisfied by $\hat{a}_{\lambda }$.
By the inductive assumption (3), $\mathfrak{t}\left( x\right) $ is
approximately realized in $N$. Since $N$ is countably saturated, we can
conclude that $\mathfrak{t}\left( x\right) $ is realized in $N$. We then
define $\hat{b}_{\lambda }$ to be any realization of $\mathfrak{t}\left(
x\right) $ in $N$. It is clear by definition of $\mathfrak{t}\left( x\right) 
$ together with the inductive assumption that such a choice indeed satisfies
(3). This concludes the recursive construction.

Observe that, by (1), $\left\{ \hat{a}_{\lambda }:\lambda <\omega
_{1}\right\} $ is a dense subset of $M$. Similarly, by (2), $\{\hat{b}%
_{\lambda }:\lambda <\omega _{1}\}$ is a dense subset of $N$. Granted the
construction, one can define the map $\Phi :\left\{ \hat{a}_{\lambda
}:\lambda <\omega _{1}\right\} \rightarrow \{\hat{b}_{\lambda }:\lambda
<\omega _{1}\}$ by $\Phi \left( \hat{a}_{\lambda }\right) =\hat{b}_{\lambda
} $. By (3), this extends to an isomorphism $\Phi :M\rightarrow N$,
concluding the proof.
\end{proof}

\subsection{Ultrapowers and the Continuum Hypothesis}

We now specialize the discussion to C*-algebras. Let $A$ be an \emph{%
infinite-dimensional} \emph{separable }C*-algebra, and $\mathcal{U}$ is a
nonprincipal ultrafilter over $\mathbb{N}$. Recall that this means that $%
\mathcal{U}$ does not contain any finite set, which is equivalent to the
assertion that $\mathcal{U}$ is countably incomplete. We consider the
ultrapower $A^{\mathcal{U}}$ and the relative commutant $A^{\prime }\cap A^{%
\mathcal{U}}$. A basic question is: how large $A^{\mathcal{U}}$ and $%
A^{\prime }\cap A^{\mathcal{U}}$ are? Considering that the continuum $%
\mathfrak{c}$ is also the size of the set of functions from $\mathbb{N}$ to $%
\mathbb{N}$, since we are assuming that $A$ is separable, and that $\mathcal{%
U}$ is an ultrafilter over $\mathbb{N}$, it is easy to see that $A^{\mathcal{%
U}}$ has density character \emph{at most }$\mathfrak{c}$. (Recall that the 
\emph{density character }of a C*-algebra is the least size of a dense set.)
Clearly, the same conclusion applies to $A^{\prime }\cap A^{\mathcal{U}}$,
which is a C*-subalgebra of $A^{\mathcal{U}}$. We now claim that, in fact, $%
A^{\prime }\cap A^{\mathcal{U}}$ and hence $A^{\mathcal{U}}$ have density
character exactly $\mathfrak{c}$. In order to show this, we will use the
following criterion, which we formulate in the general setting of structures
in an arbitrary language $L$.

\begin{proposition}
\label{Proposition:size-ultraproduct}Suppose that $L$ is a language, $%
\mathcal{U}$ is a nonprincipal ultrafilter over $\mathbb{N}$, $\mathfrak{t}%
\left( x\right) $ is an $L$-type in the variable $x$ with corresponding
domain $D$, and $\mathcal{C}$ is a class of $L$-structure. Assume that $L$
is separable for $\mathcal{C}$. Consider a sequence $\left( M_{n}\right)
_{n\in \mathbb{N}}$ of $L$-structures. Assume that there exist

\begin{itemize}
\item $\delta >0$, and

\item for every $n\in \mathbb{N}$, an $n$-tuple $\bar{a}^{(n)}=(a_{1}^{(n)},%
\ldots ,a_{n}^{(n)})$ of elements of $D^{M_{n}}$,
\end{itemize}

such that

\begin{enumerate}
\item for every $n\in \mathbb{N}$, and $1\leq i,j\leq n$, $d^{M_{n}}\left(
a_{i},a_{j}\right) \geq \delta $, and

\item for every finite set $\mathfrak{t}_{0}\left( x\right) $ of conditions
in $\mathfrak{t}\left( x\right) ^{+}$, every element of $\bar{a}^{(n)}$
satisfies $\mathfrak{t}_{0}\left( x\right) $ for all but finitely many $n\in 
\mathbb{N}$.
\end{enumerate}

Let $M$ be the ultraproduct $\prod_{\mathcal{U}}M_{n}$, and $M^{\mathfrak{t}%
} $ be the set of realizations of $\mathfrak{t}\left( x\right) $ in $\prod_{%
\mathcal{U}}M_{n}$. Then $M^{\mathfrak{t}}$ contains a family $\left(
a_{i}\right) _{i<\mathfrak{c}}$ of size continuum of elements satisfying $%
d^{M}\left( a_{i},a_{j}\right) \geq \delta $ for every $i,j<\mathfrak{c}$.
\end{proposition}

In fact, a more general version of Proposition \ref%
{Proposition:size-ultraproduct} holds, where one considers a type in more
than one variable.

In order to prove Proposition \ref{Proposition:size-ultraproduct}, we will
use the following basic lemma from combinatorial set theory. Let us say that
two functions $f,g:\mathbb{N}\rightarrow \mathbb{N}$ are eventually distinct
if $\left\{ n\in \mathbb{N}:f\left( n\right) =g\left( n\right) \right\} $ is
finite.

\begin{lemma}
\label{Lemma:eventually-distinct}There exists a family $\mathcal{F}$ of size
continuum of pairwise eventually distinct functions $f:\mathbb{N}\rightarrow 
\mathbb{N}$ satisfying $f\left( n\right) \leq n$ for every $n\in \mathbb{N}$
and $\lim_{n\rightarrow +\infty }f\left( n\right) =+\infty $.
\end{lemma}

\begin{proof}
We will use the fact that the continuum is equal to the cardinality of the
collection of \emph{infinite }subsets of $\mathbb{N}$. For an infinite
subset $A$ of $\mathbb{N}$, define the function $f_{A}:\mathbb{N}\rightarrow 
\mathbb{N}$ by%
\begin{equation*}
f_{A}\left( n\right) =\sum_{k<\left\lfloor \log _{2}n\right\rfloor }\chi
_{A}\left( k\right) 2^{k}\text{.}
\end{equation*}%
Observe that $\left\vert A\cap \left[ 1,\log _{2}n\right] \right\vert \leq
f_{A}\left( n\right) \leq n$ for every $n\in \mathbb{N}$. Furthermore, if $%
A,B$ are distinct infinite subsets of $\mathbb{N}$, then $f_{A}$ and $f_{B}$
are eventually distinct.
\end{proof}

We can now use Lemma \ref{Lemma:eventually-distinct} to prove Proposition %
\ref{Proposition:size-ultraproduct}.

\begin{proof}[Proof of Proposition \protect\ref%
{Proposition:size-ultraproduct}]
Let $\mathcal{F}$ be a family of size continuum consisting of pairwise
eventually distinct functions $f:\mathbb{N}\rightarrow \mathbb{N}$ such that 
$f\left( n\right) \leq n$ for every $n\in \mathbb{N}$ and $%
\lim_{n\rightarrow +\infty }f\left( n\right) =+\infty $. For $n\in \mathbb{N}
$, let $\bar{a}^{(n)}=(a_{1}^{(n)},\ldots ,a_{n}^{(n)})$ be the tuple of
elements of $D^{M_{n}}$ given by hypothesis. Let $M$ be the ultraproduct $%
\prod_{\mathcal{U}}M_{n}$. For every $f\in \mathcal{F}$ define $\boldsymbol{a%
}^{f}$ to be the element of $D^{M}$ with representative sequence $%
(a_{f\left( n\right) }^{(n)})$. The assumption that, for every finite set $%
\mathfrak{t}_{0}\left( x\right) $ of conditions in $\mathfrak{t}\left(
x\right) ^{+}$, every element of $\bar{a}^{(n)}$ satisfies $\mathfrak{t}%
_{0}\left( x\right) $ for all but finitely many $n\in \mathbb{N}$, and the
fact that $\lim_{n\rightarrow +\infty }f\left( n\right) =+\infty $, implies
by \L os' theorem that $\boldsymbol{a}^{f}$ satisfies $\mathfrak{t}\left(
x\right) $.

If $f$ and $g$ are different elements of $\mathcal{F}$, then they are
eventually distinct. In particular, since $\mathcal{U}$ is nonprincipal, $%
\left( \mathcal{U}n\right) $, $f\left( n\right) \neq g\left( n\right) $.
Hence, $\left( \mathcal{U}n\right) $, $d(a_{f\left( n\right)
}^{n},a_{g\left( n\right) }^{n})\geq \delta $. By \L os' theorem again, we
then have $d\left( \boldsymbol{a}^{f},\boldsymbol{a}^{g}\right) \geq \delta $%
. This concludes the proof.
\end{proof}

Using Proposition \ref{Proposition:size-ultraproduct} we can give a
sufficient condition for a C*-algebra $A$ to have, for any nonprincipal
ultrafilter $\mathcal{U}$ over $\mathbb{N}$, relative commutant in the
ultrapower $A^{\prime }\cap A^{\mathcal{U}}$ of density character at least $%
\mathfrak{c}$. For convenience, we isolate the following notion.

\begin{definition}
\label{Definition:many-central}A C*-algebra $A$ has \emph{many
asymptotically central elements} if there exists $\delta >0$ such that, for
every $n\in \mathbb{N}$, finite subset $F$ of $A$, and $\varepsilon >0$,
there exist $a_{1},\ldots ,a_{n}$ in $A$ of norm at most $1$ such that, for
every $1\leq i<j\leq n$ and $b\in F$, $\left\Vert a_{i}-a_{j}\right\Vert
\geq \delta $ and $\left\Vert a_{i}b-ba_{i}\right\Vert \leq \varepsilon $.
\end{definition}

Thus Proposition \ref{Proposition:size-ultraproduct} gives the following.

\begin{proposition}
\label{Proposition:many-central}Suppose that $A$ is a separable C*-algebra
that has many asymptotically central elements, and $\mathcal{U}$ is a
nonprincipal ultrafilter over $\mathbb{N}$. Then $A^{\prime }\cap A^{%
\mathcal{U}}$ and $A^{\mathcal{U}}$ both have density character and
cardinality equal to $\mathfrak{c}$.
\end{proposition}

\begin{proof}
We have already observed above that $A^{\mathcal{U}}$ has density character
at most $\mathfrak{c}$. In order to see that $A^{\prime }\cap A^{\mathcal{U}%
} $ has density character at least $\mathfrak{c}$, one can apply Proposition %
\ref{Proposition:size-ultraproduct} to the language $L^{\text{C*}}\left(
A\right) $, the sequence $\left( M_{n}\right) _{n\in \mathbb{N}}$ constantly
equal to $A$, and the $\emph{relative}$ \emph{commutant type }$\mathfrak{t}%
\left( x\right) $ consisting of the conditions $\left\Vert xa-ax\right\Vert
\leq 0$ for $a\in A$ of norm at most $1$.

This shows that $A^{\prime }\cap A^{\mathcal{U}}$ and $A^{\mathcal{U}}$ both
have density character $\mathfrak{c}$. Fix a dense subset $E$ of $A^{%
\mathcal{U}}$ of size $\mathfrak{c}$. Observe that, since $E$ is dense, the
cardinality of $A^{\mathcal{U}}$ is bounded by the cardinality of the set of 
\emph{sequences }of elements of $E$, which is still equal to $\mathfrak{c}$.
Therefore $A^{\mathcal{U}}$ has cardinality at most $\mathfrak{c}$,
concluding the proof.
\end{proof}

Due to Proposition \ref{Proposition:many-central}, our original question on
the size of ultrapowers and relative commutants leads us to consider which
C*-algebras have many asymptotically central elements. As it turns out, 
\emph{every }infinite-dimensional C*-algebra has many asymptotically central
elements, as we will show below. We begin with the \emph{abelian case}.
Recall that we are assuming all C*-algebras to be unital.

\begin{lemma}
\label{Lemma:many-abelian}Suppose that $A$ is an $\emph{abelian}$
infinite-dimensional C*-algebra. Then $A$ has many asymptotically central
elements.
\end{lemma}

\begin{proof}
Since $A$ is abelian, $A$ is isomorphic to the algebra $C\left( X\right) $
of continuous complex-valued function over a compact Hausdorff space $X$.
Since $A$ is infinite-dimensional, $X$ is not finite. Thus, by normality of $%
X$ we can find, for every $n\in \mathbb{N}$, positive elements $a_{1},\ldots
,a_{n}$ of $C\left( X\right) $ of norm $1$ with disjoint supports. (The
support of a function $a:X\rightarrow \left[ 0,1\right] $ is the set $%
\left\{ t\in X:a\left( t\right) \neq 0\right\} $.) This guarantees that $%
\left\Vert a_{i}-a_{j}\right\Vert =1$ for every $1\leq i<j\leq n$. This
concludes the proof.
\end{proof}

We now consider the class of \emph{continuous trace }C*-algebras \cite%
{raeburn_morita_1998}. These are the C*-algebras that can be realized as the
algebras of sections of a bundle over a compact Hausdorff space $X$, with
finite-dimensional C*-algebras as fibers. Clearly, commutative C*-algebras
correspond to the case of bundles with $1$-dimensional fibers.

\begin{lemma}
\label{Lemma:many-continuous-trace}Suppose that $A$ is an
infinite-dimensional continuous trace C*-algebra. Then $A$ has many
asymptotically central elements.
\end{lemma}

\begin{proof}
As noted above, $A$ is the algebra of sections of a bundle of over a compact
Hausdorff space $X$, with finite-dimensional C*-algebras as fibers. Since $A$
is infinite-dimensional, $X$ is infinite. Clearly, the abelian C*-algebra $%
C\left( X\right) $ is isomorphic to a C*-subalgebra of the \emph{center} $%
A^{\prime }\cap A$ of $A$. This implies that, since $C\left( X\right) $ has
many asymptotically central elements, so does $A$.
\end{proof}

We consider now the case when $A$ does \emph{not }have continuous trace.

\begin{lemma}
\label{Lemma:many-not-continuous-trace}Let $A$ be a separable C*-algebra
that does not have continuous trace. Then $A$ has many asymptotically
central elements.
\end{lemma}

\begin{proof}
By \cite[Theorem 2.4]{akemann_central_1979}, there exists a sequence $\left(
d_{n}\right) $ of positive elements of norm $1$ in $A$ such that, for every $%
a\in A$, $\lim_{n\rightarrow +\infty }\left\Vert d_{n}a-ad_{n}\right\Vert =0$
and $\delta _{a}:=\limsup_{n\rightarrow +\infty }\left\Vert
d_{n}-a\right\Vert >0$. Consider then the $L^{\text{C*}}\left( A\right) $%
-type $\mathfrak{t}\left( x\right) $ consisting of the conditions $%
\left\Vert xa-ax\right\Vert \leq 0$ and $\left\Vert x-a\right\Vert \geq
\delta _{a}$ for $a\in A$ and $\left\vert 1-\left\Vert x\right\Vert
\right\vert \leq 0$. Then the type $\mathfrak{t}\left( x\right) $ is
approximately realized in $A$, and hence it is realized in $A^{\mathcal{U}}$
by countable saturation. We identify $A$ with a subalgebra of $A^{\prime
}\cap A^{\mathcal{U}}$. If $\boldsymbol{a}$ is a realization of $\mathfrak{t}%
\left( x\right) $ in $A^{\mathcal{U}}$, then $\boldsymbol{a}$ is an element
of norm $1$ of $A^{\prime }\cap A^{\mathcal{U}}$ that does not belong to $A$%
. By completeness of $A$, there exists $\delta >0$ such that $\left\Vert 
\boldsymbol{a}-b\right\Vert >\delta $ for every $b\in A$ of norm $1$.

Fix an enumeration $\left( b_{n}\right) $ of a dense subset of the unit ball
of $A$. If $\left( a_{n}\right) $ is a representative sequence of $%
\boldsymbol{a}$, then using \L os' theorem one can recursively define an
increasing sequence $\left( n_{k}\right) $ in $\mathbb{N}$ such that, for
every $k<m$, $\left\Vert a_{k}-a_{m}\right\Vert >\delta $ and $\left\Vert
b_{m}a_{k}-a_{k}b_{m}\right\Vert <2^{-k}$. This shows that $A$ has many
asymptotically central elements.
\end{proof}

Lemma \ref{Lemma:many-continuous-trace} and Lemma \ref%
{Lemma:many-not-continuous-trace} together show that every
infinite-dimensional C*-algebra has many asymptotically central elements in
the sense of Definition \ref{Definition:many-central}. (Obviously, the
converse holds as well.) Combining this with Proposition \ref%
{Proposition:many-central} and Theorem \ref{Theorem:isomorphic-saturated},
we finally obtain the following.

\begin{theorem}
\label{Theorem:ultrapowers}Suppose that $A$ is a separable
infinite-dimensional C*-algebra. If $\mathcal{U}$ is a nonprincipal
ultrafilter over $\mathbb{N}$, then the ultrapower $A^{\mathcal{U}}$ and the
relative commutant $A^{\prime }\cap A^{\mathcal{U}}$ have density character $%
\mathfrak{c}$. If the Continuum Hypothesis holds, and $\mathcal{U}$, $%
\mathcal{V}$ are nonprincipal ultrafilters over $\mathbb{N}$, then $A^{%
\mathcal{U}}\cong A^{\mathcal{V}}$ and $A^{\prime }\cap A^{\mathcal{U}}\cong
A^{\prime }\cap A^{\mathcal{V}}$.
\end{theorem}

\begin{proof}
The first assertion is an immediate consequence of Proposition \ref%
{Proposition:many-central} and the observations above. Suppose now that the
Continuum Hypothesis holds. If $\mathcal{U}$, $\mathcal{V}$ are nonprincipal
ultrafilters over $\mathbb{N}$, then $A^{\mathcal{U}}$ and $A^{\mathcal{V}}$
have density character $\aleph _{1}$, they are elementarily equivalent by \L %
os' theorem, and they are countable saturated by Proposition \ref%
{Proposition:countably-saturated}. Therefore Theorem \ref%
{Theorem:isomorphic-saturated} implies that $A^{\mathcal{U}}$ and $A^{%
\mathcal{V}}$ are isomorphic. Furthermore, since the diagonal embedding of $%
A $ in both $A^{\mathcal{U}}$ and $A^{\mathcal{V}}$ is elementary by \L os'
theorem, again by Theorem \ref{Theorem:isomorphic-saturated} there exists an
isomorphism $\Phi :A^{\mathcal{U}}\rightarrow A^{\mathcal{V}}$ which is the
identity on $A$ (canonically identified with a C*-subalgebra of $A^{\mathcal{%
U}}$ and $A^{\mathcal{V}}$). Henceforth, $\Phi $ restricts to an isomorphism
from $A^{\prime }\cap A^{\mathcal{U}}$ onto $A^{\prime }\cap A^{\mathcal{V}}$%
. This concludes the proof.
\end{proof}

\section{Strongly self-absorbing C*-algebras\label{Section:strongly}}

The class of strongly self-absorbing C*-algebras, initially introduced by
Toms and Winter in \cite{toms_strongly_2007}, has played in recent years a
pivotal role in the study of structure and classification of simple nuclear
C*-algebras. In the rest of this section, we want to present some
model-theoretic results concerning these algebras, their ultrapowers and
relative commutants, obtained in \cite{farah_relative_2017}.

Recall that a separable C*-algebra $D$ has \emph{approximately inner
half-flip} if the canonical embeddings $D\rightarrow D\otimes D$ defined by $%
d\mapsto d\otimes 1_{D}$ and $d\mapsto 1_{D}\otimes d$ are approximately
unitarily equivalent. (Here and in the following, we consider the \emph{%
minimal }tensor product of C*-algebras; see \cite[Section II.9]%
{blackadar_operator_2006}.) In other words, there exists a sequence $\left(
u_{n}\right) $ of unitary elements of $D\otimes D$ such that $\left\Vert
u_{n}\left( d\otimes 1_{D}\right) -\left( d\otimes 1_{D}\right)
u_{n}\right\Vert \rightarrow 0$ for $n\rightarrow +\infty $ for every $d\in
D $.\ A separable C*-algebra $D$ is \emph{strongly self-absorbing} if it is
not isomorphic to $\mathbb{C}$, and the canonical embedding \textrm{id}$%
_{D}\otimes 1:D\rightarrow D\otimes D$, $d\mapsto d\otimes 1_{D}$ is
approximately unitarily equivalent to an isomorphism $D\cong D\otimes D$.
This condition is very restrictive, indeed a strongly self-absorbing
C*-algebra $D$ is automatically simple and nuclear, has approximately
inner-half flip, and it is isomorphic to the infinite tensor product $%
D^{\otimes \mathbb{N}}$. The only currently known examples of strongly
self-absorbing C*-algebras are the infinite type uniformly hyperfinite (UHF)
C*-algebras \cite{glimm_certain_1960}, the Cuntz algebras $\mathcal{O}_{2}$
and $\mathcal{O}_{\infty }$ \cite{cuntz_simple_1977}, the Jiang--Su algebra $%
\mathcal{Z}$ \cite{jiang_simple_1999}, and their tensor products. In the
following, all the strongly self-absorbing C*-algebras are assumed to be
separable.

Suppose that $D$ is a strongly self-absorbing C*-algebra, and $A$ is a
separable C*-algebra. One says that $A$ is tensorially $D$-absorbing, or
simply $D$-absorbing, if $D\otimes A$ is isomorphic to $A$. The notion of $D$%
-absorption plays a crucial role in the current study of nuclear
C*-algebras. The goal of the next subsection is to present a proof of a
well-known characterization of $D$-absorption for separable C*-algebras in
terms of the notion of positive existential embedding.

\subsection{A criterion for $D$-absorption}

Throughout this section, we let $D$ be a fixed strongly self-absorbing
C*-algebra, and $A,B$ be separable C*-algebras. In the following lemma, we
consider objects which are triples $(A,\hat{A},\eta )$ of two C*-algebras $A,%
\hat{A}$ together with an embedding $\eta :A\rightarrow \hat{A}$. These can
be regarded as structures in a two-sorted language $\mathcal{L}$ which has
sorts for the C*-algebras $A,\hat{A}$, function and relation symbols for the
C*-algebra structure on $A,\hat{A}$, and a function symbol for the embedding 
$\eta :A\rightarrow \hat{A}$.

\begin{lemma}
\label{Lemma:pairs}Suppose that $A,B$ are C*-algebras and $\Phi
:A\rightarrow B$ is a positively existential embedding. If $C$ is a nuclear
C*-algebra, then $\mathrm{id}_{C}\otimes \Phi :C\otimes A\rightarrow
C\otimes B$, $c\otimes a\mapsto c\otimes \Phi \left( a\right) $ is a
positively existential embedding. Furthermore, the pair $\left( \Phi ,%
\mathrm{id}_{C}\otimes \Phi \right) $ defines a positive existential
embedding from $\left( A,C\otimes A,1_{C}\otimes \mathrm{id}_{A}\right) $ to 
$\left( B,C\otimes B,1_{C}\otimes \mathrm{id}_{B}\right) $ regarded as $%
\mathcal{L}$-structures.
\end{lemma}

\begin{proof}
We will use below that, since $C$ is nuclear, maximal and minimal tensor
products with $C$ coincide.

Fix a countably incomplete ultrafilter $\mathcal{U}$. Since $\Phi $ is a
positively existential embedding, there exists a morphism $\Psi
:B\rightarrow A^{\mathcal{U}}$ such that $\Psi \circ \Phi =\Delta _{A}$. One
can then consider \textrm{id}$_{C}\otimes \Psi :C\otimes B\rightarrow
C\otimes A^{\mathcal{U}}$. Since $C$ is nuclear, the tensor product $%
C\otimes A^{\mathcal{U}}$ can be identified with the maximal tensor product.
Let $\eta :C\otimes A^{\mathcal{U}}\rightarrow \left( C\otimes A\right) ^{%
\mathcal{U}}$ be the canonical morphism obtained via the universal property
of the maximal tensor product from the morphisms with commuting ranges $%
\Delta _{C}:C\rightarrow C^{\mathcal{U}}\subset \left( C\otimes A\right) ^{%
\mathcal{U}}$ and $\mathrm{id}_{A^{\mathcal{U}}}:A^{\mathcal{U}}\rightarrow
A^{\mathcal{U}}\subset \left( C\otimes A\right) ^{\mathcal{U}}$. Observe
that 
\begin{equation*}
\eta \circ \left( \mathrm{id}_{C}\otimes \Psi \right) \circ \left( \mathrm{id%
}_{C}\otimes \Phi \right) =\eta \circ \left( \mathrm{id}_{C}\otimes \Delta
_{A}\right) =\mathrm{id}_{C\otimes A}\text{.}
\end{equation*}%
Thus $\eta \circ \left( \mathrm{id}_{C}\otimes \Psi \right) $ witnesses that 
$\mathrm{id}_{C}\otimes \Phi $ is positively existential.

Consider now $\left( A,C\otimes A,\mathrm{id}_{C}\otimes A\right) $ and $%
\left( B,C\otimes B,\mathrm{id}_{C}\otimes B\right) $ as $\mathcal{L}$%
-structures, and observe that the pair $\left( \Phi ,\mathrm{id}_{C}\otimes
\Phi \right) $ defines an $\mathcal{L}$-morphism between such $\mathcal{L}$%
-structures. Observe also that $\left( A^{\mathcal{U}},\left( C\otimes
A\right) ^{\mathcal{U}},\left( 1_{C}\otimes \mathrm{id}_{A}\right) ^{%
\mathcal{U}}\right) $ can be regarded as the ultrapower of $\left(
A,C\otimes A,1_{C}\otimes \mathrm{id}_{A}\right) $ as an $\mathcal{L}$%
-structure. Furthermore, the pair $\left( \Psi ,\eta \circ \left( \mathrm{id}%
_{C}\otimes \Psi \right) \right) $ defines an $\mathcal{L}$-morphism from $%
\left( B,C\otimes B,1_{C}\otimes \mathrm{id}_{B}\right) $ to $\left( A^{%
\mathcal{U}},\left( C\otimes A\right) ^{\mathcal{U}},\left( 1_{C}\otimes 
\mathrm{id}_{A}\right) ^{\mathcal{U}}\right) $ whose composition with $%
\left( \Phi ,\mathrm{id}_{C}\otimes \Phi \right) $ is the canonical
embedding of $\left( A,C\otimes A,\mathrm{id}_{C}\otimes A\right) $ into its
ultrapower. This witnesses that the pair $\left( \Phi ,\mathrm{id}%
_{C}\otimes \Phi \right) $ is a positively existential embedding from $%
\left( A,C\otimes A,1_{C}\otimes \mathrm{id}_{A}\right) $ to $\left(
B,C\otimes B,1_{C}\otimes \mathrm{id}_{B}\right) $ regarded as $\mathcal{L}$%
-structures.
\end{proof}

The following fundamental intertwining lemma is \cite[Proposition 2.3.5]%
{rordam_classification_2002}.

\begin{lemma}
\label{Lemma:intertwining}Suppose that $\Phi :A\rightarrow B$ is an
embedding. Assume that for every $\varepsilon >0$ and for every finite
subset $F_{A}$ of $A$ and $F_{B}$ of $B$ there exists a unitary element $z$
of $B$ such that

\begin{enumerate}
\item $\left\Vert z\Phi \left( a\right) z^{\ast }-\Phi \left( a\right)
\right\Vert \leq \varepsilon $ for $a\in F_{A}$, and

\item $\inf_{x\in A}\left\Vert z^{\ast }bz-\Phi \left( x\right) \right\Vert
\leq \varepsilon $ for $b\in F_{B}$.
\end{enumerate}

Then $\Phi $ is approximately unitarily equivalent to an isomorphism $\Psi $.
\end{lemma}

Using Lemma \ref{Lemma:intertwining} one can obtain the following.

\begin{lemma}
\label{Lemma:existential-absorb}If the canonical embedding $1_{D}\otimes 
\mathrm{id}_{A}:A\rightarrow D\otimes A$ is positively existential, then it
is approximately unitarily equivalent to an isomorphism.
\end{lemma}

\begin{proof}
Fix a finite subset $F$ of $A$, a finite subset $F^{\prime }$ of $D\otimes A$%
, and $\varepsilon >0$. Consider the structure 
\begin{equation*}
\mathcal{A}=\left( A,D\otimes A,1_{D}\otimes \mathrm{id}_{A}:A\rightarrow
D\otimes A\right)
\end{equation*}%
and the structure%
\begin{equation*}
\mathcal{B}=\left( D\otimes A,D\otimes D\otimes A,1_{D}\otimes \mathrm{id}%
_{D\otimes A}:D\otimes A\rightarrow D\otimes D\otimes A\right) \text{.}
\end{equation*}%
One can naturally consider $\mathcal{A}$ and $\mathcal{B}$ as $\mathcal{L}$%
-structures, where $\mathcal{L}$ is the multi-sorted language considered in
Lemma \ref{Lemma:pairs}. We are assuming that the embedding%
\begin{equation*}
1_{D}\otimes \mathrm{id}_{A}:A\rightarrow D\otimes A
\end{equation*}%
is positively existential. Therefore by Lemma \ref{Lemma:pairs} the embedding%
\textrm{\ id}$_{D}\otimes 1_{D}\otimes \mathrm{id}_{A}:D\otimes A\rightarrow
D\otimes D\otimes A$ is positively existential. Furthermore, the pair $\Psi
=\left( 1_{D}\otimes \mathrm{id}_{A},\mathrm{id}_{D}\otimes 1_{D}\otimes 
\mathrm{id}_{A}\right) $ defines an existential embedding from $\mathcal{A}$
to $\mathcal{B}$ regarded as $\mathcal{L}$-structures.

Since $D$ has approximately inner half-flip, there exists a unitary element $%
v$ of $D\otimes D$ such that, if $\hat{u}:=v\otimes 1\in D\otimes D\otimes A$%
, then

\begin{enumerate}
\item $\left\Vert \hat{u}\left( 1_{D}\otimes 1_{D}\otimes a\right) \hat{u}%
^{\ast }-\left( 1_{D}\otimes 1_{D}\otimes a\right) \right\Vert <\varepsilon $
for $a\in F$, and

\item $\mathrm{dist}\left( \hat{u}^{\ast }\left( \mathrm{id}_{D}\otimes
1_{D}\otimes \mathrm{id}_{A}\right) \left( b\right) \hat{u},1_{D}\otimes
D\otimes A\right) <\varepsilon $ for $b\in F^{\prime }$.
\end{enumerate}

Using the fact that the $\mathcal{L}$-morphism $\Psi :\mathcal{A}\rightarrow 
\mathcal{B}$ is positively $\mathcal{L}$-existential and that the unitary
group is a positively existentially definable set, one can conclude that
there exists a unitary element $u$ of $D\otimes A$ such that

\begin{enumerate}
\item $\left\Vert u\left( 1_{D}\otimes a\right) u^{\ast }-\left(
1_{D}\otimes a\right) \right\Vert <\varepsilon $ for $a\in F$, and

\item \textrm{dist}$\left( u^{\ast }bu,1_{D}\otimes A\right) <\varepsilon $
for $b\in F^{\prime }$.
\end{enumerate}

This witnesses that $1_{D}\otimes \mathrm{id}_{A}:\left( A,\alpha \right)
\rightarrow \left( D\otimes A,\mathrm{id}_{D}\otimes \alpha \right) $
satisfies the assumptions of Lemma \ref{Lemma:intertwining}.
\end{proof}

\begin{lemma}
\label{Lemma:absorb-existential}The embedding 
\begin{equation*}
1_{D}\otimes \mathrm{id}_{D\otimes A}:D\otimes A\rightarrow D\otimes
D\otimes A
\end{equation*}%
is positively existential.
\end{lemma}

\begin{proof}
Since $D$ is strongly self-absorbing, it is enough to show that the embedding%
\begin{equation*}
1_{D}\otimes \mathrm{id}_{D^{\otimes \mathbb{N}}\otimes A}:D^{\otimes 
\mathbb{N}}\otimes A\rightarrow D\otimes D^{\otimes \mathbb{N}}\otimes A
\end{equation*}%
is positively existential. For every $n\in \mathbb{N}$, let%
\begin{equation*}
\Psi _{n}:D\rightarrow D^{\otimes \mathbb{N}}\otimes A
\end{equation*}%
be the embedding induced by the embedding of $D$ into $D^{\otimes \mathbb{N}%
} $ as $n$-th tensor factor. Suppose that $\varphi \left( x_{1},\ldots
,x_{\ell }\right) $ is a positive existential $L^{\text{C*}}\left( A\right) $%
-formula, $\varepsilon >0$, $r\in \mathbb{R}$, and consider the condition $%
\varphi \left( x_{1},\ldots ,x_{\ell }\right) \leq r$. Assume that $%
\sum_{k}d_{i,k}\otimes a_{i,k}$ for $i=1,2,\ldots ,\ell $ is an $\ell $%
-tuple in $D\otimes D^{\otimes \mathbb{N}}\otimes A$ satisfying the
condition $\varphi \left( x_{1},\ldots ,x_{\ell }\right) <r$, where $%
d_{i,k}\in D$ and $a_{i,k}\in D^{\otimes \mathbb{N}}\otimes A$. Then, for $%
n\in \mathbb{N}$ large enough, $\sum_{k}\Psi _{n}\left( d_{i,k}\right)
a_{i,k}$ for $i=1,2,\ldots ,\ell $ is an $\ell $-tuple in $D^{\otimes 
\mathbb{N}}\otimes A$ satisfying the condition $\varphi \left( x_{1},\ldots
,x_{\ell }\right) \leq r+\varepsilon $. This concludes the proof that $%
1_{D}\otimes \mathrm{id}_{D^{\otimes \mathbb{N}}\otimes A}$ is a positively
existential embedding.
\end{proof}

Using the results above, we can give a characterization of $D$-absorption.

\begin{theorem}
\label{Theorem:absorption}Suppose that $A$ is a separable C*-algebra, and $D$
is a separable strongly self-absorbing C*-algebra. Let $\mathcal{U}$ be a
countably incomplete ultrafilter. The following statements are equivalent:

\begin{enumerate}
\item $A$ is $D$-absorbing;

\item the embedding $\mathrm{id}_{A}\otimes 1_{D}:A\rightarrow A\otimes D$
is positively existential;

\item the embedding $\mathrm{id}_{A}\otimes 1_{D}:A\rightarrow A\otimes D$
is approximately unitarily equivalent to an isomorphisms;

\item $D$ embeds into $A^{\prime }\cap A^{\mathcal{U}}$;

\item if $\mathfrak{t}\left( \bar{x}\right) $ is a positive quantifier-free
type approximately realized in $D$, then the type $\mathfrak{t}\left( \bar{x}%
\right) \cup \left\{ \left\Vert x_{i}a-ax_{i}\right\Vert \leq 0:a\in
A\right\} $ is approximately realized in $A$.
\end{enumerate}
\end{theorem}

\begin{proof}
The implication (1)$\Rightarrow $(2) is a consequence of Lemma \ref%
{Lemma:absorb-existential}, while the implication (2)$\Rightarrow $(3) is a
consequence of Lemma \ref{Lemma:existential-absorb}. The equivalence (2)$%
\Leftrightarrow $(4) is a consequence of Lemma \ref%
{Proposition:commutant-existential}, observing that $D$ is simple so a
morphism $D\rightarrow A^{\prime }\cap A^{\mathcal{U}}$ is necessarily an
embedding, while the equivalence (4)$\Leftrightarrow $(5) follows from \L %
os' theorem and countable saturation of ultrapowers. Finally, the
implication (3)$\Rightarrow $(1) is obvious.
\end{proof}

Theorem \ref{Theorem:absorption} motivates the following definition.

\begin{definition}
\label{Definition:D-absorbing}A (not necessarily separable) C*-algebra $A$
is $D$-absorbing if it satisfies Condition (5) in Theorem \ref%
{Theorem:absorption}.
\end{definition}

In view of Theorem \ref{Theorem:absorption}, this definition is consistent
with the usual one in the case of separable C*-algebras. Furthermore, when $%
A $ is positively quantifier-free saturated, $A$ is $D$-absorbing if and
only if, for every separable subalgebra $S$ of $A$, $D$ embeds into $%
S^{\prime }\cap A$. This shows that $S^{\prime }\cap D$ is still $D$%
-absorbing and positively quantifier-free saturated

It is clear from the definition that the property of being $D$-absorbing is
axiomatizable. Indeed, this is witnessed by the conditions%
\begin{equation*}
\sup_{x_{1},\ldots ,x_{n}}\inf_{y_{1},\ldots ,y_{n}}\max \left\{ \varphi
\left( x_{1},\ldots ,x_{n}\right) ,\left\Vert
x_{i}y_{j}-y_{j}x_{i}\right\Vert :1\leq i,j\leq n\right\} \leq 0
\end{equation*}%
where $\varphi \left( \bar{x}\right) $ varies among the positive
quantifier-free formulas for which the condition $\varphi \left( \bar{x}%
\right) \leq 0$ is realized in $D$.

\begin{definition}
\label{Definition:sup-inf}A (positive) $\sup \inf $-formula is a formula of
the form $\sup_{\bar{x}}\inf_{\bar{y}}\psi \left( \bar{x},\bar{y}\right) $
where $\psi \left( \bar{x},\bar{y}\right) $ is a (positive) quantifier-free
formula. A class is (positively) $\sup \inf $-axiomatizable if it is
axiomatizable, as witnessed by conditions of the form $\varphi \leq r$,
where $\varphi $ is a (positive) $\sup \inf $-formula.
\end{definition}

The above argument shows that the class of $D$-absorbing C*-algebras is $%
\sup \inf $-axiomatizable.

\subsection{Relative commutants of $D$-absorbing C*-algebras}

Suppose as above that $D$ is a strongly self-absorbing C*-algebra. Let now $%
C $ be a (not necessarily separable) $D$-absorbing C*-algebra, in the sense
just defined. Recall that two morphisms $\Phi _{1},\Phi _{2}:A\rightarrow B$
are unitarily equivalent if there exists a unitary element $u$ of $B$ such
that $\Phi _{2}=\mathrm{Ad}\left( u\right) \circ \Phi _{1}$. In the
following theorem, we denote by $\aleph _{1}$ the first uncountable cardinal
number.

\begin{theorem}
\label{Theorem:ssa}Let $D$ be a strongly self-absorbing C*-algebra. Suppose
that $\theta :D\rightarrow C$ is an embedding, and that $C$ is $D$-absorbing
and positively quantifier-free countably saturated. The following assertions
hold:

\begin{enumerate}
\item Any two embeddings $D\rightarrow C$ are unitarily equivalent;

\item For every separable C*-subalgebra $A$ of $\theta \left( D\right)
^{\prime }\cap C$ and every separable C*-subalgebra $B$ of $C$ there exists
a unitary $u\in A^{\prime }\cap C$ such that $uBu^{\ast }\subset \theta
\left( D\right) ^{\prime }\cap C$;

\item $\theta \left( D\right) ^{\prime }\cap C$ is an elementary
substructure of $C$;

\item If $C$ has density character $\aleph _{1}$, then the inclusion $\theta
\left( D\right) ^{\prime }\cap C\hookrightarrow C$ is approximately
unitarily equivalent to an isomorphism;

\item If $C$ is countably saturated, then $\theta \left( D\right) ^{\prime
}\cap C$ is countably saturated.
\end{enumerate}
\end{theorem}

\begin{proof}
Since $C$ is $D$-absorbing and positively quantifier-free saturated, for
every separable C*-subalgebra $S$ of $C$ one has that $S^{\prime }\cap C$ is
also $D$-absorbing and positively quantifier-free saturated.

(1): Let $\theta _{1}:D\rightarrow C$ be an embedding. We will show that $%
\theta _{1}$ is unitarily equivalent to $\theta $. Let us initially assume
that the ranges of $\theta $ and $\theta _{1}$ commute. We can choose a
sequence $\left( u_{n}\right) $ of unitaries in $D$ witnessing the fact that 
$D$ has approximately inner half-flip. Define $\Theta :D\otimes D\rightarrow
C$ by $d_{0}\otimes d_{1}\mapsto \theta \left( d_{0}\right) \theta
_{1}\left( d_{1}\right) $. Considering the unitaries $\Theta \left(
u_{n}\right) $ for $n\in \mathbb{N}$ and applying the fact that $C$ is
positively quantifier-free countably saturated, we conclude that there
exists a unitary $u\in A$ such that $u\theta (d)=u\Theta (d\otimes 1)=\Theta
(1\otimes d)u=\theta _{1}(d)u$. This shows that $\theta ,\theta _{1}$ are
unitarily equivalent.

In the general case, when the ranges of $\theta ,\theta _{1}$ do not
necessarily commute, since $D$ is strongly self-absorbing and $C$ is
positively quantifier-free countably saturated, we can find an embedding $%
\theta _{2}:D\rightarrow C$ whose range commutes with the ranges of $\theta $
and $\theta _{1}$.\ Therefore by the above we have that $\theta $ is
unitarily equivalent to $\theta _{2}$ and $\theta _{1}$ is unitarily
equivalent to $\theta _{2}$. Hence, $\theta $ is unitarily equivalent to $%
\theta _{1}$.

For the rest of the proof, we identify $D$ with its image under $\theta $
inside $C$.

(2): Fix a separable C*-subalgebra $A$ of $D^{\prime }\cap C$ and a
separable C*-subalgebra $B$ of $C$.\ We want to show that there exists $u\in
A^{\prime }\cap C$ such that $uBu^{\ast }\subset D^{\prime }\cap C$.

Since $C$ is $D$-absorbing and positively quantifier-free saturated, the
same holds for $A^{\prime }\cap C$ and $A^{\prime }\cap B^{\prime }\cap C$.
Thus we can fix an embedding $\Psi :D\rightarrow A^{\prime }\cap B^{\prime
}\cap C$. By (1) applied to the pair of embeddings $D\rightarrow A^{\prime
}\cap C$ given by $\Psi $ and the inclusion $\iota :D\subset A^{\prime }\cap
C$, there exists a unitary $u\in A^{\prime }\cap C$ such that $\Psi =\mathrm{%
Ad}\left( u^{\ast }\right) \circ \iota $ or, equivalently, $u^{\ast }du=\Psi
(d)$ for every $d\in D$. Therefore, if $b\in B$ and $d\in D$, we have that%
\begin{equation*}
\left\Vert ubu^{\ast }d-dubu^{\ast }\right\Vert =\left\Vert bu^{\ast
}du-u^{\ast }dub\right\Vert =\left\Vert b\Psi (d)-\Psi (d)b\right\Vert =0%
\text{.}
\end{equation*}%
This concludes the proof.

(3): By the Tarski--Vaught test (Proposition \ref{Proposition:Tarski-Vaught}%
), it suffices to show that if $r\in \mathbb{R}$, $\varphi \left( \bar{x},%
\bar{y}\right) $ is a formula, $\bar{a}$ is a tuple in $D^{\prime }\cap C$,
and $\bar{b}$ is a tuple in $C$ such that $\varphi ^{C}\left( \bar{a},\bar{b}%
\right) <r$, then there exists a tuple $\bar{d}$ in $D^{\prime }\cap C$ such
that $\varphi ^{C}\left( \bar{a},\bar{d}\right) <r$. Let $A$ be the
C*-subalgebra of $D^{\prime }\cap C$ generated by $\bar{a}$, and let $B$ be
the C*-subalgebra of $C$ generated by $\bar{b}$. Then by (2) there exists a
unitary $u\in A^{\prime }\cap C$ such that $uBu^{\ast }\subset D^{\prime
}\cap C$. Thus we have that $\bar{d}:=u\bar{b}u^{\ast }$ is a tuple in $%
D^{\prime }\cap C$ such that%
\begin{equation*}
\varphi ^{C}\left( \bar{a},\bar{d}\right) =\varphi ^{C}\left( u\bar{a}%
u^{\ast },u\bar{b}u^{\ast }\right) =\varphi ^{C}\left( \bar{a},\bar{b}%
\right) <r\text{,}
\end{equation*}%
concluding the proof.

(4): Suppose now that $C$ has density character $\aleph _{1}$. Thus we can
fix an enumeration $\left( b_{\lambda }\right) _{\lambda <\omega _{1}}$ of a
subset of $C$ such that, for every $\mu <\lambda $, $\left\{ b_{\lambda
}:\mu <\lambda <\omega _{1}\right\} $ is a dense subset of $C$. Similarly,
we can fix an enumeration $\left( a_{\lambda }\right) _{\lambda <\omega
_{1}} $ of a subset of $D^{\prime }\cap C$ such that, for every $\mu
<\lambda $, $\left\{ a_{\lambda }:\mu <\lambda <\omega _{1}\right\} $ is a
dense subset of $D^{\prime }\cap C$. Here, $\omega _{1}$ denotes the first
uncountable ordinal. We will define by transfinite recursion elements $\hat{a%
}_{\lambda } $ of $D^{\prime }\cap C$, elements $\hat{b}_{\lambda }$ of $C$,
and unitaries $u_{\lambda }$ in $C$ for $\lambda <\omega _{1}$ such that,
for every $\mu \leq \lambda <\omega _{1}$,%
\begin{equation*}
u_{\lambda }b_{\mu }u_{\lambda }^{\ast }=\hat{a}_{\mu }\text{\quad and\quad }%
u_{\lambda }\hat{b}_{\mu }u_{\lambda }^{\ast }=a_{\mu }\text{.}
\end{equation*}%
Granted the construction, we can define the map $\Phi :\left\{ a_{\mu }:\mu
<\omega _{1}\right\} \rightarrow C$ by setting $\Phi \left( a_{\mu }\right) =%
\hat{b}_{\mu }$. Then by construction, $\Phi $ is approximately unitarily
equivalent to the inclusion map of $\left\{ a_{\mu }:\mu <\omega
_{1}\right\} $ inside $C$. Therefore $\Phi $ extends to a *-homomorphism $%
\Phi :D^{\prime }\cap C\rightarrow C$, which is approximately unitarily
equivalent to the inclusion map $D^{\prime }\cap C\rightarrow C$. It remains
to show that $\Phi $ is onto. Fix $\varepsilon >0$ and $\mu <\omega _{1}$.
Then there exists an ordinal $\mu <\lambda <\omega _{1}$ such that $%
\left\Vert \hat{a}_{\mu }-a_{\lambda }\right\Vert <\varepsilon $. Thus we
have that%
\begin{equation*}
\left\Vert \Phi \left( \hat{a}_{\mu }\right) -b_{\mu }\right\Vert \leq
\varepsilon +\left\Vert \Phi \left( a_{\lambda }\right) -u_{\lambda }^{\ast }%
\hat{a}_{\mu }u_{\lambda }\right\Vert \leq 2\varepsilon +\left\Vert \hat{b}%
_{\lambda }-u_{\lambda }^{\ast }a_{\lambda }u_{\lambda }\right\Vert
=2\varepsilon \text{.}
\end{equation*}%
Since this holds for every $\mu <\omega _{1}$ and for every $\varepsilon >0$%
, this shows that $\Phi $ is onto.

It remains to describe the recursive construction. Suppose that the elements 
$\hat{a}_{\mu },\hat{b}_{\mu },u_{\mu }$ have been constructed for $\mu
<\lambda $. Define $A$ to be the separable C*-subalgebra of $D^{\prime }\cap
C$ generated by $\hat{a}_{\mu },a_{\mu }$ for $\mu <\lambda $. Let also $B$
be the separable C*-subalgebra of $C$ generated by $a_{\mu },b_{\mu },\hat{a}%
_{\mu },\hat{b}_{\mu }$ for $\mu <\lambda $.

Suppose initially that $\lambda $ is a successor ordinal. We let $\lambda -1$
be the immediate predecessor of $\lambda $. Then by (2) there exists a
unitary $v\in A^{\prime }\cap C$ such that $vu_{\lambda -1}b_{\lambda
}u_{\lambda -1}^{\ast }v^{\ast }\in D^{\prime }\cap C$. Set then $u_{\lambda
}:=vu_{\lambda -1}$, $\hat{a}_{\lambda }:=u_{\lambda }b_{\lambda }u_{\lambda
}^{\ast }\in D^{\prime }\cap C$ and $\hat{b}_{\lambda }=u_{\lambda }^{\ast
}a_{\lambda }u_{\lambda }$. Now we have that, by definition%
\begin{equation*}
u_{\lambda }b_{\lambda }u_{\lambda }^{\ast }=\hat{a}_{\lambda }\text{\quad
and\quad }u_{\lambda }\hat{b}_{\lambda }u_{\lambda }^{\ast }=a_{\lambda }%
\text{.}
\end{equation*}%
For $\mu <\lambda $ we have that, since $v\in A^{\prime }\cap C$,%
\begin{equation*}
u_{\lambda }b_{\mu }u_{\lambda }^{\ast }=vu_{\lambda -1}b_{\mu }u_{\lambda
-1}^{\ast }v^{\ast }=v\hat{a}_{\mu }v^{\ast }=\hat{a}_{\mu }
\end{equation*}%
and%
\begin{equation*}
u_{\lambda }\hat{b}_{\mu }u_{\lambda }^{\ast }=vu_{\lambda -1}\hat{b}_{\mu
}u_{\lambda -1}^{\ast }v^{\ast }=va_{\mu }v^{\ast }=a_{\mu }\text{.}
\end{equation*}%
This concludes the construction in the case when $\lambda $ is a successor
ordinal. When $\lambda $ is a limit ordinal, one can obtain $u_{\lambda }$
by applying positive quantifier-free countable saturation of $C$ to the
positive quantifier-free $L^{\text{C*}}\left( B\right) $-type $\mathfrak{t}%
\left( x\right) $ consisting of the conditions $\left\Vert x^{\ast
}x-1\right\Vert \leq 0$, $\left\Vert xx^{\ast }-1\right\Vert \leq 0$, $%
\left\Vert xb_{\mu }x^{\ast }-\hat{a}_{\mu +1}\right\Vert \leq 0$ for $\mu
<\lambda $, and $\left\Vert x\hat{b}_{\mu }x^{\ast }-a_{\mu }\right\Vert
\leq 0$ for $\mu <\lambda $. Such a type is approximately realized in $C$ by
the inductive hypothesis that $u_{\mu }$ has been defined for $\mu <\lambda $%
. Therefore such a type is realized in $C$. One can then let $u_{\lambda }$
be any realization of $\mathfrak{t}\left( x\right) $. This concludes the
recursive construction.

(5): Suppose that $C$ is countably saturated. Suppose that $A$ is a
separable C*-subalgebra of $D^{\prime }\cap C$, and $\mathfrak{t}\left( \bar{%
x}\right) $ is an $L^{\text{C*}}\left( A\right) $-type which is
approximately realized in $D^{\prime }\cap C$. Then $\mathfrak{t}\left( \bar{%
x}\right) $ is also approximately realized in $C$. Since by assumption $C$
is countably saturated, $\mathfrak{t}\left( \bar{x}\right) $ has a
realization $\bar{b}$ in $C$. By (2) there exists a unitary $u\in A^{\prime
}\cap C$ such that $u\bar{b}u^{\ast }\subset D^{\prime }\cap C$. Since $u\in
A^{\prime }\cap C$, the map $\mathrm{Ad}\left( u\right) $ is an $L^{\text{C*}%
}\left( A\right) $-automorphism of $C$. For every condition $\varphi \left( 
\bar{x}\right) \leq r$ in $\mathfrak{t}\left( \bar{x}\right) $ we have that $%
\varphi ^{C}\left( \bar{b}\right) \leq r$, and hence $\varphi ^{C}\left( u%
\bar{b}u^{\ast }\right) \leq r$. Since $ubu^{\ast }\in D^{\prime }\cap C$
and $D^{\prime }\cap C$ is an elementary substructure of $C$, we have that $%
\varphi ^{D^{\prime }\cap C}\left( u\bar{b}u^{\ast }\right) \leq r$. Since
this holds for every condition in $\mathfrak{t}\left( \bar{x}\right) $, $u%
\bar{b}u^{\ast }$ is a realization of $\mathfrak{t}\left( \bar{x}\right) $
in $D^{\prime }\cap C$. This concludes the proof that $D^{\prime }\cap C$ is
countably saturated.
\end{proof}

\begin{corollary}
Let $D$ be a strongly self-absorbing C*-algebra, and $A$ is a separable $D$%
-absorbing C*-algebra. If $\mathcal{U}$ is a countably incomplete
ultrafilter, then the conclusions of Theorem \ref{Theorem:ssa} holds for $%
C=A^{\mathcal{U}}$ and for $C=A^{\prime }\cap A^{\mathcal{U}}$. In
particular, if the Continuum Hypothesis holds, and if $D\subset A$ is a copy
of $D$ inside $A$, then $D^{\prime }\cap A^{\mathcal{U}}$ and $A^{\mathcal{U}%
}$ are isomorphic.
\end{corollary}

\begin{proof}
It suffices to observe that the algebras $A^{\mathcal{U}}$ and $A^{\prime
}\cap A^{\mathcal{U}}$ are $D$-absorbing and positively quantifier-free
saturated. The ultrapower $A^{\mathcal{U}}$ has density character $\mathfrak{%
c}$ by \ref{Theorem:ultrapowers}. Since CH is the assertion that $\mathfrak{c%
}$ is the first uncountable cardinal $\aleph _{1}$, CH implies that $A^{%
\mathcal{U}}$ has density character $\aleph _{1}$. Thus the last assertion
is a consequence of item (4) of Theorem \ref{Theorem:ssa}.
\end{proof}

\subsection{Strongly self-absorbing C*-algebras are existentially closed}

Suppose that $\mathcal{C}$ is a class of C*-algebras. Then a C*-algebra $A$
is existentially closed in $\mathcal{C}$ if $A$ belongs to $\mathcal{C}$ and
any embedding $\Phi :A\rightarrow B$ of $A$ into a C*-algebra $B$ from $%
\mathcal{C}$, $\Phi $ is existential. This means that if $\varphi \left( 
\bar{x}\right) $ is an existential formula and $\bar{a}$ is a tuple in $A$,
then $\varphi ^{A}\left( \bar{a}\right) =\varphi ^{B}\left( \Phi \left( \bar{%
a}\right) \right) $.

Let now $D$ be a strongly self-absorbing C*-algebra, and consider the class $%
\mathcal{C}_{D}$ of C*-algebras that embed into an ultrapower of $D$.
Observe that $\mathcal{C}_{D}$ is an elementary class.\ Indeed, a C*-algebra 
$A$ belongs to $\mathcal{C}_{D}$ if and only if $\varphi ^{A}\leq \varphi
^{D}$ for every universal sentence $\varphi $. When $D$ is the infinite type
UHF algebra $\bigotimes_{n\in \mathbb{N}}M_{n}\left( \mathbb{C}\right)
^{\otimes \mathbb{N}}$, $\mathcal{C}_{D}$ is the class of MF algebras \cite%
{blackadar_generalized_1997,carrion_groups_2013}. The Kirchberg embedding
conjecture asserts that, when $D$ is the Cuntz algebra $\mathcal{O}_{2}$, $%
\mathcal{C}_{D}$ contains all C*-algebras \cite{goldbring_kirchbergs_2015}.

\begin{proposition}
Let $D$ be a strongly self-absorbing C*-algebra. Then $D$ is existentially
closed in $\mathcal{C}_{D}$.
\end{proposition}

\begin{proof}
Suppose that $D\subset B$ where $B\in \mathcal{C}_{D}$. Then we can assume
that $B\subset D^{\mathcal{U}}$ for some countably incomplete ultrafilter $%
\mathcal{U}$. Since all the embeddings of $D$ into $D^{\mathcal{U}}$ are
unitarily conjugate, we can assume that the composition of the inclusions $%
D\subset B$ and $B\subset D^{\mathcal{U}}$ is the diagonal embedding of $D$
into $D^{\mathcal{U}}$.

Fix a quantifier-free formula $\varphi \left( \bar{x},\bar{y}\right) $. If $%
\bar{a}$ is a tuple in $D$ and $\bar{b}$ is a tuple in $B\subset D^{\mathcal{%
U}}$ such that $\varphi \left( \bar{a},\bar{b}\right) <r$, by \L os' theorem
there exists a tuple $\bar{d}$ in $D$ such that $\varphi \left( \bar{a},\bar{%
d}\right) <r$. This shows that the inclusion $D\subset B$ is existential.
\end{proof}

\section{Model theory and nuclear C*-algebras\label{Section:nuclear}}

\subsection{The classification programme\label{Subsection:classification}}

Much of the theory of C*-algebras in the last twenty years has focused on
the structure and classification of simple, separable nuclear C*-algebra, in
the framework of the Elliott classification programme. \emph{Nuclearity},
also called \emph{amenability}, is a regularity property for C*-algebras
with several equivalent reformulations. It can naturally be defined in the
broader category of operator systems and unital completely positive maps.\ A
linear map $\Phi :A\rightarrow B$ between C*-algebras is \emph{unital
completely positive }(ucp) if $\Phi \left( 1\right) =1$ and, for every $n\in 
\mathbb{N}$, $\mathrm{id}_{M_{n}\left( \mathbb{C}\right) }\otimes \Phi
:M_{n}\left( \mathbb{C}\right) \otimes A\rightarrow M_{n}\left( \mathbb{C}%
\right) \otimes B$ maps positive elements to positive elements. A C*-algebra
is \emph{nuclear }if the identity map of $A$ is the pointwise limit of ucp
maps of the form $\Psi \circ \Phi :A\rightarrow A$, where $\Phi
:A\rightarrow M_{n}\left( \mathbb{C}\right) $ and $\Psi :M_{n}\left( \mathbb{%
C}\right) \rightarrow A$ are ucp maps, and $n\in \mathbb{N}$. This
definition admits several equivalent reformulations, including prominently
the following: for any C*-algebra $B$, the maximal and the minimal tensor
product norms on the algebraic tensor product of $A$ and $B$ coincide. Thus,
the class of C*-algebras is endowed with a single canonical C*-algebraic
tensor product.

The first hints that a satisfactory classification of separable, nuclear
C*-algebras could be achieved goes back to the seminal works of Glimm \cite%
{glimm_certain_1960} and Elliott \cite{elliott_classification_1976}, who
classified those separable, nuclear C*-algebras that can be realized as
direct limits of finite-dimensional C*-algebras.\ In the modern perspective,
the invariant used in these results is the $K_{0}$-group. This is a
countable ordered abelian group with a distinguished order unit. Together
with the $K_{1}$-group, it constitutes the $K$-theory of a given C*-algebra.
Originally developed in algebraic geometry \cite{atiyah_riemann-roch_1959}, $%
K$-theory was then translated into purely algebraic language \cite%
{rosenberg_algebraic_1994}, and then incorporated in the theory of
C*-algebras \cite{schochet_algebraic_1994}.

The class of arbitrary separable, nuclear C*-algebras is extensive, in that
it contains all the algebras of the form $C\left( X\right) $ where $X$ is a
compact metrizable space. For such algebras, isomorphism coincides with
homeomorphism of the corresponding space. Since a meaningful
classificationof arbitrary compact metrizable spaces is considered out of
reach, it is natural to impose restrictions on the class of algebras under
consideration to rule out complexity arising from the purely topological
setting. One such a natural assumption consists in demanding that the
C*-algebras under consideration be simple, i.e.\ have no nontrivial ideals.

Building on the results of Glimm and Elliott mentioned above, broader
classes of simple, separable C*-algebras have been classified in the 1980s
and 1990s due to the work of many hands. In this case, the invariant
consisted on the K-theory together with the \emph{trace simplex}. A \emph{%
tracial state }on a C*-algebra is a unital linear functional $\tau $
satisfying the trace identity $\tau \left( xy\right) =\tau \left( yx\right) $%
.\ The space $T\left( A\right) $ of tracial states over a separable
C*-algebra $A$ forms a Choquet simplex, which is called the trace simplex.
Traces can be seen as a noncommutative analog of measures, and so the trace
simplex encodes the measure-theoretic information on the given C*-algebras.\
The conjunction of the K-theory of a C*-algebra together with its trace
simplex and a canonical pairing between them, is called \emph{Elliott
invariant}.

Motivated by these positive results, Elliott proposed the programme---known
as the Elliott classification programme---of classifying simple, separable,
nuclear C*-algebras by their Elliott invariant \cite%
{elliott_classification_1995}. Despite substantial progress, examples due to
R\o rdam and Toms showed that, in general, the Elliott invariant is not a
complete invariant for simple, separable, nuclear C*-algebras. In \cite%
{toms_infinite_2008,toms_comparison_2009}, nonisomorphic simple separable
C*-algebras with the same Elliott invariant have been constructed.\ The
invariant used to distinguish such algebras is the \emph{radius of comparison%
}. Interestingly, such an invariant is captured by the first-order theory of
a C*-algebra \cite[Section 8.4]{farah_model_2017}. In fact, no example of
not elementarily equivalent simple, separable, nuclear C*-algebras with the
same Elliott invariant is currently known. This motivated the following
problem, asked in \cite{farah_model_2017}.

\begin{problem}
Is the Elliott invariant \emph{together with the first-order theory} a
complete invariant for simple, separable, nuclear C*-algebras?
\end{problem}

In a different direction, the counterexamples due R\o rdam and Toms have
suggested to restrict the scope of the Elliott classification programme to a
suitable class of \textquotedblleft well-behaved\textquotedblright\ simple,
separable, nuclear C*-algebras. One interpretation of what well-behaved
should mean for a simple, separable, nuclear C*-algebra is to absorb
tensorially the Jiang-Su algebra $\mathcal{Z}$. This is a strongly
self-absorbing C*-algebra, and in fact the unique strongly self-absorbing
C*-algebra that embeds into any other self-absorbing C*-algebra. The
property of $\mathcal{Z}$-absorption is conjecturally equivalent for simple,
separable, nuclear C*-algebras to other regularity properties of very
different nature (topological, cohomological). This is the subject of the 
\emph{Toms--Winter conjecture}, which has been so far verified in many
cases, prominently including the case when the trace simplex is itself
well-behaved (it has closed extreme boundary).

The revised Elliott classification programme for $\mathcal{Z}$-\emph{%
absorbing }simple, separable, nuclear C*-algebras has recently been
completed, due to the work of many authors, modulo the standing assumption
that the algebras considered satisfy the Universal Coefficient Theorem
(UCT). A technical statement regarding the relation between different
K-theoretic invariants (K-theory and KK-theory), the UCT is an assumption
(sometimes automatically satisfied) in all the positive classification
results for separable, nuclear C*-algebras that have been obtained so far.
At the same time, no example of separable, nuclear C*-algebra that does 
\emph{not }satisfy the UCT is currently known. This has brought considerable
interest to the following \emph{UCT problem}.

\begin{problem}
Does every separable, nuclear C*-algebra satisfy the UCT?
\end{problem}

The UCT has been verified for several important classes of C*-algebras. In
fact, it holds for all the separable C*-algebras that can be obtained
starting from finite-dimensional and abelian C*-algebras by the standard
constructions of C*-algebra theory, such as inductive limits and crossed
products by amenable groups. Again, the problem of whether such a class
contains in fact all separable, nuclear C*-algebras is currently open. It is
clear that the answer to this problem is intimately connected with the
problem of finding other methods of constructing nuclear C*-algebras, other
than the standard constructions in C*-algebra theory. This provides a
connection with model theory, which is a source of a different kind of
constructions, generally known as model-theoretic forcing or building models
by games.

Another connection of the UCT problem and model theory arises from the
following reformulation due to Kirchberg: is the Cuntz algebra $\mathcal{O}%
_{2}$ uniquely determined by its K-theory among simple, separable, purely
infinite C*-algebras? This reformulation can be seen as the problem of
whether $\mathcal{O}_{2}$ is the unique model of its theory that omits a
certain collection of types. Model theory provides criteria (omitting types
theorems) that allow one to decide when a theory admits a model omitting
certain types. It is thus clear that these results may be particularly
relevant to this problem.

The first natural step towards the possible application of methods from
model theory to such problems consists in clarifying the model-theoretic
content of notions such as nuclearity, $\mathcal{Z}$-stability, and the
Elliott invariant.

\subsection{Nuclearity is not elementary\label{Subsection:nonelementary}}

It turns out that the property of being nuclear for C*-algebras is \emph{not 
}axiomatizable. Also the more generous property of exactness is not
axiomatizable. A C*-algebra is \emph{exact} if, roughly speaking, it can be
approximately represented---as an operator space---into full matrix
algebras. A deep result of Kirchberg asserts that a separable C*-algebra is
exact if and only if it embeds into a nuclear C*-algebra, which can be
chosen to be the Cuntz algebra $\mathcal{O}_{2}$. An arbitrary C*-algebra is
exact if and only if all its separable C*-subalgebras are exact. Any nuclear
C*-algebra is, in particular, exact. It is also important to notice that
exactness (or nuclearity) of a C*-algebra only depends on the underlying
operator system structure, and it is inherited by passing to operator
subsystems. This has the following implications. For C*-algebras $A,B$, a 
\emph{complete order embedding }from $A$ to $B$ is a ucp map $\Phi
:A\rightarrow B$ with a ucp inverse $\Psi :\Phi \left( A\right) \rightarrow
A $. If $A,B$ are C*-algebras, $\Phi :A\rightarrow B$ is a complete order
embedding, and $B$ is exact, then $A$ is exact.

\begin{lemma}
\label{Lemma:elementary-exact}Suppose that $A$ is a C*-algebra. If every
C*-algebra elementarily equivalent to $A$ is exact, then $A$ is $n$%
-subhomogeneous for some $n\in \mathbb{N}$.
\end{lemma}

\begin{proof}
Suppose that, for every $n\in \mathbb{N}$, $A$ is \emph{not }$n$%
-subhomogeneous. Thus, for every $n\in \mathbb{N}$, $A$ has an irreducible
representation on a Hilbert space of dimension at least $n$. This allows one
to find, for every $n\in \mathbb{N}$, a (not necessarily unital) subalgebra $%
B_{n}$ of $A$ and an ideal $J_{n}$ of $B_{n}$ such that the quotient $%
B_{n}/J_{n}$ is isomorphic to $M_{n}\left( \mathbb{C}\right) $. Thus, if $%
\mathcal{U}$ is a nonprincipal ultrafilter over $\mathbb{N}$, then $A^{%
\mathcal{U}}$ has $\prod_{\mathcal{U}}B_{n}$ as a C*-subalgebra.
Furthermore, $\prod_{\mathcal{U}}J_{n}$ is a closed two-sided ideal of $%
\prod_{\mathcal{U}}B_{n}$, and the corresponding quotient can be identified
with $\prod_{\mathcal{U}}(B_{n}/J_{n})\cong \prod_{\mathcal{U}}M_{n}\left( 
\mathbb{C}\right) $.

If $H$ is a separable infinite-dimensional Hilbert space, one can consider
an increasing sequence of projections $p_{n}\in B\left( \mathbb{C}\right) $
with $\mathrm{rank}\left( p_{n}\right) =n$ such that $p_{n}\rightarrow 1$ in
the strong operator topology. Then, for $n\in \mathbb{N}$, $p_{n}B\left(
H\right) p_{n}\cong M_{n}\left( \mathbb{C}\right) $. Furthermore the map $%
B\left( H\right) \rightarrow \prod_{\mathcal{U}}\left( p_{n}B\left( H\right)
p_{n}\right) \cong \prod_{\mathcal{U}}M_{n}\left( \mathbb{C}\right) $
defined by $a\mapsto \left[ \left( p_{n}ap_{n}\right) \right] $ is a
complete order embedding. Since $B\left( H\right) $ is not exact, $\prod_{%
\mathcal{U}}M_{n}\left( \mathbb{C}\right) $ is not exact either. Since
exactness passes to quotients, this implies that $\prod_{\mathcal{U}}B_{n}$
is not exact, and since exactness passes to subalgebras, $A^{\mathcal{U}}$
is not exact.
\end{proof}

The converse of Lemma \ref{Lemma:elementary-exact} holds as well. Indeed, if 
$A$ is $n$-subhomogeneous for some $n\in \mathbb{N}$, since the class of $n$%
-subhomogeneous C*-algebras is axiomatizable, any C*-algebra elementarily
equivalent to $A$ is $n$-subhomogeneous, and in particular nuclear, and
exact.

It follows from Lemma \ref{Lemma:elementary-exact} that, if $\mathcal{C}$ is
an elementary class of C*-algebras, and $\mathcal{C}$ only consists of exact
C*-algebras, then in fact $\mathcal{C}$ only consists of $n$-subhomogeneous
C*-algebras for some $n\in \mathbb{N}$. In particular, the class of exact
C*-algebras and the class of nuclear C*-algebras are \emph{not }elementary.
The same conclusions apply to several other classes of C*-algebras that are
important for the classification programme, such as:

\begin{itemize}
\item the class of C*-algebras that can be locally approximated by full
matrix algebras, known as uniformly hyperfinite (UHF) algebras;

\item the class of C*-algebras that can be locally approximated by
finite-dimensional C*-algebras, known as approximately finite-dimensional
(AF) algebras.
\end{itemize}

Similarly, the class of \emph{simple }C*-algebras is not elementary. Indeed,
for every $n\in \mathbb{N}$, $M_{n}\left( \mathbb{C}\right) $ is simple.
However, if $\mathcal{U}$ is a nonprincipal ultrafilter over $\mathbb{N}$,
then the ultraproduct $\prod_{\mathcal{U}}M_{n}\left( \mathbb{C}\right) $ is
not simple: the set of elements $\boldsymbol{a}$ of $\prod_{\mathcal{U}%
}M_{n}\left( \mathbb{C}\right) $ with representative sequence satisfying $%
\lim_{n\rightarrow \mathcal{U}}\left\Vert x_{n}\right\Vert _{2}=0$, where $%
\left\Vert \cdot \right\Vert _{2}$ denotes the normalized Hilbert--Schmidt
norm, is a nontrivial closed two-sided ideal of $\prod_{\mathcal{U}%
}M_{n}\left( \mathbb{C}\right) $.

\subsection{Infinitary formulas\label{Subsection:infinitary}}

In order to capture properties such as nuclearity, one needs to consider a
more generous notion of \textquotedblleft formula\textquotedblright . We
will therefore introduce \emph{infinitary }formulas, where one is allowed to
take countably infinite \textquotedblleft conjunctions and
disjunctions\textquotedblright , which in this setting are expressions of
the form $\sup_{n\in \mathbb{N}}\varphi _{n}$ and $\inf_{n\in \mathbb{N}%
}\varphi _{n}$ where $\left( \varphi _{n}\right) _{n\in \mathbb{N}}$ is a
sequence of formulas subject to certain restrictions. Notice that this
construction is not allowed in our previous definition of formulas, in which
case one is only allowed to take infima and suprema over a $\emph{variable}$%
. In particular, it is important to keep in mind that infinitary formulas
are, in general, not formulas in the strict sense that we have considered so
far. To avoid confusions, the formulas as we have previously defined are
also called, for completeness, \emph{finitary }or \emph{first-order }%
formulas. The same adjectives should be applied to the notions we have
defined in terms of first-order formulas, such as axiomatizable classes,
definable predicates, and so on.

Let us now introduce formally infinitary formulas in an arbitrary language $%
L $ as in the framework considered in Subsection \ref{Subsection:domains}.
In fact, we will only consider a special case of infinitary formulas, which
we call $\sup \bigvee \inf $-formulas following \cite%
{goldbring_enforceable_2017}. Recall that, if $\varphi \left( \overline{z}%
\right) $ is a finitary formula (or, more generally, a definable predicate)
and $M$ is an $L$-structure, then the interpretation $\varphi ^{M}$ of $%
\varphi $ in $M$ is a uniformly continuous function.\ Furthermore, one can
find uniform continuity modulus for $\varphi ^{M}$ which is independent of $%
M $, and can be explicitly computed in terms of the uniform continuity of
the function and relation symbols in the language and their bounds. We refer
to this as the continuity modulus $\varpi ^{\varphi }$ of $\varphi $. An
(infinitary) $\sup \bigvee \inf $-formula is an expression $\varphi \left( 
\bar{x}\right) $ of the form%
\begin{equation*}
\sup_{\bar{y}\in \bar{D}}\inf_{n\in \mathbb{N}}\psi _{n}\left( \bar{x},\bar{y%
}\right)
\end{equation*}%
where $\bar{x}$ is a tuple of variables with corresponding domains $\bar{D}$%
, and $\left( \psi _{n}\left( \bar{x},\bar{y}\right) \right) _{n\in \mathbb{N%
}}$ is a sequence of \emph{existential }$L$-formulas (or existential
definable predicates) such that the function $\varpi ^{\varphi }\left( \bar{s%
},\bar{r}\right) :=\sup_{n\in \mathbb{N}}\min \left\{ \varpi ^{\psi
_{n}}\left( \bar{s},\bar{r}\right) ,1\right\} $ tuples $\bar{s},\bar{r}$ of
positive real numbers satisfies $\varpi ^{\varphi }\left( \bar{s},\bar{r}%
\right) \rightarrow 0$ for $\bar{s}\rightarrow 0$ and $\bar{r}\rightarrow 0$%
. Given an $L$-structure $M$, one can define the interpretation $\varphi
^{M} $ of $\varphi $ in $M$ in the obvious way. The requirement on the
continuity moduli guarantees that $\varphi ^{M}$ is a uniformly continuous
function with modulus $\varpi ^{\varphi }$ (independent of $M$). It is
important to notice that the analog of \L os' theorem does \emph{not }hold
in general for $\sup \bigvee \inf $-formulas. Observe also that any
(finitary) $\sup \inf $-formula is, in particular, a $\sup \bigvee \inf $%
-formula. A $\sup \bigvee \inf $-sentence will be a $\sup \bigvee \inf $%
-formula with no free variables. \emph{Positive }$\sup \bigvee \inf $%
-formulas are defined as above, but starting from \emph{positive }%
quantifier-free formulas (or definable predicates).

If now $\mathcal{C}$ is a class of $L$-structure, then we say that $\mathcal{%
C}$ admits an infinitary $\sup \bigvee \inf $-axiomatization if there exists
a \emph{countable }collection of conditions\emph{\ }of the form $\varphi
\leq r$, where $\varphi $ is a $\sup \bigvee \inf $-sentence, such that an $%
L $-structure $M$ belongs to $\mathcal{C}$ if and only if $\varphi ^{M}\leq
r $ for every such a condition.

\subsection{Infinitary axiomatization of nuclearity\label%
{Subsection:nuclearity}}

We now remark how important classes of C*-algebras admit an infinitary $\sup
\bigvee \inf $-axiomatization.

\subsubsection{UHF algebras}

Recall that a C*-algebra $A$ is UHF if and only if for every tuple $\bar{a}$
in the unit ball of $A$ and $\varepsilon >0$ there exist $n\in \mathbb{N}$
and a unital copy $M_{n}\left( \mathbb{C}\right) \subset A$ such that every
element of $\bar{a}$ is at distance at most $\varepsilon $ from the unit
ball of $M_{n}\left( \mathbb{C}\right) $. Recall from Subsection \ref%
{Subsection:quantifier-free} that the set of matrix units for a unital copy
of $M_{n}\left( \mathbb{C}\right) $ is definable, being the zeroset of the
stable formula $\varphi _{M_{n}\left( \mathbb{C}\right) }\left(
z_{ij}\right) $ given by%
\begin{equation*}
\max \left\{ \left\Vert z_{ij}z_{k\ell }-\delta _{jk}z_{i\ell }\right\Vert
,\left\Vert 1-\sum_{j=1}^{n}z_{jj}\right\Vert ,\left\vert 1-\left\Vert
z_{ij}\right\Vert \right\vert ,\left\Vert x_{ij}-x_{ji}^{\ast }\right\Vert
:1\leq i,j,k,\ell \leq n\right\} \text{.}
\end{equation*}%
Recall also that the set $\mathbb{D}$ of scalar multiples of the identity of
norm at most $1$ (identified with the set of complex number of absolute
value at most $1$) is also definable. One can thus consider, given $\ell
,n\in \mathbb{N}$, the existential definable predicate $\psi _{n}\left(
z_{1},\ldots ,z_{\ell }\right) $ given by%
\begin{equation*}
\inf_{e_{ij}}\inf_{\lambda _{ij}^{(m)}}\max \left\{ \left\Vert
z_{m}-\sum_{i,j=1}^{n}\lambda _{ij}^{(m)}e_{ij}\right\Vert :m=1,2,\ldots
,\ell \right\}
\end{equation*}%
where $e_{ij}$ for $1\leq i,j\leq n$ range among the matrix units for a
unital copy of $M_{n}\left( \mathbb{C}\right) $ and $\lambda _{ij}^{(m)}$
for $1\leq i,j\leq n$ and $1\leq m\leq \ell $ range among the complex
numbers of absolute value at most $1$. It is clear that the continuity
modulus $\varpi ^{\psi _{n}}$ of $\psi _{n}$ satisfies $\varpi ^{\psi
_{n}}\left( t_{1},\ldots ,t_{n}\right) \leq \max \left\{ t_{1},\ldots
,t_{\ell }\right\} $. Therefore we can consider the infinitary $\sup \bigvee
\inf $-sentence $\varphi _{\ell }$ given by 
\begin{equation*}
\sup_{z_{1},\ldots ,z_{\ell }\in D_{1}}\inf_{n\in \mathbb{N}}\psi _{n}\left(
z_{1},\ldots ,z_{\ell }\right) \text{.}
\end{equation*}%
It is clear that the $\sup \bigvee \inf $-conditions $\varphi _{\ell }\leq 0$
for $\ell \in \mathbb{N}$ indeed provide an infinitary $\sup \bigvee \inf $%
-axiomatization for the class of UHF algebras.

\subsubsection{AF algebras}

The treatment of AF algebras is entirely analogous. Indeed, a C*-algebra $A$
is AF if and only if for every tuple $\bar{a}$ in the unit ball of $A$ and $%
\varepsilon >0$ there exists a finite-dimensional unital C*-subalgebra $%
F\subset A$ such that every element of $\bar{a}$ is at distance at most $%
\varepsilon $ from the unit ball of $F$. One can then consider $\sup \bigvee
\inf $-conditions defined as above, where one replaced full matrix algebras
with arbitrary finite-dimensional C*-algebras.

\subsubsection{Nuclear algebras}

The proof in the case of nuclearity is similar, but slightly more delicate.
Recall that a C*-algebra $A$ is nuclear if and only if for every finite
tuple $\bar{a}=\left( a_{1},\ldots ,a_{\ell }\right) $ in the unit ball of $%
A $ and $\varepsilon >0$ there exist $n\in \mathbb{N}$ and ucp maps $\Phi
:A\rightarrow M_{n}\left( \mathbb{C}\right) $ and $\Psi :M_{n}\left( \mathbb{%
C}\right) \rightarrow A$ such that $\left\Vert \left( \Psi \circ \Phi
\right) \left( a_{i}\right) -a_{i}\right\Vert <\varepsilon $ for $%
i=1,2,\ldots ,\ell $. We have been identifying the field $\mathbb{C}$ of
scalars with the set of scalar multiples of the unit of $A$. Similarly, we
can canonically identify $M_{n}\left( \mathbb{C}\right) $ with the
subalgebra $M_{n}\left( \mathbb{C}\right) \otimes 1\subset M_{n}\left( 
\mathbb{C}\right) \otimes A\cong M_{n}\left( A\right) $.

Consider the relation $R$ on $M_{n}\left( \mathbb{C}\right) \times A$
defined, for $\alpha \in M_{n}\left( \mathbb{C}\right) $ and $a\in A$ of
norm at most $1$, by%
\begin{equation*}
R\left( \alpha ,a\right) =\inf_{\Psi }\left\Vert \Psi \left( \alpha \right)
-a\right\Vert \text{.}
\end{equation*}%
Here $\Psi $ ranges among all the ucp maps $\Psi :M_{n}\left( \mathbb{C}%
\right) \rightarrow A$. A key step consists in showing that $R$ given by a
quantifier-free definable predicate. To see this, one should observe that a
linear map $\Psi :M_{n}\left( \mathbb{C}\right) \rightarrow A$ is ucp if and
only if $\Psi \left( 1\right) =1$, and the matrix%
\begin{equation*}
\sum_{ij}e_{ij}\otimes \Psi \left( e_{ij}\right) \in M_{n}\left( \mathbb{C}%
\right) \otimes A
\end{equation*}%
is positive. As we have shown the set of positive elements of $A$ is
quantifier-free definable, and the same holds for the set of positive
elements of $M_{n}\left( \mathbb{C}\right) \otimes A$. Thus the predicate $R$
is quantifier-free definable as well.

On then needs to show that the relation on $A\times M_{n}\left( \mathbb{C}%
\right) $ defined by%
\begin{equation*}
S\left( a,\alpha \right) =\inf_{\Phi }\left\Vert \Phi \left( a\right)
-\alpha \right\Vert \text{,}
\end{equation*}%
where $\Phi $ range among all the ucp maps $\Phi :A\rightarrow M_{n}\left( 
\mathbb{C}\right) $, is given by a quantifier-free definable predicate. This
can be shown via the correspondence between completely positive maps $\Phi
:A\rightarrow M_{n}\left( \mathbb{C}\right) $ and positive linear
functionals $s_{\Phi }$ on $M_{n}\left( \mathbb{C}\right) \otimes A$ such
that, for every $a\in A$,%
\begin{equation*}
\frac{1}{n}\Phi \left( a\right) =\sum_{ij}s_{\Phi }\left( e_{ij}\otimes
a\right) e_{ij}\text{,}
\end{equation*}%
where $\left( e_{ij}\right) $ denote the matrix units of $M_{n}\left( 
\mathbb{C}\right) $ \cite[Proposition 5.8.5]{farah_model_2017}.

Using these facts, one can then consider the existential definable predicate 
$\psi _{n}\left( z_{1},\ldots ,z_{\ell }\right) $ given by%
\begin{equation*}
\inf_{\alpha _{1},\ldots ,\alpha _{\ell }}\max_{m}\left\{ S\left(
z_{m},\alpha _{m}\right) ,R\left( \alpha _{m},z_{m}\right) :m=1,2,\ldots
,\ell \right\}
\end{equation*}%
where $\alpha _{1},\ldots ,\alpha _{m}$ range within the unit ball of $%
M_{n}\left( \mathbb{C}\right) $. Finally, we have that 
\begin{equation*}
\sup_{z_{1},\ldots ,z_{\ell }\in D_{1}}\inf_{n}\psi _{n}\left( z_{1},\ldots
,z_{\ell }\right)
\end{equation*}%
is a $\sup \bigvee \inf $-sentence $\varphi _{\ell }$. The conditions $%
\varphi _{\ell }\leq 0$ for $\ell \in \mathbb{N}$ witness that the class of
nuclear C*-algebras admits an infinitary $\sup \bigvee \inf $-axiomatization.

\subsubsection{Simple C*-algebras}

We conclude by showing that the class of simple C*-algebras admits an
infinitary $\sup \bigvee \inf $-axiomatization. Recall that a C*-algebra is
simple if it has no nontrivial closed two-sided ideal. We will use the
following characterization of simplicity: a C*-algebra is simple $A$ if and
only if for every positive elements $a,d\in A$ such that $\left\Vert
a\right\Vert =1$ and $\left\Vert d\right\Vert \leq 1/2$ there exist $n\in 
\mathbb{N}$ and $c_{1},\ldots ,c_{n}\in A$ such that $c_{1}^{\ast
}c_{1}+\cdots +c_{n}^{\ast }c_{n}$ is a contraction, and 
\begin{equation*}
\left\Vert c_{1}^{\ast }ac_{1}+\cdots +c_{n}^{\ast }ac_{n}-d\right\Vert
<\varepsilon \text{;}
\end{equation*}%
see \cite[Lemma 5.10.1]{farah_model_2017}.

Fix $n\in \mathbb{N}$, and observe that the set of $n$-tuples $\bar{c}%
=\left( c_{1},\ldots ,c_{n}\right) $ such that $c_{1}^{\ast }c_{1}+\cdots
+c_{n}^{\ast }c_{n}$ is a contraction, is quantifier-free definable. Indeed,
the norm of $c_{1}^{\ast }c_{1}+\cdots +c_{n}^{\ast }c_{n}$ is equal to the
norm of the matrix column vector%
\begin{equation*}
C:=%
\begin{bmatrix}
c_{1} \\ 
c_{2} \\ 
\vdots \\ 
c_{n}%
\end{bmatrix}%
\end{equation*}%
which can be identified with the element%
\begin{equation*}
\begin{bmatrix}
c_{1} & 0 & 0 & \cdots & 0 \\ 
c_{2} & 0 & 0 & \cdots & 0 \\ 
c_{3} & 0 & 0 & \cdots & 0 \\ 
\vdots & \vdots & \vdots & \ddots & \vdots \\ 
c_{n} & 0 & 0 & \cdots & 0%
\end{bmatrix}%
\end{equation*}%
of $M_{n}\left( A\right) $. Thus we can consider for $n\in \mathbb{N}$ the
quantifier-free definable predicate $\psi _{n}\left( x,y\right) $ given by%
\begin{equation*}
\inf_{z_{1},\ldots ,z_{n}}\max \left\{ \left\Vert z_{1}^{\ast }\left(
x^{\ast }x\right) z_{1}+\cdots +z_{n}^{\ast }\left( x^{\ast }x\right)
z_{n}-y^{\ast }y\right\Vert ,\left\Vert y\right\Vert -2^{-1/2},\left\vert
1-\left\Vert x\right\Vert \right\vert \right\}
\end{equation*}%
where $z_{1},\ldots ,z_{n}$ range among all the $n$-tuples such that $%
z_{1}^{\ast }z_{1}+\cdots +z_{n}^{\ast }z_{n}$ is a contraction. We observe
that the continuity modulus $\varpi ^{\psi _{n}}$ of $\psi _{n}$ satisfies $%
\varpi ^{\psi _{n}}\left( t,s\right) \leq t+s$. Indeed, suppose that $\bar{c}
$ is a tuple in a C*-algebra $A$ such that $c_{1}^{\ast }c_{1}+\cdots
+c_{n}^{\ast }c_{n}$ is a contraction, and $a,b$ are elements of the unit
ball of $A$. Denote by $C$ the column vector with entries $c_{1},\ldots
,c_{n}$ defined as above. Let also $a\otimes I_{n}$ be the $n\times n$
matrix in $M_{n}\left( A\right) $ with $a$ in the diagonal and zeros
elsewhere. Then we have that 
\begin{eqnarray*}
&&\left\vert \left\Vert c_{1}^{\ast }a^{\ast }ac_{1}+\cdots +c_{n}^{\ast
}a^{\ast }ac_{n}\right\Vert -\left\Vert c_{1}^{\ast }b^{\ast }bc_{1}+\cdots
+c_{n}^{\ast }b^{\ast }bc_{n}\right\Vert \right\vert \\
&=&\left\vert \left\Vert \left( a\otimes I_{n}\right) C\right\Vert
-\left\Vert \left( b\otimes I_{n}\right) C\right\Vert \right\vert \leq
\left\Vert \left( a\otimes I_{n}\right) C-\left( b\otimes I_{n}\right)
C\right\Vert \\
&=&\left\Vert \left( \left( a-b\right) \otimes I_{n}\right) C\right\Vert
\leq \left\Vert \left( a-b\right) \otimes I_{n}\right\Vert \left\Vert
C\right\Vert \leq \left\Vert a-b\right\Vert \text{.}
\end{eqnarray*}%
This easily gives the above claim on the continuity modulus of $\psi _{n}$.

We can thus consider the $\sup \bigvee \inf $-sentence $\varphi $ given by%
\begin{equation*}
\sup_{x,y\in D_{1}}\inf_{n\in \mathbb{N}}\psi _{n}\left( x,y\right) \text{.}
\end{equation*}%
It is clear from the characterization of simplicity recalled above that the
condition $\varphi \leq 0$ witnesses that the class of simple C*-algebras
admits an infinitary $\sup \bigvee \inf $-axiomatization.

\subsection{Model-theoretic forcing}

Suppose that $L$ is a language and $\mathcal{C}$ is an elementary class of $%
L $-structure. We present here the technique of model-theoretic forcing
following the presentation from \cite{ben_yaacov_model_2009}. Alternative
approaches can be found in \cite%
{farah_model_2017,farah_omitting_2014,goldbring_enforceable_2017}.

We suppose that $L$ is separable for $\mathcal{C}$. Recall that this means
that the seminorm on $L$-formulas $\varphi \left( x_{1},\ldots ,x_{n}\right) 
$ defined by%
\begin{equation*}
\left\Vert \varphi \right\Vert =\sup \left\{ \varphi ^{M}(\bar{a}):M\in 
\mathcal{C},\bar{a}\in D_{1}^{M}\times \cdots \times D_{n}^{M}\right\}
\end{equation*}%
is separable. After replacing formulas (or definable predicates) with
formulas from a fixed countable dense subset, we can assume that $L$ only
contains countably many symbols. For simplicity, we work in the setting of
languages containing a single domain of quantification, and where all the
bounds for the relation symbols in $L$ are the interval $\left[ 0,1\right] $%
. The arguments can be adapted to the more general setting a straightforward
way.

One can consider a canonical dense set of formulas. This is the collection
of \emph{restricted }$L$-formulas, which are formulas where the only
connectives used are among the following ones: $t\mapsto 1-t$, $t\mapsto t/2$%
, $\left( t,s\right) \mapsto s\dot{+}t:=\min \left\{ t+s,1\right\} $.
Observe that, since we are assuming that the language $L$ only contains
countably many function and relation symbols, the set of restricted $L$%
-formulas is countable. Restricted infinitary $\sup \bigvee \inf $-formulas
are defined as $\sup \bigvee \inf $-formulas, but starting from restricted $%
L $-formulas.

Let $C$ be a countable set of constant symbols that do not already belong to 
$L$. Set then $L\left( C\right) $ to be the language obtained from $L$ by
adding the constant symbols from $C$. We will denote $L\left( C\right) $%
-structures as $M^{+}$, $N^{+},$ etc. If $M^{+}$ is an $L\left( C\right) $%
-structure, the $L$-structure obtained from $M^{+}$ by \textquotedblleft
forgetting\textquotedblright\ about the interpretation of constants in $C\ $%
(which is called the $L$-\emph{reduct} of $M^{+}$) is denoted by $M$. In
this case, we say that $M^{+}$ is an \emph{expansion }of $M$. An $L\left(
C\right) $-structure is \emph{canonical} if the set of interpretations of
constants from $C$ is dense.

Let us say that an \emph{open condition }is an expression of the form $%
\varphi <r$, where $\varphi $ is a \emph{quantifier-free} \emph{restricted }$%
L\left( C\right) $-sentence, and $r\in \mathbb{Q}$. An $L\left( C\right) $%
-structure $M^{+}$ in $\mathcal{C}$ satisfies $\varphi <r$ if $\varphi
^{M}<r $. A \emph{forcing condition }$p$ is a finite set of \emph{open
conditions }for which there exists an $L\left( C\right) $-structure $M^{+}$
whose $L$-reduct $M$ belongs to $\mathcal{C}$, and such that all the open
conditions in $p$ are satisfied in $M^{+}$. The set $\mathbb{P}$ of forcing
conditions is a countable partially ordered set with respect to reverse
inclusion.

We aim at proving the following \emph{omitting types theorem}, which
provides a sufficient (in fact, also necessary) criterion for the existence
of structures with certain properties.

\begin{theorem}[Omitting types]
\label{Theorem:omitting-types}Suppose that $P$ is a property admitting an
infinitary $\sup \bigvee \inf $-axiomatization given by a countable
collection of conditions $\varphi _{n}\leq r_{n}$, where $\varphi _{n}$ is
the $\sup \bigvee \inf $-sentence $\sup_{\bar{x}}\psi _{n}\left( \bar{x}%
\right) $ and $\psi _{n}\left( \bar{x}\right) $ is of the form $\inf_{m\in 
\mathbb{N}}\inf_{\bar{y}}\sigma _{n,m}\left( \bar{x},\bar{y}\right) $ for
some quantifier-free definable predicates $\sigma _{n,m}\left( \bar{x},\bar{y%
}\right) $. Suppose that for every forcing condition $q$, every $n\in 
\mathbb{N}$, every $\varepsilon >0$, and every tuple of constants $\bar{c}$
in $C$, the set of open (infinitary) conditions $q\cup \left\{ \psi
_{n}\left( \bar{c}\right) <r_{n}+\varepsilon \right\} $ is satisfied in some 
$L\left( C\right) $-structure whose $L$-reduct is in $\mathcal{C}$. Then
there exists a separable $L$-structure satisfying $P$.
\end{theorem}

Let us say that a forcing condition $p$ \emph{forces }an open condition $%
\varphi <r$, for some atomic $L\left( C\right) $-sentence $\varphi $, in
formulas $p\vdash \varphi <r$, if there is $r_{0}\leq r$ such that the open
condition $\varphi <r_{0}$ belongs to $p$. We extend the definition of
forcing to more general formulas by induction on the complexity \cite[Remark
2.3]{ben_yaacov_model_2009}:

\begin{itemize}
\item $p\vdash \frac{1}{2}\varphi <r$ iff $p\vdash \varphi <2r$;

\item $p\vdash (1-\varphi )<r$ iff $\exists s>1-r$ such that for every $%
q\leq p$, $q\nvdash \varphi <s$;

\item $p\vdash \varphi \dot{+}\psi <r$ iff there exist $s_{0},s_{1}$ such
that $s_{0}+s_{1}<r$ and $p\vdash \varphi <s_{0}$ and $p\vdash \psi <s_{1}$;

\item $p\vdash \inf_{n}\varphi _{n}<r$ iff there is $n$ such that $p\vdash
\varphi _{n}<r$;

\item $p\vdash \inf_{x}\varphi \left( x\right) <r$ iff there is $c\in C$
such that $p\vdash \varphi \left( c\right) <r$.
\end{itemize}

This forcing notion satisfy the following properties:

\begin{enumerate}
\item Given $p,q\in \mathbb{P}$ such that $q\leq p$, if $p\vdash \varphi <r$%
, then $q\vdash \varphi <r$;

\item Given $p\in \mathbb{P}$, $\varepsilon >0$, and an $L\left( C\right) $%
-term $\tau $ without variables, there exists $q\leq p$ and $c\in C$ such
that $q\vdash d\left( \tau ,c\right) <\varepsilon $

\item Given $p\in \mathbb{P}$, $r>0$, and $L\left( C\right) $-terms $\tau
,\sigma $ without free variables, if $p\vdash d\left( \sigma ,\tau \right)
<r $, then there exists $q\leq p$ such that $q\vdash d\left( \tau ,\sigma
\right) <r$;

\item Given $p\in \mathbb{P}$, a quantifier-free $L$-formula $\varphi \left(
x_{1},\ldots ,x_{\ell }\right) $ with continuity modulus $\varpi ^{\varphi }$%
, $L\left( C\right) $-terms $\tau _{1},\ldots ,\tau _{\ell },\sigma
_{1},\ldots ,\sigma _{\ell }$ without variables, and $\delta _{1},\ldots
,\delta _{\ell }>0$ such that $p\vdash d\left( \tau _{i},\sigma \right)
<\delta $, then there exists $q\leq p$ such that $q\vdash \left\vert \varphi
\left( \sigma _{1},\ldots ,\sigma _{\ell }\right) -\varphi \left( \tau
_{1},\ldots ,\tau _{\ell }\right) \right\vert <\varpi ^{\varphi }\left(
\delta _{1},\ldots ,\delta _{\ell }\right) $.
\end{enumerate}

In term of this forcing relation, we define a \emph{weak }forcing relation
as follows. Suppose that $p$ is a forcing condition and $\varphi <r$ is an
open condition. Then $p$ weakly forces $\varphi <r$, in formulas $p\vdash
^{w}\varphi <r$ if and only if there exists $r^{\prime }<r$ such that for
every $q\leq p$ there exists $q^{\prime }\leq q$ such that $q^{\prime
}\vdash \varphi <r^{\prime }$.

The following is the main lemma in the proof of Theorem \ref%
{Theorem:omitting-types}; see \cite[Proposition 4.5]{ben_yaacov_model_2009}

\begin{lemma}
\label{Lemma:main-forcing}Suppose that $\varphi $ is the $\sup \bigvee \inf $%
-sentence $\inf_{\bar{x}}\psi \left( \bar{x}\right) $ where $\psi \left( 
\bar{x}\right) =\inf_{m}\inf_{\bar{y}}\sigma _{m}\left( \bar{x},\bar{y}%
\right) $. Fix a forcing condition $p\in \mathbb{P}$, and $\varepsilon >0$.
Assume that for every forcing condition $q\leq p$ and for every tuple $\bar{c%
}$, there exists a tuple $\bar{d}$ in $C$ and $m\in \mathbb{N}$ such that $%
q\cup \left\{ \sigma _{m}(\bar{c},\bar{d})<r\right\} $ is a forcing
condition. Then $p\vdash \varphi <r+\varepsilon $.
\end{lemma}

In order to get an $L\left( C\right) $-structure from such a forcing notion,
we need to start from a filter\emph{\ }on $\mathbb{P}$. As in the
order-theoretic terminology, a subset $G$ of $\mathbb{P}$ is a \emph{filter}
if it satisfies the following properties:

\begin{enumerate}
\item if $p,q\in \mathbb{P}$ are such that $q\leq p$ and $q\in G$, then also 
$p\in G$;

\item if $n\in \mathbb{N}$ and $p_{1},\ldots ,p_{n}\in G$, then there exists 
$q\in G$ such that $q\leq p_{i}$ for $i=1,2,\ldots ,n$.
\end{enumerate}

\begin{definition}
A filter $G\subset \mathbb{P}$ is \emph{generic }if for every $L\left(
C\right) $-sentence $\varphi $ and $\varepsilon >0$ there exist $p\in G$ and 
$r_{0},r_{1}\in \mathbb{R}$ such that $p\vdash \varphi <r_{0}$ and $p\vdash
1-\varphi <r_{1}$, where $r_{0}+r_{1}<1+\varepsilon $.
\end{definition}

A standard argument in forcing shows that generic filters do exist; see also 
\cite[Proposition 2.12]{ben_yaacov_model_2009}.

\begin{lemma}
Fix $p\in \mathbb{P}$. Then there is a generic filter $G\subset \mathbb{P}$
that contains $p$.
\end{lemma}

\begin{proof}
Fix an enumeration of all the pairs $\left( \varphi _{n},\varepsilon
_{n}\right) $ where $\varphi _{n}$ is an $L\left( C\right) $-sentence and $%
\varepsilon \in \mathbb{Q}\cap \left( 0,+\infty \right) $. Define a sequence 
$\left( p_{n}\right) $ in $\mathbb{P}$ as follows. Set $p_{0}:=p$. Assuming
that $p_{n}$ has been defined. By definition of the forcing notion, there
exist $p_{n+1}\leq p_{n}$ and $r_{0},r_{1}\in \mathbb{R}$ such that $%
p_{n+1}\vdash \varphi _{n}<r_{0}$, $p_{n+1}\vdash \left( 1-\varphi
_{n}\right) <r_{1}$, and $r_{0}+r_{1}<1+\varepsilon _{n}$. This defines $%
p_{n+1}$, concluding the recursive construction. Let now $G$ be the filter
generated by the sequence $\left( p_{n}\right) $. This is just the set of $%
q\in \mathbb{P}$ such that $q\geq p_{n}$ for some $n\in \mathbb{N}$. Then $G$
is a generic filter containing $p$.
\end{proof}

Suppose now that $G$ is a generic filter for $\mathbb{P}$. If $\varphi <r$
is an open condition, define $G\vdash \varphi <r$ if and only if $p\vdash
\varphi <r$ for some $p\in G$. This is equivalent to the assertion that $%
p\vdash ^{w}\varphi <r$ for some $p\in G$ \cite[Lemma 2.13]%
{ben_yaacov_model_2009}. For an $L\left( C\right) $-sentence $\varphi $, set 
$\varphi ^{G}$ to be the infimum of $r$ such that $G\vdash \varphi <r$.

The properties of the forcing notion listed above readily imply the
following \cite[Lemmas 2.14, 2.15]{ben_yaacov_model_2009}:

\begin{enumerate}
\item If $\varphi $ is an $L\left( C\right) $-sentence, then $\left(
1-\varphi \right) ^{G}=1-\varphi ^{G}$;

\item For every $L\left( C\right) $-term $\tau $ and $\varepsilon >0$ there
exists a term $c\in C$ such that $d\left( \tau ,c\right) <\varepsilon $;

\item If $\tau ,\sigma $ are $L\left( C\right) $-terms without variables,
then $d\left( \sigma ,\tau \right) ^{G}=d\left( \tau ,\sigma \right) ^{G}$.

\item Given an atomic $L\left( C\right) $-formula $\varphi \left(
x_{1},\ldots ,x_{\ell }\right) $ with continuity modulus $\varpi ^{\varphi }$%
, $L\left( C\right) $-terms $\tau _{1},\ldots ,\tau _{\ell },\sigma
_{1},\ldots ,\sigma _{\ell }$ without variables, and $\delta _{1},\ldots
,\delta _{\ell }>0$, if $d\left( \tau _{i},\sigma _{i}\right) ^{G}<\delta
_{i}$ for $i=1,2,\ldots ,\ell $, then $\left\vert \varphi \left( \sigma
_{1},\ldots ,\sigma _{\ell }\right) -\varphi \left( \tau _{1},\ldots ,\tau
_{\ell }\right) \right\vert ^{G}<\varpi ^{\varphi }\left( \delta _{1},\ldots
,\delta _{\ell }\right) $.
\end{enumerate}

Given a generic filter $G$, one can define the corresponding canonical \emph{%
compiled structure }$M_{G}^{+}$, whose $L$-reduct is denoted by $M_{G}$, as
follows. Consider the set $M_{0}$ of $L\left( C\right) $-terms, and define a
metric $d_{M_{0}}$ on $M_{0}$ by setting, in the notation above, $%
d_{M_{0}}\left( t_{0},t_{1}\right) =\varphi ^{G}$, where $\varphi $ is the
atomic $L\left( C\right) $-sentence $d\left( t_{0},t_{1}\right) $. Let then $%
M_{G}$ be the Hausdorff completion of $M_{0}$. By abuse of notation, we
identify $M_{0}$ as a subset of $M_{G}$. One can define interpretation of
function and relation symbols from $L$ in $M$ as follows. If $c\in C$ then
we let $c^{M_{G}^{+}}$ be $c$, which belongs to $M_{0}$ as an $L\left(
C\right) $-term. If $f$ is an $n$-ary function symbol and $t_{1},\ldots
,t_{n}\in M_{0}$ are $L\left( C\right) $-terms, then we let $%
f^{M_{G}^{+}}\left( t_{1},\ldots ,t_{n}\right) $ be the $L\left( C\right) $%
-term $f\left( t_{1},\ldots ,t_{n}\right) \in M_{0}$. Uniform continuity
guarantees that $f^{M_{G}^{+}}$ extends to a continuous function $%
f^{M_{G}^{+}}:M_{G}\rightarrow M_{G}$. Similarly, for an $n$-ary relation
symbol $R$ and $t_{1},\ldots ,t_{n}\in M_{0}$ are $L\left( C\right) $-terms,
we let $R^{M_{G}^{+}}\left( t_{1},\ldots ,t_{n}\right) $ be $\varphi ^{G}$,
where $\varphi $ is the atomic $L\left( C\right) $-sentence $R\left(
t_{1},\ldots ,t_{n}\right) $. Again, one can then extend $R^{M}$ to the
whole of $M_{G}$ by uniform continuity.

The properties of the assignment $\varphi \mapsto \varphi ^{G}$ listed above
show that $M_{G}^{+}$ is indeed a canonical $L\left( C\right) $-structure.
Furthermore, one can show by induction on the complexity that, if $\varphi $
is an $L\left( C\right) $-formula, then $\varphi ^{M_{G}^{+}}=\varphi ^{G}$.

\begin{proof}[Proof of Theorem \protect\ref{Theorem:omitting-types}]
Let $G$ be a generic filter for $\mathbb{P}$, $M_{G}^{+}$ be the
corresponding compiled structure, and $M_{G}$ is its $L$-reduct. Recall that 
$P$ is a property admitting an infinitary $\sup \bigvee \inf $%
-axiomatization given by a countable collection of conditions $\varphi
_{n}\leq r_{n}$, where $\varphi _{n}$ is the $\sup \bigvee \inf $-sentence $%
\sup_{\bar{x}}\psi _{n}\left( \bar{x}\right) $ and $\psi _{n}\left( \bar{x}%
\right) $ is of the form $\inf_{m\in \mathbb{N}}\inf_{\bar{y}}\sigma
_{n,m}\left( \bar{x},\bar{y}\right) $ for some quantifier-free definable
predicates $\sigma _{n,m}\left( \bar{x},\bar{y}\right) $.

\begin{claim*}
$M_{G}$ satisfies $P$.
\end{claim*}

Observe that, by approximating them, we can assume that the $\sup \bigvee
\inf $-sentences $\varphi _{n}$ are actually \emph{restricted }$\sup \bigvee
\inf $-sentences. Furthermore, after replacing the conditions $\varphi
_{n}\leq r_{n}$ for $n\in \mathbb{N}$ with the conditions $\varphi _{n}\leq
r_{n}+2^{-m}$ for $n,m\in \mathbb{N}$, it actually suffices to show that the
claim holds under the following assumptions: for every forcing condition $p$%
, every $n\in \mathbb{N}$, and every tuple of constants $\bar{c}$ in $C$,
the set of open (infinitary) conditions $p\cup \left\{ \psi _{n}\left( \bar{c%
}\right) <r_{n}\right\} $ is satisfied in some structure from $\mathcal{C}$.
Finally, since we will show that \emph{every }generic filter works, it
suffices to consider the case when $P$ is axiomatized by a single condition $%
\varphi \leq r$, where $\varphi $ is the infinitary $\sup \bigvee \inf $%
-formula $\sup_{\bar{x}}\psi \left( \bar{x}\right) $ and $\psi \left( \bar{x}%
\right) $ is the formula $\inf_{m}\inf_{\bar{y}}\sigma _{m}\left( \bar{x},%
\bar{y}\right) $. In this case, we want to prove that the claim holds under
the assumption that for every forcing condition $q$, and every tuple of
constants $\bar{c}$ from $\mathcal{C}$, there exists an $L$-structure
satisfying $\psi \left( \bar{c}\right) <r$.

Fix $\varepsilon >0$. By the properties of the compiled structure, we have
that $\varphi ^{M_{G}}=\varphi ^{G}$. Furthermore, $\varphi
^{G}<r+\varepsilon $ if and only if $p\vdash \varphi <r+\varepsilon $ for
some $p\in G$. Fix an arbitrary $p\in G$. Fix a forcing condition $q\leq p$,
and a tuple $\bar{c}$. Since $q\cup \left\{ \psi \left( \bar{c}\right)
<r\right\} $ is satisfiable in an $L\left( C\right) $-structure whose $L$%
-reduct is in $\mathcal{C}$, there exist $k\in \mathbb{N}$ and a tuple $\bar{%
d}$ in $C$ such that $q^{\prime }:=q\cup \left\{ \sigma \left( \bar{c},\bar{d%
}\right) <r\right\} \ $is a forcing condition. By Lemma \ref%
{Lemma:main-forcing}, this shows that $p\vdash \varphi <r+\varepsilon $,
concluding the proof.
\end{proof}

\subsection{Building models by games}

\label{Subsection:building}

One can alternatively present the ideas above using the formalism of
\textquotedblleft building models by games\textquotedblright . This has been
developed in classical discrete logic in \cite{hodges_building_1985}. A
version in the setting of logic for metric structures is considered in \cite%
{goldbring_enforceable_2017}. In this setting, one considers game involving
two players (Abelard and Eloise). The players alternate turns, and the game
runs for infinitely many turns. Abelard starts by playing a forcing
condition $p_{0}$, and Eloise has to reply with a forcing condition $p_{1}$
such that $p_{1}\leq p_{0}$, in the sense that every open condition in $%
p_{0} $ also belongs to $p_{1}$. Abelard then replies with a forcing
condition $p_{2}$ containing $p_{1}$, and so on. The game runs for
infinitely many turns, producing a chain $p_{0}\geq p_{1}\geq p_{2}\geq
\cdots $ of conditions. One then lets $\overline{p}$ be their union.

A play of the game is \emph{definitive }if, for every atomic $L\left(
C\right) $-sentence $\varphi $ there exists $r^{\varphi }\in \mathbb{R}$
(only depending on $\varphi $) such that, for every $L$-structure $M$ in $%
\mathcal{C}$ satisfying $\overline{p}$ one has that $\varphi ^{M}=r^{\varphi
}$. In this case, one can define a canonical $L\left( C\right) $-structure $%
M^{+}\left( \overline{p}\right) $, called the \emph{compiled structure}, as
follows. Consider the set $M_{0}$ of $L\left( C\right) $-terms, and define a
metric $d_{M_{0}}$ on $M_{0}$ by setting, in the notation above, $%
d_{M_{0}}\left( t_{0},t_{1}\right) =r^{\varphi }$, where $\varphi $ is the
atomic $L\left( C\right) $-sentence $d\left( t_{0},t_{1}\right) $. Let then $%
M\left( \overline{p}\right) $ be the Hausdorff completion of $M_{0}$. By
abuse of notation, we identify $M_{0}$ as a subset of $M\left( \overline{p}%
\right) $. One can define interpretation of function and relation symbols
from $L$ in $M\left( \overline{p}\right) $ as follows. If $c\in C$ then we
let $c^{M}$ be $c$, which belongs to $M_{0}$ as an $L\left( C\right) $-term.
If $f$ is an $n$-ary function symbol and $t_{1},\ldots ,t_{n}\in M_{0}$ are $%
L\left( C\right) $-terms, then let $f^{M}\left( t_{1},\ldots ,t_{n}\right) $
be the $L\left( C\right) $-term $f\left( t_{1},\ldots ,t_{n}\right) \in
M_{0} $. Uniform continuity guarantees that $f^{M}$ extends to a continuous
function $f^{M}:M\rightarrow M$. Similarly, for an $n$-ary relation symbol $%
R $ and $t_{1},\ldots ,t_{n}\in M_{0}$ are $L\left( C\right) $-terms, one
let $R^{M}\left( t_{1},\ldots ,t_{n}\right) $ be, in the notation above, $%
r^{\varphi }$ where $\varphi $ is the atomic $L\left( C\right) $-sentence $%
R\left( t_{1},\ldots ,t_{n}\right) $. Again, one can then extend $R^{M}$ to
the whole of $M$ by uniform continuity.

Then one can reformulate the weak forcing relation as follows \cite[Theorem
2.22]{goldbring_enforceable_2017}.\ If $p$ is a forcing condition and $%
\varphi <r$ is an open condition, then $p\vdash ^{w}\varphi <r$ if and only
if the following holds: suppose that the game has been played up to the $k$%
-th turn, defining forcing conditions $p_{0}\geq p_{1}\geq \cdots \geq p_{k}$
such that $p_{k}\leq p$. Then, regardless of Abelard's moves, Eloise can
keep playing the game in such a way that the game is definitive and the
compiled structure $M\left( \overline{p}\right) $ satisfies the open
condition $\varphi <r$. In other words, the \emph{position }$\left(
p_{0},\ldots ,p_{k}\right) $ is a winning position for Eloise in the game $%
G\left( \overline{p}\right) $ whose winning conditions for Eloise are that
the game is definitive and the compiled structure satisfies the open
condition $\varphi <r$.

\subsection{Forcing and the UCT problem}

The omitting types theorem (Theorem \ref{Theorem:omitting-types}) provides a
method of constructing C*-algebras, which is different from any other the
standard constructions in\ C*-algebra theory. It is therefore reasonable to
expect that it might at least have some bearing on the UCT question; see
Subsection \ref{Subsection:classification}. Indeed, one can use Theorem \ref%
{Theorem:omitting-types} to provide a \textquotedblleft
concrete\textquotedblright\ reformulation of the UCT question, which in some
sense can be seen as the combinatorial core---or \emph{a }combinatorial
core---of such a problem.

We will consider a reformulation of the UCT problem due to Kirchberg \cite[%
Corollary 8.4.6]{rordam_classification_2002}. This reformulation asserts
that every separable nuclear C*-algebra satisfies the UCT if and only if $%
\mathcal{O}_{2}$ is uniquely characterized among separable, nuclear, simple,
and purely infinite C*-algebras by its Elliott invariant. Since for purely
infinite C*-algebras the trace simplex is empty, this is equivalent to the
assertion that $\mathcal{O}_{2}$ is the unique separable, nuclear, simple,
and purely infinite C*-algebra with trivial $K_{0}$ and $K_{1}$ groups.

We claim that the class of simple, purely infinite C*-algebra with trivial $%
K_{0}$-group is elementary, and in fact $\sup \inf $-axiomatizable. We have
seen in Subsection \ref{Subsection:axiomatize-C*-more} that the class of
simple, purely infinite C*-algebras is elementary, and the proof there shows
that it is in fact $\sup \inf $-axiomatizable. For a simple, purely infinite
C*-algebra $A$, the $K_{0}$-group of $A$ is trivial if and only if any two
nonzero projections of $A$ are Murray--von Neumann equivalent \cite[%
Proposition 4.1.4]{rordam_classification_2002}. We have shown in Subsection %
\ref{Subsection:definability-C*} that the relation of Murray--von Neumann
equivalent is definable, as witnessed by a \emph{existential }definable
predicate. This easily shows that the class of simple, purely infinite
C*-algebras with trivial $K_{0}$-group is $\sup \inf $-definable.

We now claim that there is a $\sup \inf $-axiomatizable class $\mathcal{C}$
such that the set of separable nuclear C*-algebras that belong to $\mathcal{C%
}$ are precisely the separable, nuclear, simple, and purely infinite
C*-algebras with trivial $K_{0}$ and $K_{1}$ groups. Let $A$ be a separable,
nuclear, simple, and purely infinite C*-algebra $A$. Then, the $K_{1}$-group
of $A$ is trivial if and only if the unitary group $U(A)$ is connected \cite[%
Proposition 4.1.15]{rordam_classification_2002}. Furthermore, $A$ absorbs
tensorially the Cuntz algebra $\mathcal{O}_{\infty }$ by Kirchberg's $%
\mathcal{O}_{\infty }$-absorption theorem \cite{kirchberg_embedding_2000}.\
Thus by \cite[Theorem 3.1]{phillips_real_2002}, any element in the connected
component $U_{0}(A)$ of the identity in $U(A)$ is connected to the identity
by a path of length at most $2\pi $. Since any unitary at distance less than 
$2$ from the identity is in $U_{0}\left( A\right) $, this allows one to
conclude that $A$ has connected unitary group if and only if it satisfies
the $\sup \inf $-condition%
\begin{equation*}
\sup_{u\text{ unitary}}\inf_{v_{1},v_{2},v_{3}\text{ unitaries}}\max \left\{
\left\Vert u-v_{1}\right\Vert ,\left\Vert v_{1}-v_{2}\right\Vert ,\left\Vert
v_{2}-v_{3}\right\Vert ,\left\Vert v_{3}-1\right\Vert \right\} \leq 7/4\text{%
.}
\end{equation*}%
Thus adding such a condition to the set of axioms for simple, purely
infinite C*-algebras with trivial $K_{0}$-group gives an axiomatization for
an elementary class $\mathcal{C}$ as desired.

Let $\mathcal{C}_{\mathrm{nuc}}$ be the class of nuclear C*-algebras in $%
\mathcal{C}$. Since $\mathcal{C}$ is $\sup \inf $-axiomatizable, and
nuclearity admits an infinitary $\sup \bigvee \inf $-axiomatization, the
class $\mathcal{C}_{\mathrm{nuc}}$ admits an infinitary $\sup \bigvee \inf $%
-axiomatization. This is just obtained by adding to the axioms of $\mathcal{C%
}$ the $\sup \bigvee \inf $-condition defining nuclearity. Let $\varphi
_{n}^{\mathcal{C}}\leq 0$ for $n\in \mathbb{N}$ be such conditions, which
can be explicitly extracted from the discussions above. For each $n\in 
\mathbb{N}$, $\varphi _{n}^{\mathcal{C}}$ is a $\sup \bigvee \inf $-formula,
which we can write as $\sup_{\bar{x}}\inf_{m\in \mathbb{N}}\inf_{\bar{y}%
}\sigma _{n,m}\left( \bar{x},\bar{y}\right) $.

Suppose now that $A$ is a separable C*-algebra from $\mathcal{C}_{\mathrm{nuc%
}}$. Then by Kirchberg's $\mathcal{O}_{2}$-absorption theorem, $A$ is
isomorphic to $\mathcal{O}_{2}$ if and only if $A$ absorbs $\mathcal{O}_{2}$
tensorially. By Theorem \ref{Theorem:absorption}, this is equivalent to the
assertion that, for every positive quantifier-free formula $\theta \left( 
\bar{y}\right) $, if $\psi _{\theta }$ is the positive $\sup \inf $-sentence%
\begin{equation*}
\sup_{\bar{x}}\inf_{\bar{y}}\max \left\{ \theta \left( \bar{y}\right)
,\left\Vert x_{i}y_{j}-y_{j}x_{i}\right\Vert :i,j\right\} \text{,}
\end{equation*}%
then $\psi _{\theta }^{\mathcal{O}_{2}}\leq \psi _{\theta }^{A}$. Thus $A$
is \emph{not }isomorphic to $\mathcal{O}_{2}$ if and only if there exists a $%
\sup \inf $-sentence $\psi _{\theta }$ such that $\psi _{\theta }^{\mathcal{O%
}_{2}}>\psi _{\theta }^{A}$. In view of the above discussion and the
omitting types theorem (Theorem \ref{Theorem:omitting-types}) applied to the
class of C*-algebras regarded as $L^{\text{C*}}$-structures, one can provide
the following sufficient criterion to establish that the UCT fails.

\begin{theorem}
Assume that there exists a quantifier-free formula $\theta $ such that the
following holds. For every finite set of open quantifier-free conditions $%
\psi _{i}\left( \bar{x},\bar{y},\overline{z}\right) <r_{i}$ for $%
i=1,2,\ldots ,\ell $ which are realized in some C*-algebra, there exist $%
m\in \mathbb{N}$, a C*-algebra $A$ satisfying $\psi _{\theta }^{A}<\psi
_{\theta }^{\mathcal{O}_{2}}$, and tuples $\bar{a},\bar{b},\bar{c}$ in $A$,
satisfying $\psi _{i}\left( \bar{a},\bar{b},\bar{c}\right) <r_{i}$, and $%
\sigma _{n,m}\left( \bar{a},\bar{b}\right) <\varepsilon $. Then the UCT
fails.
\end{theorem}

\section{Further results and outlook}

For reasons of space, we have omitted in the above discussion many important
directions of applications of model theory to operator algebras. For the
sake of completeness, we mention here some of such directions.

\subsection*{Von Neumann factors}

Finite Von Neumann factors and, more generally, tracial von Neumann algebras
also fit in the framework of first order logic for metric structures. Recall
that a \emph{tracial von Neumann algebra} is a von Neumann algebra $M$
endowed with a distinguished faithful normal tracial state $\tau $. A
tracial von Neumann algebra is a \emph{finite} \emph{factor} if it has
trivial center. In this case, the faithful normal tracial state $\tau $ is
uniquely determined by $M$. A II$_{1}$ factor is an infinite-dimensional
finite factor.

In order to regard tracial von Neumann algebras as structures, one can
consider the language of C*-algebras with an additional relation symbol $%
\tau $ to be interpreted as the given trace. In this case, the relation
symbol for the norm should be interpreted as the $2$-norm $\left\Vert
x\right\Vert _{\tau }=\tau \left( x^{\ast }x\right) ^{1/2}$ associated with
the trace $\tau $. Consistently, the canonical binary relation symbol for
the metric should be interpreted as the metric associated with such a norm.
The domains $D_{n}$ for $n\in \mathbb{N}$ should still be interpreted as the
closed balls \emph{with respect to the operator norm}. This perspective has
been used in \cite{farah_model_2013,farah_model_2014} to answer to a
question of McDuff from \cite{mcduff_central_1970} on the number of
isomorphism classes of relative commutants $M^{\prime }\cap M^{\mathcal{U}}$
associated with nonprincipal ultrafilters $\mathcal{U}$ over $\mathbb{N}$
for a given separable II$_{1}$ factor $M$. As in the case of C*-algebras,
the Continuum Hypothesis (CH) implies that all such relative commutants are
isomorphic, while the negation of CH implies that there exist at least two
nonisomorphic such relative commutants. The model-theoretic study of factors
has been further pursued in \cite%
{farah_model_2014-1,goldbring_theory_2013,farah_existentially_2016,boutonnet_factors_2015}%
, where it is shown that the set of separable models of a consistent theory
of II$_{1}$ factors has size continuum, the theory of tracial von Neumann
algebras does not have a model companion, the Connes Embedding Problem is
equivalent to the assertion that the hyperfinite II$_{1}$ factor $\mathcal{R}
$ is existentially closed, and that there exists a continuum of distinct
theories of II$_{1}$ factors.

\subsection*{Compact Hausdorff spaces}

In a series of papers going back to the 1980s \cite%
{bankston_reduced_1987,bankston_hierarchy_1999,bankston_co-elementary_1997,bankston_some_2000}%
, Bankston introduced dual notions to fundamental notions in model theory
(elementary equivalence, elementary embedding, ultraproduct), and applied
such notions to the study of compact Hausdorff spaces. It has been observed
in \cite{eagle_pseudoarc_2016} that, if one replaces a compact Hausdorff
space $X$ with the abelian C*-algebra $C\left( X\right) $ of continuous
functions over $X$, then the notions introduced by Bankston coincide with
the usual notions from model theory for metric structures where $C\left(
X\right) $ is regarded as a structure in the language of C*-algebras. This
perspective has been used in \cite%
{eagle_pseudoarc_2016,goldbring_enforceable_2017} to show that, if $\mathbb{P%
}$ is the \emph{pseudoarc }(the unique hereditarily indecomposable,
chainable, metrizable continuum), then $C\left( \mathbb{P}\right) $ is
existentially closed---in Bankston's terminology, $\mathbb{P}$ is
co-existentially closed---and it is the prime model of its theory.\
Furthermore, for zero-dimensional compact Hausdorff spaces, elementary
equivalence is equivalent to elementary equivalence of the associated
Boolean algebras of clopen sets \cite{eagle_saturation_2015}. If the spaces
have no isolated points, similar conclusions apply to countable saturation 
\cite{eagle_saturation_2015}.

\subsection*{Enforceable operator algebras}

The framework of model-theoretic forcing can be used to define the notion of
enforceable structure. Let $\mathcal{C}$ be an elementary class of $L$%
-structure, and consider the game between Abelard and Eloise defined in
Subsection \ref{Subsection:building}. One then says that a property $P$ of $%
L $-structures is \emph{enforceable }if Eloise has a winning strategy when
her winning conditions require the compiled structure to satisfy $P$. This
can be seen as a model-theoretic notion of \emph{genericity} for $L$%
-structures satisfying $P$ within the class $\mathcal{C}$. An $L$-structure $%
M$ is enforceable if the property of being isomorphic to $M$ is enforceable.

Many outstanding open problems in operator algebra theory can be
reformulated as the assertion that the known examples of strongly
self-absorbing C*-algebras are enforceable within a suitable class of
C*-algebras \cite{goldbring_enforceable_2017}. For instance, the Kirchberg
embedding problem, asking whether every C*-algebra embeds into an ultrapower
of the Cuntz algebra $\mathcal{O}_{2}$, is equivalent to the assertion that $%
\mathcal{O}_{2}$ is enforceable in the class of all C*-algebras. Similarly,
the MF problem of Blackadar and Kirchberg, asking whether every stably
finite C*-algebra embeds into an ultrapower of the rational UHF algebra $%
\mathcal{Q}=\bigotimes_{n\in \mathbb{N}}M_{n}\left( \mathbb{C}\right) $, is
equivalent to the assertion that $Q$ is enforceable within the class of
stably finite C*-algebras. Finally, the assertion that every stably finite 
\emph{projectionless }C*-algebra embeds into an ultrapower of the Jiang-Su
algebra $\mathcal{Z}$ is equivalent to the assertion that $\mathcal{Z}$ is
enforceable within the class of stably finite projectionless C*-algebras.

This framework can also be applied in the context of II$_{1}$ factors. In
this case, the Connes Embedding Problem, asking whether every II$_{1}$
factor embeds into an ultrapower of the hyperfinite II$_{1}$ factor $%
\mathcal{R}$, turns out to be equivalent to the assertion that $\mathcal{R}$
is enforceable within the class of II$_{1}$ factors.

\subsection*{Actions of compact (quantum) groups on C*-algebras}

The \emph{equivariant theory }of C*-algebras studies C*-algebras endowed
with a distinguished continuous action of a locally compact group $G$ ($G$%
-C*-algebras). The case which is best understood is when the acting group $G$
is finite or, more generally, compact. It is clear that, when $G$ is finite,
one can regard $G$-C*-algebras as structures in the language $L_{G}^{\text{C*%
}}$ obtained from the language of C*-algebras by adding unary functions
symbols $\alpha _{g}$ for $g\in G$, to be interpreted as the automorphism of
the given C*-algebras that define the $G$-action. More generally, as shown
in \cite{gardella_rokhlin_2017}, for an arbitrary compact groups $G$, $G$%
-C*-algebras fit into the framework of first order logic for metric
structures described above. To see this, one should notice the following.
Suppose that $A$ is a C*-algebra, and $\alpha $ is a continuous action of $G$
on $A$. One can regard $\alpha $ as a *-homomorphism $\alpha :A\rightarrow
C\left( G,A\right) \cong C\left( G\right) \otimes A$, $a\mapsto \left(
g\mapsto \alpha _{g}\left( a\right) \right) $. For every finite-dimensional
irreducible representation $\pi \in \mathrm{Rep}\left( \pi \right) $ of $G$
one can consider the span $C\left( G\right) _{\pi }\subset C\left( G\right) $
of the matrix units of $\pi $. The subspace $A_{\pi }=\left\{ a\in A:\alpha
\left( a\right) \in C\left( G\right) _{\pi }\otimes A\right\} $ is called 
\emph{spectral subspace }of $\alpha $ associated with $\pi $. The union of $%
A_{\pi }$ when $\pi $ varies among all the finite-dimensional
representations of $G$ is a dense *-subalgebra of $A$ (Podle\'{s} algebra) 
\cite{podles_symmetries_1995}. One can regard the $G$-C*-algebra $\left(
A,\alpha \right) $ as a two-sorted structure, with a sort for $A$ and a sort
for $C\left( G\right) \otimes A$. The language of $G$-C*-algebras is endowed
with domains $D_{\pi }$ for $\pi \in \mathrm{Rep}\left( \pi \right) $, to be
interpreted in $A$ as $A_{\pi }$ and in $C\left( G\right) \otimes A$ as $%
C\left( G\right) _{\pi }\otimes A_{\pi }$. It is shown in \cite%
{gardella_rokhlin_2017} that $G$-C*-algebras form an axiomatizable class in
such a language. This perspective, and the corresponding notion of positive
existential embedding, has been used implicitly in \cite%
{barlak_sequentially_2016} and explicitly in \cite{gardella_equivariant_2016}
to give a model-theoretic characterization of the Rokhlin property for $G$%
-C*-algebras. In turn, this characterization has been used to provide a
unified approach to several preservation results for fixed point algebras
and crossed products with respect to Rokhlin actions. More generally, a
model-theoretic characterization of Rokhlin dimension is considered in \cite%
{gardella_equivariant_2016}. This is applied to obtain preservation results
of finite nuclear dimension and finite composition rank for fixed point
algebras and crossed products with respect to actions with finite Rokhlin
dimension.

More generally, the theory can be developed in the context of actions of
compact \emph{quantum }groups on C*-algebras. A quantum group $\mathbb{G}$
is a C*-algebraic object which formally satisfies the same properties (save
from being abelian) as the C*-algebra $C\left( G\right) $ associated with a
classical compact group $G$ endowed with the \emph{comultiplication operation%
} $\Delta :C\left( G\right) \rightarrow C\left( G\times G\right) \cong
C\left( G\right) \otimes C\left( G\right) $, $f\mapsto \left( \left(
s,t\right) \mapsto f\left( st\right) \right) $. Continuous actions of a
compact quantum group $\mathbb{G}$ on C*-algebras ($\mathbb{G}$-C*-algebras)
can be defined in closed parallel with the classical case. It is also shown
in \cite{gardella_rokhlin_2017} that $\mathbb{G}$-C*-algebras form an
axiomatizable class in a suitable language, very similar to the one
described above for classical compact groups. This point of view has been
used, implicitly in \cite{barlak_spatial_2017} and explicitly in \cite%
{gardella_rokhlin_2017}, to generalize the notions of Rokhlin property and
Rokhlin dimension to the quantum setting, as well as virtually all known
preservation results of regularity properties under fixed point algebras and
crossed products by actions with finite Rokhlin dimension.

\bibliographystyle{amsplain}
\bibliography{biblio-houston}

\providecommand{\MR}[1]{}
\providecommand{\bysame}{\leavevmode\hbox to3em{\hrulefill}\thinspace}
\providecommand{\MR}{\relax\ifhmode\unskip\space\fi MR }
\providecommand{\MRhref}[2]{%
  \href{http://www.ams.org/mathscinet-getitem?mr=#1}{#2}
}
\providecommand{\href}[2]{#2}
\begin{thebibliography}{10}

\bibitem{akemann_central_1979}
Charles~A. Akemann and Gert~K. Pedersen, \emph{Central {sequences} and {inner}
  {derivations} of {separable} {C}*-{algebras}}, American Journal of
  Mathematics \textbf{101} (1979), no.~5, 1047--1061.

\bibitem{arzhantseva_almost_2015}
Goulnara Arzhantseva and Liviu Paunescu, \emph{Almost commuting permutations
  are near commuting permutations}, Journal of Functional Analysis \textbf{269}
  (2015), no.~3, 745--757. \MR{3350728}

\bibitem{arzhantseva_linear_2017}
\bysame, \emph{Linear sofic groups and algebras}, Transactions of the American
  Mathematical Society \textbf{369} (2017), no.~4, 2285--2310. \MR{3592512}

\bibitem{atiyah_riemann-roch_1959}
Micheal~F. Atiyah and Hirzebruch Hirzebruch, \emph{Riemann-{Roch} theorems for
  differentiable manifolds}, Bulletin of the American Mathematical Society
  \textbf{65} (1959), no.~4, 276--281.

\bibitem{bankston_co-elementary_1997}
Paul Bankston, \emph{Co-elementary equivalence, co-elementary maps, and
  generalized arcs}, Proceedings of the American Mathematical Society
  \textbf{125}, no.~12, 3715--3720. \MR{1422845}

\bibitem{bankston_hierarchy_1999}
\bysame, \emph{A hierarchy of maps between compacta}, The Journal of Symbolic
  Logic \textbf{64}, no.~4, 1628--1644. \MR{1780075}

\bibitem{bankston_some_2000}
\bysame, \emph{Some applications of the ultrapower theorem to the theory of
  compacta}, Applied Categorical Structures. \textbf{8}, no.~1, 45--66.
  \MR{1785837}

\bibitem{bankston_reduced_1987}
\bysame, \emph{Reduced {Coproducts} of {Compact} {Hausdorff} {Spaces}}, The
  Journal of Symbolic Logic \textbf{52} (1987), no.~2, 404--424.

\bibitem{barlak_sequentially_2016}
Sel{\c{c}}uk Barlak and G{\'{a}}bor Szab{\'{o}}, \emph{Sequentially split
  *-homomorphisms between {C}*-algebras}, International Journal of Mathematics
  \textbf{27} (2016), no.~13.

\bibitem{barlak_spatial_2017}
Sel{\c{c}}uk Barlak, G{\'{a}}bor Szab{\'{o}}, and Christian Voigt, \emph{The
  spatial {Rokhlin} property for actions of compact quantum groups}, Journal of
  Functional Analysis \textbf{272} (2017), no.~6, 2308--2360.

\bibitem{ben_yaacov_model_2008}
Ita{\"{i}} Ben~Yaacov, Alexander Berenstein, C.~Ward Henson, and Alexander
  Usvyatsov, \emph{Model theory for metric structures}, Model theory with
  applications to algebra and analysis. {Vol}. 2, London {Mathematical}
  {Society} {Lecture} {Note} {Series}, vol. 350, Cambridge University Press,
  2008, pp.~315--427.

\bibitem{ben_yaacov_model_2009}
Ita{\"{i}} Ben~Yaacov and Jos{\'{e}} Iovino, \emph{Model theoretic forcing in
  analysis}, Annals of Pure and Applied Logic \textbf{158} (2009), no.~3,
  163--174.

\bibitem{blackadar_operator_2006}
Bruce Blackadar, \emph{Operator {algebras}}, Encyclopaedia of {Mathematical}
  {Sciences}, vol. 122, Springer-Verlag, Berlin, 2006.

\bibitem{blackadar_generalized_1997}
Bruce Blackadar and Eberhard Kirchberg, \emph{Generalized inductive limits of
  finite-dimensional {C}*-algebras}, Mathematische Annalen \textbf{307} (1997),
  no.~3, 343--380.

\bibitem{boutonnet_factors_2015}
R{\'{e}}mi Boutonnet, Ionut Chifan, and Adrian Ioana, \emph{{II}{$_1$} factors
  with non-isomorphic ultrapowers}, Duke Mathematical Journal, in press.

\bibitem{brown_extensions_1977}
Lawrence~G. Brown, Ronald~G. Douglas, and Peter~A. Fillmore, \emph{Extensions
  of {C}*-algebras and {K}-homology}, Annals of Mathematics. Second Series
  \textbf{105} (1977), no.~2, 265--324. \MR{0458196}

\bibitem{brown_c*-algebras_1991}
Lawrence~G Brown and Gert Pedersen, \emph{{C}*-algebras of real rank zero},
  Journal of Functional Analysis \textbf{99}, no.~1, 131--149.

\bibitem{capraro_introduction_2015}
Valerio Capraro and Martino Lupini, \emph{Introduction to {Sofic} and
  hyperlinear groups and {Connes}' embedding conjecture}, Lecture {Notes} in
  {Mathematics}, vol. 2136, Springer, 2015, With an appendix by Vladimir
  Pestov.

\bibitem{carlson_omitting_2014}
Kevin Carlson, Enoch Cheung, Ilijas Farah, Alexander Gerhardt-Bourke, Bradd
  Hart, Leanne Mezuman, Nigel Sequeira, and Alexander Sherman, \emph{Omitting
  types and {AF} algebras}, Archive for Mathematical Logic \textbf{53} (2014),
  no.~1-2, 157--169.

\bibitem{carrion_groups_2013}
Jos{\'{e}}~R. Carri{\'{o}}n, Marius Dadarlat, and Caleb Eckhardt, \emph{On
  groups with quasidiagonal {C}*-algebras}, Journal of Functional Analysis
  \textbf{265} (2013), no.~1, 135--152.

\bibitem{cohen_independence_1963}
Paul~J. Cohen, \emph{The independence of the continuum hypothesis}, Proceedings
  of the National Academy of Sciences of the United States of America
  \textbf{50} (1963), 1143--1148. \MR{0157890}

\bibitem{cohen_independence_1964}
\bysame, \emph{The independence of the continuum hypothesis. {II}}, Proceedings
  of the National Academy of Sciences of the United States of America
  \textbf{51} (1964), 105--110. \MR{0159745}

\bibitem{coskey_automorphisms_2014}
Samuel Coskey and Ilijas Farah, \emph{Automorphisms of corona algebras, and
  group cohomology}, Transactions of the American Mathematical Society
  \textbf{366} (2014), no.~7, 3611--3630.

\bibitem{cuntz_simple_1977}
Joachim Cuntz, \emph{Simple {C}*-algebras generated by isometries},
  Communications in Mathematical Physics \textbf{57} (1977), no.~2, 173--185.

\bibitem{eagle_pseudoarc_2016}
Christopher~J. Eagle, Isaac Goldbring, and Alessandro Vignati, \emph{The
  pseudoarc is a co-existentially closed continuum}, Topology and its
  Applications \textbf{207} (2016), 1--9.

\bibitem{eagle_saturation_2015}
Christopher~J. Eagle and Alessandro Vignati, \emph{Saturation and elementary
  equivalence of {C}*-algebras}, Journal of Functional Analysis \textbf{269}
  (2015), no.~8, 2631--2664.

\bibitem{elek_sofic_2004}
G{\'{a}}bor Elek and Endre Szab{\'{o}}, \emph{Sofic groups and direct
  finiteness}, Journal of Algebra \textbf{280} (2004), no.~2, 426--434.

\bibitem{elek_sofic_2011}
\bysame, \emph{Sofic representations of amenable groups}, Proceedings of the
  American Mathematical Society \textbf{139} (2011), no.~12, 4285--4291.

\bibitem{elliott_regularity_2008}
George Elliott and Andrew Toms, \emph{Regularity properties in the
  classification program for separable amenable {C}*-algebras}, Bulletin of the
  American Mathematical Society \textbf{45} (2008), no.~2, 229--245.

\bibitem{elliott_classification_1976}
George~A Elliott, \emph{On the classification of inductive limits of sequences
  of semisimple finite-dimensional algebras}, Journal of Algebra \textbf{38}
  (1976), no.~1, 29--44.

\bibitem{elliott_classification_1995}
George~A. Elliott, \emph{The classification problem for amenable
  {C}*-algebras}, Proceedings of the {International} {Congress} of
  {Mathematicians}, {Vol}.\ 1, 2 ({Z\"{u}rich}, 1994), Birkh{\"{a}}user, Basel,
  1995, pp.~922--932.

\bibitem{farah_all_2011-1}
Ilijas Farah, \emph{All automorphisms of all {Calkin} algebras}, Mathematical
  Research Letters \textbf{18} (2011), no.~3, 489--503.

\bibitem{farah_all_2011}
\bysame, \emph{All automorphisms of the {Calkin} algebra are inner}, Annals of
  Mathematics \textbf{173} (2011), no.~2, 619--661.

\bibitem{farah_existentially_2016}
Ilijas Farah, Isaac Goldbring, Bradd Hart, and David Sherman,
  \emph{Existentially closed {II}{$_1$} factors}, Fundamenta Mathematicae
  \textbf{233} (2016), no.~2, 173--196.

\bibitem{farah_countable_2013}
Ilijas Farah and Bradd Hart, \emph{Countable saturation of corona algebras},
  Comptes Rendus Math{\'{e}}matiques de l'Acad{\'{e}}mie des Sciences
  \textbf{35} (2013), no.~2, 35--56.

\bibitem{farah_model_2017}
Ilijas Farah, Bradd Hart, Martino Lupini, Leonel Robert, Aaron Tikuisis,
  Alessandro Vignati, and Wilhelm Winter, \emph{Model {theory} of
  {C}*-algebras}, arXiv:1602.08072 (2016).

\bibitem{farah_relative_2017}
Ilijas Farah, Bradd Hart, Mikael R{\o}rdam, and Aaron Tikuisis, \emph{Relative
  commutants of strongly self-absorbing {C}*-algebras}, Selecta Mathematica
  \textbf{23} (2017), no.~1, 363--387.

\bibitem{farah_model_2013}
Ilijas Farah, Bradd Hart, and David Sherman, \emph{Model theory of operator
  algebras {I}: {stability}}, Bulletin of the London Mathematical Society
  \textbf{45} (2013), no.~4, 825--838.

\bibitem{farah_model_2014}
\bysame, \emph{Model theory of operator algebras {II}: model theory}, Israel
  Journal of Mathematics \textbf{201} (2014), no.~1, 477--505.

\bibitem{farah_model_2014-1}
\bysame, \emph{Model theory of operator algebras {III}: elementary equivalence
  and {II}{$_1$} factors}, Bulletin of the London Mathematical Society
  \textbf{46} (2014), no.~3, 609--628.

\bibitem{farah_calkin_2016}
Ilijas Farah and Ilan Hirshberg, \emph{The {C}alkin algebra is not countably
  homogeneous}, Proceedings of the American Mathematical Society \textbf{144},
  no.~12, 5351--5357.

\bibitem{farah_omitting_2014}
Ilijas Farah and Menachem Magidor, \emph{Omitting types in logic of metric
  structures}, Journal of Mathematical Logic, to appear.

\bibitem{farah_homeomorphisms_2012}
Ilijas Farah and Paul McKenney, \emph{Homeomorphisms of {Cech}-{Stone}
  remainders: the zero-dimensional case}, Proccedings of the American
  Mathematical Society, to appear.

\bibitem{farah_calkin_2013}
Ilijas Farah, Paul {McKenney}, and Ernest Schimmerling, \emph{Some calkin
  algebras have outer automorphisms}, Archive for Mathematical Logic
  \textbf{52}, no.~5, 517--524.

\bibitem{farah_dichotomy_2010}
Ilijas Farah and Saharon Shelah, \emph{A dichotomy for the number of
  ultrapowers}, Journal of Mathematical Logic \textbf{10} (2010), no.~1-2,
  45--81.

\bibitem{farah_rigidity_2016}
\bysame, \emph{Rigidity of continuous quotients}, Journal of the Institute of
  Mathematics of Jussieu \textbf{15} (2016), no.~1, 1--28.

\bibitem{gardella_rokhlin_2017}
Eusebio Gardella, Mehrdad Kalantar, and Martino Lupini, \emph{Rokhlin dimension
  for compact quantum group actions}, {arXiv}:1703.10999.

\bibitem{gardella_equivariant_2016}
Eusebio Gardella and Martino Lupini, \emph{Equivariant logic and applications
  to {C}*-dynamics}, arXiv:1608.05532 (2016).

\bibitem{glebsky_almost_2010}
Lev Glebsky, \emph{Almost commuting matrices with respect to normalized
  {Hilbert}-{Schmidt} norm}, arXiv:1002.3082 (2010).

\bibitem{glimm_certain_1960}
James~G. Glimm, \emph{On a certain class of operator algebras}, Transactions of
  the American Mathematical Society \textbf{95} (1960), no.~2, 318--340.

\bibitem{godel_consistency_1940}
Kurt G\"{o}del, \emph{The {Consistency} of the {Continuum} {Hypothesis}},
  Annals of {Mathematics} {Studies}, no. 3, Princeton University Press,
  Princeton, N. J., 1940. \MR{0002514}

\bibitem{goldbring_enforceable_2017}
Isaac Goldbring, \emph{Enforceable operator algebras}, arXiv:1706.09048 (2017).

\bibitem{goldbring_theory_2013}
Isaac Goldbring, Bradd Hart, and Thomas Sinclair, \emph{The theory of tracial
  von {Neumann} algebras does not have a model companion}, Journal of Symbolic
  Logic \textbf{78} (2013), no.~3, 1000--1004.

\bibitem{goldbring_robinson_2017}
Isaac Goldbring and Thomas Sinclair, \emph{Robinson forcing and the
  quasidiagonality problem}, International Journal of Mathematics \textbf{28},
  no.~2, 1750008, 15. \MR{3615582}

\bibitem{goldbring_kirchbergs_2015}
\bysame, \emph{On {Kirchberg}'s embedding problem}, Journal of Functional
  Analysis \textbf{269} (2015), no.~1, 155--198.

\bibitem{gromov_endomorphisms_1999}
Mikhael Gromov, \emph{Endomorphisms of symbolic algebraic varieties}, Journal
  of the European Mathematical Society \textbf{1} (1999), no.~2, 109--197.

\bibitem{hodges_building_1985}
Wilfrid Hodges, \emph{Building models by games}, London {Mathematical}
  {Society} {Student} {Texts}, vol.~2, Cambridge University Press, Cambridge,
  1985. \MR{812274}

\bibitem{jiang_simple_1999}
Xinhui Jiang and Hongbing Su, \emph{On a {simple} {unital} {projectionless}
  {C}*-{algebra}}, American Journal of Mathematics \textbf{121} (1999), no.~2,
  359--413.

\bibitem{kirchberg_embedding_2000}
Eberhard Kirchberg and N.~Christopher Phillips, \emph{Embedding of exact
  {C}*-algebras in the {Cuntz} algebra {$\mathcal{O}_2$}}, Journal f{\"{u}}r
  die reine und angewandte Mathematik \textbf{525} (2000), 17--53.

\bibitem{kirchberg_infinite_2002}
Eberhard Kirchberg and Mikael R{\o}rdam, \emph{Infinite {non}-simple
  {C}*-{algebras}: {absorbing} the {Cuntz} {algebra} {$\mathcal{O}_{\infty
  }$}}, Advances in Mathematics \textbf{167} (2002), no.~2, 195--264.

\bibitem{loring_lifting_1997}
Terry~A. Loring, \emph{Lifting solutions to perturbing problems in
  {C}*-algebras}, Fields {Institute} {Monographs}, vol.~8, American
  Mathematical Society, Providence, RI, 1997.

\bibitem{mcduff_central_1970}
Dusa McDuff, \emph{Central sequences and the hyperfinite factor}, Proceedings
  of the London Mathematical Society. Third Series \textbf{21} (1970),
  443--461. \MR{0281018}

\bibitem{ozawa_about_2004}
Narutaka Ozawa, \emph{About the {QWEP} conjecture}, International Journal of
  Mathematics \textbf{15} (2004), no.~5, 501--530. \MR{2072092}

\bibitem{pedersen_analysis_1989}
Gert~K. Pedersen, \emph{Analysis now}, Graduate {Texts} in {Mathematics}, vol.
  118, Springer-Verlag, New York, 1989.

\bibitem{pestov_hyperlinear_2008}
Vladimir~G. Pestov, \emph{Hyperlinear and sofic groups: {A} brief guide},
  Bulletin of Symbolic Logic \textbf{14} (2008), no.~04, 449--480.

\bibitem{phillips_real_2002}
N.~Christopher Phillips, \emph{Real rank and exponential length of tensor
  products with {$\mathcal{O}_{\infty }$}}, Journal of Operator Theory
  \textbf{47} (2002), no.~1, 117--130. \MR{1905816}

\bibitem{phillips_calkin_2007}
N.~Christopher Phillips and Nik Weaver, \emph{The {Calkin} algebra has outer
  automorphisms}, Duke Mathematical Journal \textbf{139} (2007), no.~1,
  185--202. \MR{MR2322680}

\bibitem{podles_symmetries_1995}
Piotr Podle{\'{s}}, \emph{Symmetries of quantum spaces. {Subgroups} and
  quotient spaces of quantum {$\mathrm{SU}_2$} and {$\mathrm{SO}_3$} groups},
  Communications in Mathematical Physics \textbf{170} (1995), no.~1, 1--20.
  \MR{MR1331688}

\bibitem{raeburn_morita_1998}
Iain Raeburn and Dana~P. Williams, \emph{Morita equivalence and
  continuous-trace {C}*-algebras}, Mathematical {Surveys} and {Monographs},
  vol.~60, American Mathematical Society, Providence, RI, 1998.

\bibitem{rordam_classification_2002}
Mikael R{\o}rdam, \emph{Classification of nuclear {C}*-algebras}, Encyclopaedia
  of {Mathematical} {Sciences}, vol. 126, Springer-Verlag, Berlin, 2002.

\bibitem{rosenberg_algebraic_1994}
Jonathan Rosenberg, \emph{Algebraic {$K$}-theory and its applications},
  Graduate {Texts} in {Mathematics}, vol. 147, Springer-Verlag, New York, 1994.
  \MR{1282290}

\bibitem{schochet_algebraic_1994}
Claude Schochet, \emph{Algebraic topology and {C}*-algebras}, {C}*-algebras:
  1943--1993, Contemp. {Math}., vol. 167, Amer. Math. Soc., Providence, RI,
  1994, pp.~218--231.

\bibitem{shoenfield_problem_1961}
Joseph~R. Shoenfield, \emph{The problem of predicativity}, Essays on the
  foundations of mathematics, Magnes Press, Hebrew Univ., Jerusalem, 1961,
  pp.~132--139.

\bibitem{szabo_strongly_2015}
G{\'{a}}bor Szab{\'{o}}, \emph{Strongly self-absorbing {C}*-dynamical systems},
  Transactions of the American Mathematical Society, to appear.

\bibitem{szabo_strongly_2016}
\bysame, \emph{Strongly self-absorbing {C}*-dynamical systems, {II}}, Journal
  of Noncommutative Geometry, to appear.

\bibitem{szabo_strongly_2017}
\bysame, \emph{Strongly self-absorbing {C}*-dynamical systems, {III}}, Advances
  in Mathematics \textbf{316}, 356--380.

\bibitem{thom_sofic_2008}
Andreas Thom, \emph{Sofic groups and diophantine approximation}, Communications
  on Pure and Applied Mathematics \textbf{61} (2008), no.~8, 1155--1171.

\bibitem{thom_about_2012}
\bysame, \emph{About the metric approximation of {Higman}'s group}, Journal of
  Group Theory \textbf{15} (2012), no.~2, 301--310.

\bibitem{todorcevic_introduction_2010}
Stevo Todorcevic, \emph{Introduction to {Ramsey} spaces}, Annals of
  {Mathematics} {Studies}, vol. 174, Princeton University Press, Princeton, NJ,
  2010.

\bibitem{toms_infinite_2008}
Andrew Toms, \emph{An infinite family of non-isomorphic {C}*-algebras with
  identical {K}-theory}, Transactions of the American Mathematical Society
  \textbf{360} (2008), no.~10, 5343--5354.

\bibitem{toms_comparison_2009}
Andrew~S. Toms, \emph{Comparison {theory} and {smooth} {minimal}
  {C}*-{dynamics}}, Communications in Mathematical Physics \textbf{289} (2009),
  no.~2, 401--433.

\bibitem{toms_strongly_2007}
Andrew~S. Toms and Wilhelm Winter, \emph{Strongly self-absorbing
  {C}*-algebras}, Transactions of the American Mathematical Society
  \textbf{359} (2007), no.~8, 3999--4029.

\bibitem{vignati_nontrivial_2017}
Alessandro Vignati, \emph{Nontrivial homeomorphisms of {C}ech-{Stone}
  remainders}, M{\"{u}}nster Journal of Mathematics \textbf{10} (2017), no.~1,
  189--200. \MR{3624107}

\bibitem{voiculescu_asymptotically_1983}
Dan Voiculescu, \emph{Asymptotically commuting finite rank unitary operators
  without commuting approximants}, Acta Universitatis Szegediensis \textbf{45}
  (1983), no.~1-4, 429--431. \MR{708811}

\end{thebibliography}

\end{document}